\documentclass[a4paper, 12pt]{amsart}
\usepackage[margin=2.75cm]{geometry}
\usepackage{tikz}
\usepackage{amssymb,amsrefs,mathtools}

\newcommand{\Aut}{\operatorname{Aut}}
\newcommand{\lsp}{\operatorname{span}}

\newcommand{\id}{\operatorname{id}}
\newcommand{\PI}{\operatorname{PIso}}

\newcommand{\Dom}{\operatorname{Dom}}

\newcommand{\CC}{\mathbb{C}}

\newcommand{\NN}{\mathbb{N}}
\newcommand{\ZZ}{\mathbb{Z}}
\newcommand{\TT}{\mathbb{T}}

\newcommand{\KK}{\mathbb{K}}

\newcommand{\Bb}{\mathcal{B}}
\newcommand{\Cc}{\mathcal{C}}

\newcommand{\Ff}{\mathcal{F}}
\newcommand{\Gg}{\mathcal{G}}

\newcommand{\Jj}{\mathcal{J}}

\newcommand{\Mm}{\mathcal{M}}
\newcommand{\Nn}{\mathcal{N}}
\newcommand{\Oo}{\mathcal{O}}

\newcommand{\Ss}{\mathcal{S}}

\newcommand{\Uu}{\mathcal{U}}

\newcommand{\lZ}[1]{Z(#1]}
\newcommand{\rZ}[1]{Z[#1)}
\newcommand{\bZ}[1]{Z(#1)}

\newcommand{\Jsig}{\tilde\sigma}
\newcommand{\Stau}{\tilde\tau}
\newcommand{\hG}{\widehat{\mathcal{G}}}
\newcommand{\hO}{\widehat{\mathcal{O}}}

\newtheorem{thm}{Theorem}[section]
\newtheorem{cor}[thm]{Corollary}

\newtheorem{lem}[thm]{Lemma}
\newtheorem{prp}[thm]{Proposition}

\theoremstyle{definition}
\newtheorem{dfn}[thm]{Definition}
\newtheorem{ntn}[thm]{Notation}
\theoremstyle{remark}
\newtheorem{rmk}[thm]{Remark}
\newtheorem{exm}[thm]{Example}
\numberwithin{equation}{section}

\makeatletter
\@namedef{subjclassname@2020}{%
  \textup{2020} Mathematics Subject Classification}
\makeatother

\tikzstyle{vertex}=[circle]
\tikzstyle{goto}=[->,shorten >=1pt,>=stealth,semithick]

\newcommand{\aeq}{\mathrm{ae}} % asymptotic equivalence

\begin{document}

%\date{\today}
\title[$KK$-duality for self-similar groupoid actions]{$KK$-duality for self-similar groupoid actions on graphs}

\author[Brownlowe]{Nathan Brownlowe}
\author[Buss]{Alcides Buss}
\author[Gon\c{c}alves]{Daniel Gon\c{c}alves}
\author[Hume]{Jeremy B. Hume}
\author[Sims]{Aidan Sims}
\author[Whittaker]{Michael F. Whittaker}

\address{Nathan Brownlowe \\ School of Mathematics and
Statistics  \\
University of Sydney} \email{nathan.brownlowe@sydney.edu.au}
\address{Alcides Buss and Daniel Gon\c{c}alves, Departamento de Matem\'atica\\
 Universidade Federal de Santa Catarina}
\email{alcides.buss@ufsc.br, daemig@gmail.com}
\address{Jeremy B. Hume and Michael F.
Whittaker \\ School of Mathematics and Statistics  \\
University of Glasgow}
\email{jeremybhume@gmail.com, Mike.Whittaker@glasgow.ac.uk}
\address{Aidan Sims \\ School of Mathematics and
Applied Statistics  \\
University of Wollongong} \email{asims@uow.edu.au}

\thanks{Sims was supported by Australian Research Council grant DP220101631
and by CAPES grant 88887.370640. The second and third authors were partially
supported by CNPq and CAPES - Brazil. Hume was supported by the European Research Council (ERC) under the European Union's Horizon 2020 research and innovation programme (grant agreement No. 817597). We thank Isnie Yusnitha for careful
reading and helpful comments and Volodia Nekrashevych for insightful conversations. We also thank Valentin Deaconu and an anonymous referee for pointing out issues with our original draft.}

\begin{abstract}
We extend Nekrashevych's $KK$-duality for $C^*$-algebras of regular, recurrent, contracting self-similar group actions to regular, contracting self-similar groupoid actions on a graph, removing the recurrence condition entirely and generalising from a finite alphabet to a finite graph.

More precisely, given a regular and contracting self-similar groupoid $(G,E)$ acting faithfully on a finite directed graph $E$, we associate two $C^*$-algebras, $\Oo(G,E)$ and $\hO(G,E)$, to it and prove that they are strongly Morita equivalent to the stable and unstable Ruelle C*-algebras of a Smale space arising from a Wieler solenoid of the self-similar limit space.  That these algebras are Spanier-Whitehead dual in $KK$-theory follows from the general result for Ruelle algebras of irreducible Smale spaces proved by Kaminker, Putnam, and the last author.
\end{abstract}



\maketitle

\section{Introduction}\label{sec:intro}

In the last 40 years self-similar groups have been fundamental in answering a
wide range of outstanding conjectures; for example, the Grigorchuk group
\cite {Gr} was the first group shown to have intermediate growth, and also
the first known example of a group that is amenable but not elementary
amenable. Since self-similar groups are defined by their actions on trees,
and hence induce actions on their boundaries, they lend themselves to study
via noncommutative analysis. Segal \cite{Segal:AJM51}, building on the work
of Maharam \cite{Maharam:PNAS42} showed that commutative von Neumann algebras
are precisely those that arise as the $L^\infty$-algebras of localisable
measure spaces, and the Gelfand--Naimark Theorem
\cite{GelfandNaimark:RMNS43} characterises commutative $C^*$-algebras as
$C_0$-algebras of locally compact Hausdorff spaces. Building on this
foundation, much of modern operator-algebra theory, including Connes'
noncommutative geometry, investigates noncommutative $C^*$-algebras by
analogy with a kind of noncommutative measure space or topological space. An
excellent example is Connes' notion of noncommutative Poincar\'e duality
\cite{Con}, which we refer to as \emph{$KK$-duality} between a pair of noncommutative $C^*$-algebras (actually a version of Spanier--Whitehead duality, see \cite{KS}). A very
general example of this is the $KK$-duality between the stable and unstable
Ruelle algebras of irreducible Smale spaces
\cite{Kaminker-Putnam-Whittaker:K-duality}. Recently, pioneering work of
Nekrashevych \cite{Nekrashevych:Cstar_selfsimilar} proved that contracting,
recurrent, and regular self-similar groups each give rise to a pair of
$C^*$-algebras that can be realised as the Ruelle algebras of an underlying
Smale space and used the duality result of
\cite{Kaminker-Putnam-Whittaker:K-duality} to see that these algebras are
$KK$-dual.

In this paper we extend this duality result to the self-similar groupoids defined in \cite{Laca-Raeburn-Ramagge-Whittaker:Equilibrium}, and simultaneously extend Nekrashevych's Smale space result
\cite{Nekrashevych:Cstar_selfsimilar} to self-similar actions that are not necessarily recurrent. We also note that our proof is completely different; in particular, we show that the underlying Smale space is a
Wieler solenoid \cite{Wieler:Smale}. Our main theorem, Theorem~\ref{thm:main duality thm}, states that the $C^*$-algebra $\Oo(G, E)$ of a contracting self-similar groupoid action as defined in
\cite{Laca-Raeburn-Ramagge-Whittaker:Equilibrium}, and the $C^*$-algebra, suggestively denoted $\hO(G, E)$, of the Deaconu--Renault groupoid of the canonical local homeomorphism on an associated limit space are
$KK$-dual. In conjunction with consequences of classification theory for $C^*$-algebras (see \cite[Theorem 1.1 and Section 4.4]{Kaminker-Putnam-Whittaker:K-duality} and \cite[Theorem
5.11]{Yamashita-Proietti}), our result implies that $\Oo(G, E) \cong \hO(G, E)$, so our $KK$-duality mimics Poincar\'e duality in topology.

To prove our main theorem, we generalise the results of Nekrashevych
\cite{Nekrashevych:Cstar_selfsimilar} to the self-similar groupoid setting of
\cite{Laca-Raeburn-Ramagge-Whittaker:Equilibrium}. We extend Nekrashevych's
construction of the limit space $\Jj$ of a self-similar group to the setting
of self-similar groupoids, and show that the shift map on $\Jj$ is open and
expansive. We employ Wieler's classification of Smale spaces with totally
disconnected stable sets to see that the projective limit of $\Jj$ with
respect to the shift map, which we identify with a natural limit solenoid
$\Ss$, is a Smale space with respect to a homeomorphism $\tilde\tau$ induced
by the shift map on $\Jj$. Nekrashevych's construction of Smale spaces from
self-similar groups is subtle, so we are careful to include all the details
in extending to the situation of self-similar groupoids. In doing so, we are
able to weaken the existing hypotheses, even for self-similar groups. The
remainder of our work goes into proving that the Cuntz--Pimsner algebra
$\Oo(G, E)$ is Morita equivalent (i.e. stably isomorphic) to the unstable Ruelle algebra of $(\Ss,
\tilde\tau)$ and the Deaconu--Renault groupoid $C^*$-algebra $\hO(G,
E)$ is Morita equivalent to the stable Ruelle algebra of $(\Ss, \tilde\tau)$.

We illustrate our results via several examples. In particular, the main example from \cite{Laca-Raeburn-Ramagge-Whittaker:Equilibrium} is defined as the self-similar groupoid $(G,E)$ arising from the following
graph and automaton:
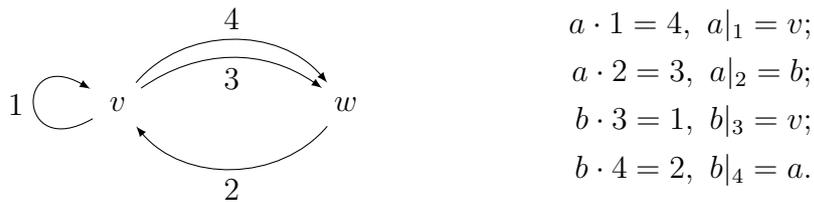
\begin{center}
\mbox{}\hfill\begin{minipage}[t]{0.4\textwidth}
\[
\begin{tikzpicture}
\begin{scope}[xshift=0,yshift=0]
\node at (0,0) {$w$};
\node[vertex] (vertexe) at (0,0)   {$\quad$};
\node at (-3,0) {$v$};
\node[vertex] (vertex-a) at (-3,0)   {$\quad$}
	edge [->,>=latex,out=35,in=145] node[below,swap,pos=0.5]{$3$} (vertexe)
	edge [->,>=latex,out=50,in=130] node[above,swap,pos=0.5]{$4$} (vertexe)
	edge [->,>=latex,out=210,in=150,loop] node[left,pos=0.5]{$1$} (vertexe)
	edge [<-,>=latex,out=310,in=230] node[below,swap,pos=0.5]{$2$} (vertexe);
\end{scope}
\end{tikzpicture}
\]
\end{minipage}\hskip2em
\begin{minipage}[t]{0.35\textwidth}
\vspace{0cm}
\begin{align*}
a\cdot 1=4,\ a|_1=v; \\
a\cdot 2=3,\ a|_2=b;\\
b\cdot 3=1, \ b|_3=v;\\
b\cdot 4=2,\ b|_4=a.
\end{align*}
%\vspace{0.25cm}
\end{minipage}\hfill\mbox{}
\end{center}
The limit space of this self-similar action is homeomorphic to the complex unit circle with the map $z \mapsto z^2$. The associated limit solenoid is the classical dyadic solenoid Smale space. Thus, using
\cite{Yi:k-theory} and Theorem~\ref{thm:main duality thm} we deduce that
\[
K_0(\Oo_{(G,E)}) \cong \mathbb{Z},\quad K_1(\Oo_{(G,E)}) \cong \mathbb{Z},\quad K^0(\Oo_{(G,E)}) \cong \mathbb{Z},\quad\text{ and } K^1(\Oo_{(G,E)}) \cong \mathbb{Z}.
\]
The Kirchberg-Phillips Theorem then implies that $\Oo_{(G,E)}$ is isomorphic to the Cuntz---Pimsner algebra of the odometer.

Another source of interesting examples are the Katsura algebras \cite{Katsura:class_II}. These were recognised as self-similar actions on graphs by Exel
and Pardo \cite{Exel-Pardo:Self-similar}. To see how they fit into our framework, see \cite[Example 7.7]{Laca-Raeburn-Ramagge-Whittaker:Equilibrium} and
\cite[Appendix A]{Laca-Raeburn-Ramagge-Whittaker:Equilibrium} for the general translation from the Exel-Pardo situation to self-similar groupoid actions. In \cite[Section 18]{Exel-Pardo:Self-similar}, Exel and Pardo show that all unital Kirchberg algebras in the UCT class have representations as self-similar groupoids. They also prove that
the $K$-theory of the Cuntz--Pimsner algebra of a self-similar groupoid of this sort is directly computable from the graph adjacency matrix and restriction matrix for the self-similar action.
An analysis using Schreier graphs shows that the limit space of such a self-similar groupoid is the total space of a bundle over the circle of copies of the  Cantor set, each with an odometer action specified by the restriction matrix. Once again, a combination of $KK$-duality with classification theory shows that the Cuntz--Pimsner algebra of such a system is isomorphic to the Deaconu--Renault groupoid of the limit space. This allows us to
compute the $K$-theory of these interesting Deaconu--Renault $C^*$-algebras, see Example \ref{Ex:Katsura}

Example~\ref{basilica type example} introduces a new example whose limit dynamical system is conjugate to that of the basilica group. By definition, the basilica
group is the iterated monodromy group (see \cite[Chapter~5]{Nekrashevych:Self-similar}) of the function $f(z)=z^2-1$ as a complex map from $\mathbb{C}\setminus\nobreak\{-1,0,1\} \to \mathbb{C}\setminus \{0,1\}$. The
$K$-groups of its
Cuntz--Pimsner algebra, and those of its dual algebra, are computed in \cite[Theorem~4.8]{Nekrashevych:Cstar_selfsimilar} and \cite[Theorem~6.6]{Nekrashevych:Cstar_selfsimilar} (see also \cite{Hume}).
Our main theorem therefore allows us also to compute the $K$-homology of both algebras.

The paper is organised as follows. In Section~\ref{sec:background} we give the necessary background for the paper. We begin with directed graphs and their $C^*$-algebras.
This leads to a section on self-similar groupoid actions on graphs and we recall relevant information from \cite{Laca-Raeburn-Ramagge-Whittaker:Equilibrium}. We conclude the section
with Smale spaces and their $C^*$-algebras.

In Section~\ref{sec:limit space},  we generalise Nekrashevych's notion of a limit space to self-similar groupoid actions. Nekrashevych's construction is clever and subtle, so we provide substantial details regarding the metric topology on the limit space that are omitted in Nekrashevych's work. We complete this section by defining the level Schreier graphs of a self-similar groupoid action and how these relate
to the limit space. In the following two sections we seek to understand the dynamics on the limit space. In particular, we show that the shift map on the limit space is a Wieler solenoid, and hence
the natural extension is a Smale space with totally disconnected stable sets.

Sections \ref{sec:O(G,E)}~and~\ref{sec:hO(G,E)} define two natural
$C^*$-algebras associated to a self-similar groupoid action. One is the
Cuntz--Pimsner algebra of \cite{Laca-Raeburn-Ramagge-Whittaker:Equilibrium};
our main goal is to provide a groupoid model for this algebra, extending that
given by Nekrashevych in \cite[Section~5]{Nekrashevych:Cstar_selfsimilar}.
The other is the $C^*$-algebra of a generalisation to self-similar groupoids
of Nekrashevych's Deaconu--Renault groupoid of the limit space of a
self-similar group. This becomes the dual algebra for the Cuntz--Pimsner
algebra. Corollary~\ref{cor:Hausdorffness} establishes the exact condition
required for the groupoid of germs to be Hausdorff.

Our main result appears in Section~\ref{sec:poincare}. We prove that the
Cuntz--Pimsner algebra is strongly Morita equivalent to the unstable Ruelle
algebra of a Smale space and that the Deaconu--Renault groupoid algebra is
strongly Morita equivalent to its stable Ruelle algebra. We then deduce our
main $KK$-duality result from~\cite{Kaminker-Putnam-Whittaker:K-duality}.

\section{Background}\label{sec:background}

\subsection{Graphs and $C^*$-algebras}\label{subsec:graphs}

In this paper, we use the notation and conventions of
\cite{Raeburn:Graph_algebras} for graphs and their $C^*$-algebras. A
\emph{directed graph} $E$ is a quadruple $E = (E^0, E^1, r, s)$ consisting of
sets $E^0$ and $E^1$ and maps $r,s : E^1 \to E^0$. The elements of $E^0$ are
called \emph{vertices} and we think of them as dots, and elements of $E^1$
are called edges, and we think of them as arrows pointing from one vertex to
another: $e \in E^1$ points from $s(e)$ to $r(e)$.

A \emph{path} in $E$ is either a vertex, or a string $\mu = e_1 \dots e_n$
such that each $e_i \in E^1$ and $s(e_i) = r(e_{i+1})$. For $n \ge 2$ we
define $E^n = \{e_1 \dots e_n \mid e_i \in E^1, s(e_i) = r(e_{i+1})\}$ for the set of paths of length $n$ in $E$. The \emph{length}
$|\mu|$ of the path $\mu$ is given by $|\mu| = n$ if and only if $\mu \in
E^n$. The collection $E^* := \bigcup^\infty_{n=0} E^n$ of all paths in $E$ is
a small category with identity morphisms $E^0$, composition given by
concatenation of paths of nonzero length together with the identity rules
$r(\mu) \mu = \mu = \mu s(\mu)$, and domain and codomain maps $r,s$.

For $\mu \in E^*$ and $X \subseteq E^*$, we write $\mu X = \{\mu \nu \mid \nu
\in X, s(\mu) = r(\nu\}$ and $X \mu = \{\nu\mu \mid \nu \in X, r(\mu) =
s(\nu)\}$. We write $\mu X \nu$ for $\mu X \cap X \nu$.

We say that a graph $E$ is \emph{row finite} if $vE^1$ is finite for each $v
\in E^0$ and that it has \emph{no sources} if each $vE^1$ is nonempty. We say
that it is \emph{finite} if both $E^0$ and $E^1$ are finite. We say that $E$
is \emph{strongly connected} if for all $v,w \in E^0$, the set $v E^* w$ is
nonempty, and $E$ is not the graph with one vertex and no edges. If $E$ is
strongly connected, then $vE^1$ and $E^1v$ are nonempty for all $v $ in
$E^0$.

In this paper, we will need to work with left-infinite, right-infinite and bi-infinite paths in a
directed graph $E$. We will use the following notation:
\begin{align*}
    E^\infty &= \{e_1 e_2 e_3 \dots \mid e_i \in E^1, s(e_i) = r(e_{i+1})\text{ for all }i\},\\
    E^{-\infty} &= \{\dots e_{-3} e_{-2} e_{-1} \mid e_i \in E^1, s(e_i) = r(e_{i+1})\text{ for all }i\},\quad\text{ and}\\
    E^\ZZ &= \{\dots e_{-2} e_{-1} e_0 e_1 e_2 \dots \mid e_i \in E^1, s(e_i) = r(e_{i+1})\text{ for all }i\}.
\end{align*}
For $x = x_1 x_2\dots \in E^\infty$ we write $r(x) = r(x_1)$ and for $x =
\dots x_{-2} x_{-1} \in E^{-\infty}$, we write $s(x) = s(x_{-1})$.

We endow these spaces with the topologies determined by cylinder sets. These
cylinder sets are indexed by finite paths in each of the three spaces
involved, so we will distinguish them with the following slightly
non-standard notation: for $\mu \in E^n$, we define
\begin{align*}
\rZ{\mu} &:= \{x \in E^\infty \mid x_1\dots x_n = \mu\},\quad\text{ and}\\
\lZ{\mu} &:= \{x \in E^{-\infty} \mid x_{-n} \dots x_{-1} = \mu\}.
\end{align*}
For $n \ge 0$ and $\mu  \in E^{2n+1}$, we write
\[
\bZ{\mu} := \{x \in E^\ZZ \mid x_{-n} \dots x_n = \mu\}.
\]

In this paper, all graphs will be finite. The spaces $E^\infty$, $E^{-\infty}$ and $E^\ZZ$ are then
totally disconnected compact Hausdorff spaces and the collections of cylinder sets are bases for the topologies that are closed under
intersections.

There are standard metrics realising these topologies. The metric on $E^{-\infty}$ is given by
\begin{equation}\label{eq:cylinder metric}
d(x,y) = \begin{cases}
    \inf\{c^{-n} \mid x,y \in \lZ{\mu}\text{ for some } \mu \in E^n\} &\text{ if $s(x) = s(y)$}\\
    2 &\text{ if $s(x) \not= s(y)$,}
\end{cases}
\end{equation}
and the other two are defined analogously: for $E^\infty$, we
replace $\lZ{\mu}$ with $\rZ{\mu}$ and the source map with the range map; for
$E^\ZZ$ we replace ``$\lZ{\mu}$ for some $\mu \in E^n$" with ``$\bZ{\mu}$ for
some $\mu \in E^{2n+1}$,'' and the conditions ``$s(x) = s(y)$'' and ``$s(x)
\not= s(y)$'' with ``$x_0 = y_0$'' and ``$x_0 \not= y_0$.''

Given a finite directed graph with no sources, a \emph{Cuntz--Krieger
$E$-family} in a $C^*$-algebra $A$ is a pair $(p, s)$ of functions $p : v
\mapsto p_v$ from $E^0$ to $A$ and $s : e \mapsto s_e$ from $E^1$ to $A$ such
that the $p_v$ are mutually orthogonal projections, each $s_e^*s_e =
p_{s(e)}$, and $p_v = \sum_{e \in vE^1} s_e s^*_e$ for all $v  \in E^0$. The
\emph{graph $C^*$-algebra}, denoted $C^*(E)$, is the universal $C^*$-algebra
generated by a Cuntz--Krieger $E$-family, see \cite{Raeburn:Graph_algebras}.

\subsection{Self similar actions of groupoids on graphs}\label{subsec:ssa}

Recall that a \emph{groupoid} is a small category $\Gg$ with inverses. The identity morphisms are
called \emph{units} and the collection of all identity morphisms is called the \emph{unit space}
and denoted $\Gg^{(0)}$. The set of composable
pairs of elements in $\Gg$ is denoted $\Gg^{(2)}$.

Self-similar actions of groupoids on graphs were introduced in
\cite{Laca-Raeburn-Ramagge-Whittaker:Equilibrium}, inspired by Exel and Pardo's work in
\cite{Exel-Pardo:Self-similar}. The precise relationship between the two constructions is detailed
in the appendix of \cite{Laca-Raeburn-Ramagge-Whittaker:Equilibrium}.

Given a directed graph $E$ with no sources, and given $v,w  \in E^0$, a
\emph{partial isomorphism} of $E^*$ is a bijection $g : vE^* \to wE^*$ that
preserves length and preserves concatenation in the sense that $g(\mu e) \in
g(\mu)E^1$ for all $\mu  \in E^*$ and $e  \in E^1$. The expected formula
$g(\mu e) = g(\mu)g(e)$ does not even make sense since $g$ is not typically
defined on $s(\mu)E^*$. For each $v  \in E^*$, the identity map $\id_v : vE^*
\to vE^*$ is a partial isomorphism.

The set $\PI(E^*)$ of all partial isomorphisms of $E^*$ is a groupoid with
units $\id_v$ indexed by the vertices of $E$ and multiplication given by
composition of maps. We will identify the unit space of $\PI(E^*)$ with $E^0$
in the canonical way; this is consistent with our notation for graphs since
the map $\mu \mapsto v\mu$ coincides with $\id_v : vE^* \to vE^*$.

We will write $c,d : \PI(E^*) \to E^0$ for the codomain and domain maps on the groupoid
$\PI(E^*)$, because the symbols $s$ and $r$ are already fairly overloaded. So if $g : vE^* \to
wE^*$ is a partial isomorphism, then $c(g) = w$ and $d(g) = v$.

A \emph{faithful action} of a groupoid $G$ with unit space $E^0$ on the graph $E$ is an injective
groupoid homomorphism $\phi : G \to \PI(E^*)$ that restricts to the identity map on $E^0$. We will generally write $g \cdot \mu$ in place of $\phi(g)(\mu)$.

If $E = (E^0, E^1, r, s)$ is a directed graph, and $G$ is a groupoid with
unit space $E^0$, that acts faithfully on $E^*$, then we say that $(G, E)$ is
a \emph{self similar groupoid action} if for every $g  \in G$ and every $e
\in d(g)E^1$ there exists $h  \in G$ such that $c(h) = s(g\cdot e)$ and
\begin{equation}\label{eq:self-similar}
    g \cdot(e\mu) = (g\cdot e)(h \cdot \mu)\quad\text{ for all $\mu \in s(g\cdot e)E^*$.}
\end{equation}

Since the groupoid action $G \curvearrowright E^*$ is faithful, for each $g
\in G$ and $e \in d(g)E^1$ there is a \emph{unique} $h$
satisfying~\eqref{eq:self-similar}. We denote this element by $g|_e$, and
call it the \emph{restriction} of $g$ to $e$. Restriction extends to finite
paths by iteration: for $g  \in G$, we define $g|_{d(g)} = g$, and for $e \in
E^1$ and $\mu \in s(e)E^*$, we recursively define
\[
g|_{e\mu} = (g|_e)|_{\mu}.
\]
So $g|_{e_1 \dots e_n} = (\dots(g|_{e_1})|_{e_2} \dots)|_{e_n}$, and then \eqref{eq:self-similar} extends to
\[
g\cdot(\mu\nu) = (g\cdot \mu)(g|_\mu \cdot \nu)
\]
whenever $g \in G$, $\mu \in d(g)E^*$ and $\nu \in E^* s(\mu)$.

We will use the following fundamental formulas without comment throughout the paper.

\begin{lem}[{\cite[Lemma~3.4 and Proposition~3.6]{Laca-Raeburn-Ramagge-Whittaker:Equilibrium}}]\label{properties SSA}
Let $(G, E)$ be a self-similar groupoid action on a finite directed graph $E$. For $(g,h) \in
G^{(2)}$, $\mu \in d(g)E^*$, $\nu \in E^* s(\mu)$ and $\eta \in c(g)E^*$, we
have
\begin{enumerate}
    \item $r(g \cdot \mu) = c(g)$ and $s(g \cdot \mu) = g|_\mu \cdot
        s(\mu)$;
    \item $g|_{\mu\nu} = (g|_\mu)|_\nu$;
    \item $\id_{r(\mu)}|_\mu = \id_{s(\mu)}$;
    \item $(hg)|_\mu = (h|_{g \cdot \mu})(g|_\mu)$; and
    \item $g^{-1}|_{\eta} = (g|_{g^{-1} \cdot \eta})^{-1}$.
\end{enumerate}
\end{lem}

We now give an example of a self-similar groupoid action by defining an $E$-automaton as described in \cite[Definition~3.7, Proposition~3.9 and
Theorem~3.9]{Laca-Raeburn-Ramagge-Whittaker:Equilibrium}. The key point of an $E$-automaton is that an
action on the edges of the graph and a restriction map satisfying specified
range and source conditions ensures that the action extends to a self-similar groupoid action on finite paths of the graph.

\begin{exm}
\label{s,r clarifying example}

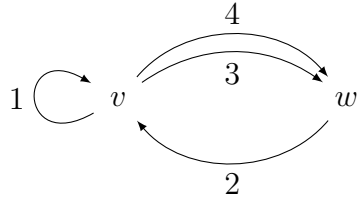
\begin{figure}
\begin{tikzpicture}[scale=1.0]
\node at (0,0) {$w$};
\node[vertex] (vertexe) at (0,0)   {$\quad$};
\node at (-3,0) {$v$};
\node[vertex] (vertex-a) at (-3,0)   {$\quad$}
	edge [->,>=latex,out=35,in=145] node[below,swap,pos=0.5]{$3$} (vertexe)
	edge [->,>=latex,out=50,in=130] node[above,swap,pos=0.5]{$4$} (vertexe)
	edge [->,>=latex,out=210,in=150,loop] node[left,pos=0.5]{$1$} (vertexe)
	edge [<-,>=latex,out=310,in=230] node[below,swap,pos=0.5]{$2$} (vertexe);
\end{tikzpicture}
\caption{Graph $E$ for Example~\ref{s,r clarifying example}}
\label{asymmetric E}
\end{figure}

The following example is carried through \cite{Laca-Raeburn-Ramagge-Whittaker:Equilibrium}. Consider the graph $E$ in Figure~\ref{asymmetric E},
and define
\begin{align}\label{defasymm}
a\cdot 1=4,\ a|_1=v;\quad\quad&b\cdot 3=1, \ b|_3=v;\\
a\cdot 2=3,\ a|_2=b;\quad\quad&b\cdot 4=2,\ b|_4=a.\notag
\end{align}
See \cite[Example 3.10]{Laca-Raeburn-Ramagge-Whittaker:Equilibrium} for a detailed exposition. To see explicitly how the groupoid action on $E^*$ manifests we compute
\[
a \cdot 242312=3(b \cdot 42312)=32(a \cdot 2312)=323(b\cdot 312)=3231(v \cdot 12)=323112.
\]
\end{exm}

\subsection{Smale spaces and $C^*$-algebras}\label{subsec:smale}

A Smale space $(X,\varphi)$ consists of a compact metric space $X$ and a homeomorphism $\varphi: X
\to X$ along with constants $\varepsilon_{X} > 0$ and $0<\lambda < 1$ and a locally defined continuous map
\[ [\cdot,\cdot]:\{(x,y) \in X \times X \mid d(x,y) \leq \varepsilon_{X}\}\to X,\quad (x,y) \mapsto [x, y]\]
satisfying
\begin{itemize}
\item[(B1)] $\left[ x, x \right] = x$,
\item[(B2)] $\left[ x, [ y, z] \right] = [ x, z]$ if both sides are defined,
\item[(B3)] $\left[ [ x, y], z \right] = [ x,z ]$ if both sides are defined,
\item[(B4)] $\varphi[x, y] = [ \varphi(x), \varphi(y)]$ if both sides are defined,
\item[(C1)] For $x,y  \in X$ such that $[x,y]=y$, we have
    $d(\varphi(x),\varphi(y)) \leq \lambda d(x,y)$, and
\item[(C2)] For $x,y  \in X$ such that $[x,y]=x$, we have
    $d(\varphi^{-1}(x),\varphi^{-1}(y)) \leq \lambda d(x,y)$.
\end{itemize}

The bracket map defines a local product structure on a Smale space as
follows, for $x  \in X$ and $ 0 < \varepsilon \leq \varepsilon_{X}$, define
\begin{align*}
 X^{s}(x, \varepsilon) & := \{ y \in X \mid d(x,y) < \varepsilon, [y,x] =x \} \,\text{ and } \\
X^{u}(x, \varepsilon) & := \{ y \in X \mid d(x,y) < \varepsilon, [x,y] =x \}.
\end{align*}
We call $X^s(x,\varepsilon)$ a \emph{local stable set} of $x$ and $X^u(x,\varepsilon)$ a
\emph{local unstable set} of $x$. Figure~\ref{Bracket intersection} gives a pictorial
representation of the local stable sets and their interactions (provided
$d(x,y)<\varepsilon_X/2$).

\begin{figure}[ht]
\begin{center}
\begin{tikzpicture}
\tikzstyle{axes}=[]
\begin{scope}[style=axes]
	\draw[<->] (-2,-1) node[left] {$X^s(x,\varepsilon_X)$} -- (1,-1);
	\draw[<->] (-1,-2) -- (-1,1) node[above] {$X^u(x,\varepsilon_X)$};
	\node at (-1.2,-1.4) {$x$};
	\node at (1.1,-1.4) {$[x,y]$};
	\pgfpathcircle{\pgfpoint{-1cm}{-1cm}} {2pt}
	\pgfpathcircle{\pgfpoint{0.5cm}{-1cm}} {2pt}
	\pgfusepath{fill}
\end{scope}
\begin{scope}[style=axes]
	\draw[<->] (-1.5,0.5) -- (1.5,0.5) node[right] {$X^s(y,\varepsilon_X)$};
	\draw[<->] (0.5,-1.5) -- (0.5,1.5) node[above] {$X^u(y,\varepsilon_X)$};
	\node at (0.7,0.2) {$y$};
	\node at (-1.6,0.2) {$[y,x]$};
	\pgfpathcircle{\pgfpoint{0.5cm}{0.5cm}} {2pt}
	\pgfpathcircle{\pgfpoint{-1cm}{0.5cm}} {2pt}
	\pgfusepath{fill}
\end{scope}
\end{tikzpicture}
\caption{The local stable and unstable sets of $x,y  \in X$ and their bracket maps}
\label{Bracket intersection}
\end{center}
\end{figure}
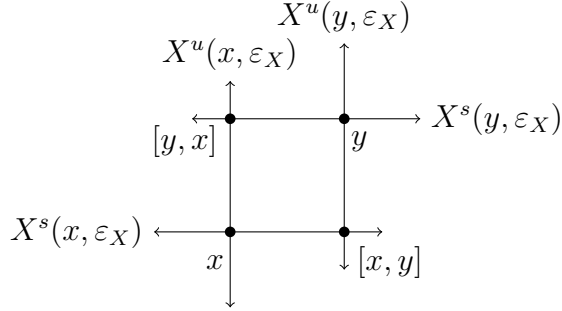

Suppose $(X,\varphi)$ is a Smale space. Then for $x,y  \in X$ the global
stable and unstable equivalence relations are given by
\begin{align*}
x \sim_s y & \text{ whenever } d(\varphi^n(x),\varphi^n(y)) \to 0 \text{ as } n \to \infty \text{ and} \\
x \sim_u y & \text{ whenever } d(\varphi^{-n}(x),\varphi^{-n}(y)) \to 0  \text{ as } n \to \infty.
\end{align*}
The stable equivalence class of $x  \in X$ is denoted $X^s(x)$ and we have
$X^s(x,\varepsilon) \subset X^s(x)$. Similarly, the unstable equivalence
class of $x  \in X$ is denoted by $X^u(x)$ and $X^u(x,\varepsilon) \subset
X^u(x)$. We consider each of the stable equivalence classes as locally
compact and Hausdorff topological spaces whose topology is generated by
$\{X^s(y,\varepsilon) \mid y \in X^s(x), 0 <\varepsilon< \varepsilon_X\}$. A
similar topology is defined in the unstable case.

A Smale space $(X,\varphi)$ is said to be \emph{irreducible} if, for all non-empty open sets $U,V
\subseteq X$, there exists $N$ such that $\varphi^N(U) \cap V \neq \varnothing$.

It is said to be \emph{mixing} if, for all non-empty open sets $U,V
\subseteq X$, there exists $N$ such that $\varphi^n(U) \cap V \neq \varnothing$, for all $n\geq N$.

We now consider various $C^*$-algebras associated with Smale spaces. Ruelle first defined
$C^*$-algebras associated to Smale spaces in \cite{Ruelle:Noncommutative}, and these
$C^*$-algebras are usually referred to as the stable and unstable algebras of the Smale space. In
\cite{Putnam:Algebras}, Putnam then defined the \emph{Ruelle algebras} as crossed product
$C^*$-algebras of the stable and unstable algebras. Putnam showed that the Ruelle algebras
generalise Cuntz--Krieger algebras. More recently, Putnam and Spielberg
\cite{Putnam-Spielberg:Structure} considerably simplified the groupoid constructions of the above algebras when the Smale space is mixing, up to Morita equivalence. Putnam discussed in
\cite[Section~2]{Putnam:Functoriality} how this simplification extends to the \textit{non-wandering} case. We shall only be interested in Smale spaces that are irreducible, which is a stronger condition than
non-wandering.

For the remainder of this section we will outline the construction of the stable
algebra $S(X,P)$ and the stable Ruelle algebra $S(X,P) \rtimes \ZZ$. A more detailed version of
these constructions is given in \cite{Putnam-Spielberg:Structure} and
\cite[Section~3]{Kaminker-Putnam-Whittaker:K-duality}.

Given an irreducible Smale space
$(X,\varphi)$, we fix a non-empty finite $\varphi$-invariant set of periodic points $P$ (in the
irreducible case periodic points are dense). Then we define $X^{u}(P) = \bigcup_{p \in P}
X^{u}(p)$, which is given a locally compact and Hausdorff topology generated by the collection
$\{X^u(x,\varepsilon) \mid x \in X^u(P),\; \varepsilon\in (0,\varepsilon_X]\}$.

The groupoid of the stable equivalence relation is
\begin{equation}
\label{stable groupoid def}
G^s(P) := \{(v,w) \in X \times X \mid v \sim_s w \textrm{ and } v,w \in X^u(P) \}.
\end{equation}
The stable groupoid can be endowed with an \'{e}tale topology, see
\cite[Lemma~3.1]{Kaminker-Putnam-Whittaker:K-duality} for details. With this structure $G^s(P)$ is
an amenable locally compact Hausdorff \'etale groupoid. The stable $C^*$-algebra $S(X,P)$ is
defined to be the groupoid $C^*$-algebra associated with $G^{s}(P)$.

There is a canonical automorphism of the $C^*$-algebra $S(X,P)$ induced by
the automorphism of the underlying groupoid $G^s(P)$ defined by
$\alpha:=\varphi \times \varphi$. This automorphism of $G^s(P)$ gives rise to
a semidirect product groupoid $G^s(P)\rtimes_\alpha \ZZ$, which is again an
amenable locally compact Hausdorff \'{e}tale groupoid. The stable Ruelle
algebra is the crossed product $S(X,P) \rtimes_{\alpha} \ZZ\cong
C^*(G^s(P)\rtimes_\alpha \ZZ)$, where we also write $\alpha$ for the
automorphism of $S(X,P)$ induced by $\alpha \in \Aut(G^s(P))$. Putnam
explains in \cite[Section~2]{Putnam:Functoriality} how $S(X,P)
\rtimes_{\alpha} \ZZ$ is strongly Morita equivalent to the Ruelle algebra
originally defined by Putnam in \cite{Putnam:Algebras}, building from the
similar result of Putnam and Spielberg in the mixing case
(\cite{Putnam-Spielberg:Structure}). The result of
\cite{Putnam-Spielberg:Structure} that the Ruelle algebra is separable,
simple, stable, nuclear, purely infinite, and satisfies the UCT extends
readily to the irreducible case.

A similar construction gives the unstable groupoid $G^u(P)$, the unstable algebra $U(X,P)$ and the
associated unstable Ruelle algebra $U(S,P)\rtimes \ZZ\cong C^*(G^u(P)\rtimes\ZZ)$. Alternatively,
the stable algebras for the Smale space $(X,\varphi^{-1})$ with the opposite bracket map are
isomorphic to the relevant unstable algebras for $(X,\varphi)$.

\section{The limit space of a self-similar groupoid action}\label{sec:limit space}

In this section we generalise Nekrashevych's construction of the limit space of a self-similar
group \cite[Chapter~3]{Nekrashevych:Self-similar} to the situation of self-similar groupoid
actions.

\begin{dfn}\label{def:asymptotic equivalence}
Let $E$ be a finite directed graph. Let $(G, E)$ be a self-similar groupoid
action. We say that left-infinite paths $x,y  \in E^{-\infty}$ are
\emph{asymptotically equivalent}, and write $x \sim_{\aeq} y$ if there is a
sequence $(g_n)_{n < 0}$ in $G$ such that $\{g_n \mid n < 0\}$ is a finite
set, and such that
\[
g_n \cdot x_n \dots x_{-1} = y_n \dots y_{-1}\quad\text{ for all }n < 0.
\]
\end{dfn}

If the sequence $(g_n)$ implements an asymptotic equivalence $x \sim_{\aeq} y$ and the sequence
$(h_n)$ an asymptotic equivalence $y \sim_{\aeq} z$ then $d(h_n) = c(g_n)$ for all $n$, and
$(h_ng_n)$ implements and asymptotic equivalence $x \sim_{\aeq} z$. Moreover, the sequence
$g_n^{-1}$ implements an asymptotic equivalence $y \sim_{\aeq} x$, and the sequence $(r(x_n))_{n <
0}$ implements an asymptotic equivalence $x \sim_{\aeq} x$. So $\sim_{\aeq}$ is an equivalence
relation.

\begin{dfn}\label{dfn:limit space}
Let $E$ be a finite directed graph. Let $(G, E)$ be a self-similar groupoid
action. The \emph{limit space} of $(G, E)$ is defined to be the quotient
space $\Jj_{G, E} := E^{-\infty}/{\sim_{\aeq}}$.
\end{dfn}

The limit space is typically not a Hausdorff space, but, just as in the
setting of \cite{Nekrashevych:Self-similar}, it is guaranteed to be Hausdorff
if the self-similar action is contracting in the following sense, introduced in
\cite{Nekrashevych:Self-similar,Laca-Raeburn-Ramagge-Whittaker:Equilibrium}.

\begin{dfn}\label{dfn:contracting} We say that a self-similar groupoid action $(G, E)$ on a finite directed graph $E$ is \emph{contracting} if there is a finite
subset $F \subseteq G$ such that for every $g  \in G$ there exists $n \ge 0$
such that $\{g|_{\mu} : \mu \in d(g)E^n\} \subseteq F$. Any such finite set
$F$ is called a \emph{contracting core} for $(G, E)$. The \emph{nucleus} of
$G$ is the set
\[
    \Nn_{G, E} := \bigcap\{F \subseteq G \mid F\text{ is a contracting core for }(G, E)\}.
\]
We will frequently just write $\Nn$ rather than $\Nn_{G, E}$ when the self-similar groupoid action
in question is clear from context.
\end{dfn}

Just as in the setting of self similar groups, the nucleus is the minimal contracting core for
$(G, E)$, and is symmetric, closed under restriction and contains $G^{(0)}$:

\begin{lem}\label{lem:N a core}
Let $E$ be a finite directed graph. Let $(G, E)$ be a contracting self-similar
groupoid action. Then $\Nn$ is a contracting core for $(G, E)$ and is contained in any other
contracting core for $(G, E)$. We have
\begin{equation}\label{eq:nucleus formula}
\Nn = \bigcup_{g \in G} \bigcap^\infty_{n=1} \{g|_\mu \mid \mu \in d(g)E^*, |\mu| \ge n\}.
\end{equation}
We have $\Nn = \Nn^{-1}$, $\Nn$ is closed under restriction, and if $E$ has no sinks, then $G^{(0)} \subseteq \Nn$.
\end{lem}
\begin{proof}
For the first statement, we first show that the collection of contracting
cores for $(G, E)$ is closed under intersections. If $F, K$ are contracting
cores for $G$ and $g  \in G$, then there exist $M, N$ such that $g|_\mu \in
F$ whenever $|\mu| \ge M$ and $g|_\nu \in K$ whenever $|\nu| \ge N$. In
particular, if $|\mu| \ge \max\{M,N\}$ then $g|_\mu \in F \cap K$. So $F \cap
K$ is a contracting core.

Now since contracting cores are, by definition, finite, there is a finite collection $\Ff$ of
contracting cores such that $\Nn = \bigcap \Ff$. So the preceding paragraph shows that $\Nn$ is a
contracting core. It is then contained in any other contracting core by definition.

Let $\Mm := \bigcup_{g \in G} \bigcap^\infty_{n=1} \{g|_\mu \mid \mu \in
d(g)E^*, |\mu| \ge n\}$. Fix $h  \in \Mm$. Then there exists $g  \in G$ and a
sequence $(\mu_i)^\infty_{i=1}$ of finite paths such that $|\mu_i| \to
\infty$ and $g|_{\mu_i} = h$ for all $i$. By definition of $\Nn$ there exists
$N$ such that $g|_\mu \in \Nn$ whenever $|\mu| \ge N$. Since $n_i \to \infty$
we have $n_i
> N$ for some $i$, and so $h = g|_{\mu_i} \in \Nn$. So $\Mm \subseteq \Nn$. For the reverse
containment, observe that we have just seen that $\Mm$ is finite. Fix $g  \in
G$. Since $G$ is contracting, the sets $R_n := \{g|_\mu \mid |\mu| \ge n\}$
indexed by $n  \in \NN$ are all finite, and they are decreasing with respect
to set containment. So there exists $N$ such that $R_n = R_N$ for all $n \ge
N$. It follows that $g|_\mu \in R_N \subseteq \Mm$ whenever $|\mu| \ge N$. So
$\Mm$ is a contracting core for $(G, E)$ and therefore $\Nn \subseteq \Mm$ by
the first assertion of the lemma.

To see that $\Nn = \Nn^{-1}$, fix $h  \in \Nn$. Then~\eqref{eq:nucleus
formula} shows that there exists $g  \in G$ and a sequence
$(\mu_i)^\infty_{i=1}$ in $d(g)E^*$ such that $|\mu_i| \to \infty$ and
$g|_{\mu_i} = h$ for all $i$. We then have $h^{-1} = (g|_{\mu_i})^{-1} =
g^{-1}|_{g\cdot \mu_i}$ for all $i$, and so~\eqref{eq:nucleus formula} gives
that $h^{-1} \in \Nn$.

That $\Nn$ is closed under restriction follows immediately from \eqref{eq:nucleus formula}.

Finally, if $E$ has no sinks, then for each $v  \in E^0$ and $n\geq 0$,
$E^{n}v\neq\emptyset$. Let $w$ be a vertex such that $wE^{n}v\neq\emptyset$
for infinitely many $n \in \mathbb{N}$. Then, $v \in \{w|_{\mu}\text{
}:\text{ }\mu\in wE^{n}, |\mu|\geq n\}$ for all $n \in \mathbb{N}$, so that
$v \in \mathcal{N}$.
\end{proof}

\begin{ntn}\label{ntn:Rk}
Let $E$ be a finite directed graph with no sources. Let $(G, E)$ be a contracting self-similar
groupoid action with nucleus $\Nn$. Since $\Nn$ is finite, so are the sets
\[\textstyle
\Nn^k := \big\{\prod^k_{i=1} g_i \mid g_1, \dots, g_k \in \Nn\text{ and }
                                        d(g_i) = c(g_{i+1})\text{ for all }i\Big\}.
\]
We write $R_k$ for the integer
\[
R_k := \min\{j \in \NN \mid h|_\mu \in \Nn\text{ for all } h \in \Nn^k\text{ and }\mu \in E^j\}.
\]
So $R_1 = 0$, and $R_i \le R_{i+1}$ for all $i$.
\end{ntn}

We now show that if $x,y  \in E^{-\infty}$ are asymptotically equivalent,
then the sequence $g_i$ implementing the asymptotic equivalence can be taken
to belong to $\Nn$ and to be consistent with respect to restriction.

\begin{lem}\label{lem:ae by nucleus}
Let $E$ be a finite directed graph. Let $(G, E)$ be a contracting
self-similar groupoid action with nucleus $\Nn$. Then $x,y  \in E^{-\infty}$
are asymptotically equivalent if and only if there exists a sequence
$(h_{n})_{n < 0}$ of elements of $\Nn$ such that $h_n \cdot x_n = y_n$ and
$h_n|_{x_n} = h_{n+1}$ for all $n$.
\end{lem}
\begin{proof}
If there is such a sequence $(h_n)$ of elements of $\Nn$, then for each $n$ we have
\[
h_n \cdot (x_n \dots x_{-1})
    = y_n (h_n|_{x_n} \cdot x_{n+1} \dots x_{-1})
    = y_n (h_{n+1} \cdot x_{n+1} \dots x_{-1})
    = \dots = y_n \dots y_{-1}.
\]
So $x \sim_{\aeq} y$.

Conversely suppose that $x \sim_{\aeq} y$, and fix a sequence $(g_n)_{n < 0}$
in $G$ with just finitely many distinct terms and satisfying $g_n \cdot x_n
\dots x_{-1} = y_n \dots y_{-1}$ for all $n$. Let $S = \{g_n \mid n < 0\}$ be
the finite set of elements appearing in the sequence $(g_n)$. Since $(G, E)$
is contracting, for each $g  \in S$ there exists $k_g$ such that $g|_\mu \in
\Nn$ whenever $|\mu| \ge k_g$. Let $k := \max_{g \in S} k_g$. We construct a
sequence $(h_n)_{n < \infty}$ in $\Nn$ iteratively as follows. Consider the
sequence $(g_n|_{x_n \dots x_{-2}})_{n < -k-1}$. By the choice of $k$, every term
of this sequence belongs to $\Nn$. Since $\Nn$ is finite, there exists
$h_{-1}  \in \Nn$ and a strictly decreasing infinite sequence
$(n^1_i)^\infty_{i=1}$ of integers $n^1_i < -2-k$ such that
$g_{n^1_i}|_{x_{n^1_i} \dots x_{-2}} = h_{-1}$ for all $i$. Since each
$g_{n^1_i} \cdot(x_{n^1_i} \dots x_{-1}) = y_{n^1_i} \dots y_{-1}$, we have
$h_{-1} \cdot x_{-1} = g_{n^1_i}|_{x_{n^1_i} \dots x_{-2}} \cdot x_{-1} =
(g_{n^1_i} \cdot x_{n^1_i} \dots x_{-1})_{-1} = y_{-1}$.

Now suppose that we have chosen $h_{m}, \dots, h_{-1} \in \Nn$ such that each
$h_j \cdot x_j = y_j$ and $h_j|_{x_j} = h_{j+1}$, and a strictly decreasing
sequence $(n^m_i)^\infty_{i=1}$ of integers $n^m_i < m-k-1$ such that
$g_{n^m_i}|_{x_{n^m_i} \dots x_{m-1}} = h_m$ for all $i$. Then the sequence
$(g_{n^m_i}|_{x_{n^m_i} \dots x_{m-2}})_i$ is contained in $\Nn$, so there
exists $h_{m-1} \in \Nn$ and a subsequence $n^{m-1}_i$ of the sequence
$n^m_i$ with the property that $n^{m-1}_i|_{x_{n^m_i} \dots x_{m-2}} =
h_{m-1}$ for all $i$. We have
\[
h_{m-1}|_{m-1}
    = (g_{n^{m-1}_1}|_{x_{n^m_i} \dots x_{m-2}})|_{x_{m-1}}
    = g_{n^{m-1}_1}|_{x_{n^m_i} \dots x_{m-1}}
    = h_m,
\]
and $h_{m-1} \cdot x_{m-1} = y_{m-1}$ by a calculation just like the one we used to see that
$h_{-1} \cdot x_{-1} = y_{-1}$.

The above procedure produces a sequence $(h_n)_{n<1}$ in $\Nn$ with the desired properties.
\end{proof}

\begin{cor}\label{cor:finite orbits}
Let $E$ be a finite directed graph. Let $(G, E)$ be a contracting
self-similar groupoid action with limit space $\Jj = \Jj_{G, E}$. Let $q :
E^{-\infty} \to \Jj$ be the quotient map. For each $x  \in E^{-\infty}$, the
equivalence class $[x] := q^{-1}(q(x))$ satisfies $|[x]| \le |\Nn|$.
\end{cor}
\begin{proof}
Fix $x  \in E^{-\infty}$. Let $y^1, \dots, y^{l}$ be distinct elements of
$[x]$. We must show that $l \le |\Nn|$. Fix $m < 0$ such that the finite
paths $\mu^i := y^i_m \dots y^i_{-1}$ for $i \le l$ are all distinct.
Lemma~\ref{lem:ae by nucleus} implies that there are elements $n^1, \dots,
n^l  \in \Nn$ such that $\mu^i = n^i \cdot x_m \dots x_{-1}$ for all $i$.
Since the $\mu^i$ are distinct, the $n^i$ are distinct, forcing $l \le
|\Nn|$.
\end{proof}

To construct a Smale space from the limit space $\Jj$ we will show that the shift map on
$E^{-\infty}$ descends to a self-mapping of $\Jj$, and that under a regularity hypothesis similar
to that used by Nekrashevych \cite{Nekrashevych:Self-similar}, this self-mapping is locally
expanding and hence a local homeomorphism.

To do this, we need to describe a basis for the
topology on $\Jj$. We start with a preliminary lemma about quotient topologies.

\begin{lem}\label{lem:quotient top}
Let $X$ be a compact and metrisable Hausdorff space and let $\sim$ be an equivalence
relation on $X$. Let $Y := X/{\sim}$ be the quotient space, and $q : X \to Y$
the quotient map. For each $A \subseteq X$, let $U_A := \{y \in Y \mid
q^{-1}(y) \subseteq A\}$. If $A$ is open in $X$, then $U_A$ is open in $Y$.
If $|q^{-1}(y)| < \infty$ for each $y  \in Y$, then for any basis $\Bb$ for
the topology on $X$, the set
\[
    \Uu_\Bb := \{U_B \mid B \text{ is a finite union of elements of } \Bb\}
\]
is a basis for the quotient topology on $Y$. If $q : X \to Y$ is a closed map, then $Y$ is
metrisable.
\end{lem}

\begin{proof}
By definition of the quotient topology, $U_A$ is open in $Y$ if and only if
$q^{-1}(U_A)$ is open in $X$. By definition of $U_A$, we have $[x] \in U_A$
if and only if $[x] \subseteq A$, and so
\[
q^{-1}(U_A) = \{x \in X \mid [x] \subseteq A\} = X \setminus \{x \in X \mid [x] \setminus A \not=\varnothing\}.
\]
So it suffices to show that if a net $(x_i)_{i\in I}$ in $X$ converges to
some $x \in X$, and if each $[x_i] \setminus A$ is nonempty, then $[x]
\setminus A$ is nonempty. To see this, note that for each $i$, there exists
$y_i  \in [x_i] \setminus A$. Since $X$ is compact we can pass to a subnet
$(y_{i_j})$ that converges in $X$. Since $A$ is open, we have that $y :=
\lim_j y_{i_j} \not\in A$. Since the quotient map is continuous we have $q(y)
= \lim_j q(y_{i_j}) = \lim_j q(x_{i_j}) = \lim_i q(x_i) = [x]$, and so $y \in
[x] \setminus A$.

Now suppose that $\Bb$ is a basis for the topology on $X$. Let $V$ be an open
subset of $Y$ and fix $y  \in V$. Since $q^{-1}(V)$ is open in $X$, for each
point $x  \in q^{-1}(y)$, we can find $B_x \in \Bb$ such that $x \in B_x
\subseteq q^{-1}(V)$. Let $B := \bigcup_{x \in q^{-1}(y)} B_x$. Since
$q^{-1}(y)$ is finite, this is a finite union of elements of $\Bb$, so it
suffices to show that $y \in U_B \subseteq V$. By definition of $B$ we have
$q^{-1}(y) \subseteq B$ and so $y \in U_B$. To see that $U_B \subseteq V$,
take $y'  \in U_B$. Then $q^{-1}(y) \subseteq B$ by definition of $U_B$.
Since each $B_x \subseteq q^{-1}(V)$, we have $B \subseteq q^{-1}(V)$ and
hence $q(B) \subseteq V$. So $y \in q(B) \subseteq V$, as required.

The last statement follows from \cite[Theorem~4.2.13]{Engelking:General_Topology}.
\end{proof}

Our next lemma describes how asymptotic equivalence interacts with the action of the nucleus on cylinder sets.

\begin{lem}\label{lem:asymp cylinders}
Let $E$ be a finite directed graph. Let $(G, E)$ be a contracting
self-similar groupoid action. If there exists $g  \in \Nn$ and $\mu, \nu $ in
$E^*$ such that $g \cdot \mu=\nu$, then $q(\lZ{\mu}) \cap q(\lZ{\nu}) \neq
\varnothing$.
\end{lem}
\begin{proof}
Fix $\mu, \nu$, and $g$. Since $\Nn$ is closed under restriction, there exist
$e \in E^1$ and $h \in \Nn$ such that $h|_e=g$. Let $f := h \cdot e$. We
claim that $s(e)=r(\mu)$ and $s(f)=r(\nu)$ so that $h\cdot e\mu=f\nu$.
Indeed, by Lemma~\ref{properties SSA}(1) we have
\[
s(f)=s(h\cdot e)=h|_e \cdot s(e)=g\cdot s(e).
\]
Since $d(g)=r(\mu)$ we have that $s(e)=r(\mu)$ and then $s(f)=g\cdot
r(\mu)=r(\nu)$, proving the claim. By applying the above procedure
recursively, we can construct paths $x  \in \lZ{\mu}$ and $y  \in \lZ{\nu}$
such that $x \sim_{\aeq} y$.
\end{proof}

We can now describe a basis for the topology on the limit space of a contracting self-similar groupoid action.

\begin{cor}\label{cor:limit space basis}
Let $E$ be a finite directed graph with no sources or sinks. Let $(G, E)$ be contracting self-similar groupoid action. The sets
\[\textstyle
    U_\mu := \Big\{y \in \Jj \mid q^{-1}(y) \subseteq \bigcup_{g \in \Nn \cap d^{-1}(r(\mu))} \lZ{g \cdot \mu}\Big\},
\]
indexed by $\mu  \in E^*$, are a basis for the topology on $\Jj$.
\end{cor}
\begin{proof}
By Lemma~\ref{lem:quotient top} we know that $U_\mu$ is open. Now fix an open
set $V \subseteq \Jj$ and $y \in V$.

If $x \in Z(\mu]$ for some $\mu \in E^{n}$, and $x'\sim_{a.e} x$, then by
Lemma~\ref{lem:ae by nucleus} there exists $g \in \mathcal{N}\cap
d^{-1}(r(\mu))$ such that $x' \in Z(g\cdot\mu]$. So, if $q(x) = y$, then $y
\in U_{x_{-n}...x_{-1}}$ for all $n \in \mathbb{N}$. Let $X_{n} = \bigcup_{g
\in \Nn \cap d^{-1}(r(x_{-n}))} \lZ{g \cdot x_{-n}...x_{-1}}$. Since $\Nn$ is
closed under restriction, $X_{n+1}\subseteq X_{n}$ for all $n \in
\mathbb{N}$. By Lemma~\ref{lem:ae by nucleus}, $\bigcap_{n\in\mathbb{N}}X_{n}
= q^{-1}(y)$. Hence, the compact sets $Y_{n} := q(X_{n})$ satisfy
$Y_{n+1}\subseteq Y_{n}$ and $\bigcap_{n\in\mathbb{N}}Y_{n} = \{y\}$.
Therefore, there exists $k \in \mathbb{N}$ such that $Y_{k}\subseteq V$.
Since $y \in U_{x_{-k}...x_{-1}}\subseteq Y_{k}$, the result follows.
\end{proof}
In addition to having a canonical basis, one can also find a metric for $\Jj$ inducing the same topology, as the next corollary of Lemma~\ref{lem:quotient top} shows.
\begin{cor}\label{cor:limit space metrisable}
Let $(G, E)$ be a contracting self-similar groupoid action. Then its limit space $\Jj$
is compact and metrisable.
\end{cor}
\begin{proof}
Recall that $(E^{-\infty},d)$ is compact in the  product metric $d$
of~\eqref{eq:cylinder metric}. By Lemma~\ref{lem:quotient top}, it suffices to show that $\sim_{\aeq}$ is
a closed equivalence relation. For this, suppose that $C$ is a closed subset
of $E^{-\infty}$, and fix a net $(x_\alpha) \in C$ and a point $x $ in
$E^{-\infty}$ such that $q(x_\alpha)$ converges to $q(x)$. We must show that
$q(x) \in q(C)$. Since $E^{-\infty}$ is compact, so is $C$, so we may assume
that $x_\alpha$ converges to some $z  \in C$. Hence $q(x_\alpha) \to q(z)$.
So we must show that $x \sim_{\aeq} z$.

Fix an integer $n < 0$. By Corollary~\ref{cor:limit space basis}, we have
that $q(x_\alpha) \in U_{x_n \dots x_{-1}}$ for large $\alpha$, and since
$x_\alpha$ converges to $z$ we have that $x_\alpha \in \lZ{z_{n} \dots
z_{-1}}$ for large $\alpha$. It follows that there exist $g_n \in \Nn$ such
that $\lZ{g_n \cdot x_n \dots x_{-1}} \cap \lZ{z_n \dots z_{-1}} \not=
\varnothing$. Since $\lZ{\mu} \cap \lZ{\nu} = \varnothing$ for distinct
$\mu,\nu  \in E^n$, we deduce that $g_n \cdot x_n \dots x_{-1} = z_n \dots
z_{-1}$. Since $n < 0$ was arbitrary, it follows that $x \sim_{\aeq} z$ as
claimed.
\end{proof}

\subsection{Schreier graphs and recurrent self-similar actions}

Schreier graphs define useful combinatorial approximations to the limit space
of self-similar actions. We begin with the definition that suits our
situation. Note that Schreier graphs of groups have a rather general
definition that generalises Cayley graphs.

\begin{dfn}
Let $(G,E)$ be a finitely generated self-similar groupoid, and let $A$ be a
generating set for $G$ that is closed under inverses and restriction. The
level-$n$ Schreier graph $\Gamma_n:=\Gamma_n(G,A)$ is the (undirected) graph
with vertex set $\Gamma_n^0:=E^n$ and an edge labelled by $a \in A$ between
$\mu$ and $\nu$ if and only if $d(a)=r(\mu)$ and $a \cdot \mu = \nu$.
\end{dfn}

Note that we could label an edge in $\Gamma_n$ by either $a \in A$ or
$a^{-1}$ since $A$ is closed under inverses and if $a \cdot \mu = \nu$, then
$a^{-1} \cdot \nu = \mu$. We will make use of the geodesic distance,
$d_{\text{geo}}$, on the vertex set of an undirected graph:
$d_{\text{geo}}(v,w)$ is the minimum length of a path between $v$ and $w$.
The following generalises \cite[Proposition
3.6.6]{Nekrashevych:Self-similar}; the proof is virtually identical.

\begin{prp}\label{prp:schreier approximation}
Let $(G, E)$ be a finitely generated, contracting self-similar groupoid
action on a finite directed graph $E$. Let $A$ be a finite generating set for $G$ that is closed under
inverses and restriction. Let $\Gamma$ be the level-n Schreier graph
$\Gamma_n(G, A)$. There is a map $\psi_n:\Gamma_n \to \Gamma_{n-1}$ defined
by
\begin{align*}
&\psi_n(e\mu)=\mu \quad \text{for } e\text{ in } E^1 \text{ and } v \text{ in } s(e)E^{n-1} \\
&\psi_n(a: e\mu \to f\nu)=a|_e: \mu \to \nu.
\end{align*}
For $x,y \in E^{-\infty}$, the sequence $\big(d_{\text{geo}}(x_{-n} \ldots
x_{-1},y_{-n}\ldots y_{-1})\big)^\infty_{n=1}$ is bounded if and only if $x$
and $y$ are asymptotically equivalent.
\end{prp}

We now generalise Nekrashevych's notion of a recurrent self-similar group action to groupoid actions. While we do not require recurrence for the main results of this paper, it does illuminate interesting
topological properties of the dynamics and the limit space. We note that Nekrashevych synonymously uses \emph{recurrence}, \emph{self-replicating}, and \emph{fractal} for the notion below.

\begin{dfn}\label{recurrent}
A self-similar groupoid action $(G, E)$ is said to be \emph{recurrent} if,
for any $e,f  \in E^1$ and $h  \in G$ with $d(h)=s(e)$ and $c(h)=s(f)$, there is $g$ in
$Gr(e)$ such that $g\cdot e=f$ and $g|_e=h$.
\end{dfn}

Recurrence of a self-similar groupoid action is obviously a rather strong condition. For example, if $(G,E)$ is recurrent, then we immediately see that the in-degree of all vertices of the graph must be equal.
Another immediate consequence of recurrence is the following.

\begin{prp}
Suppose $(G,E)$ is a recurrent self-similar groupoid action on a finite directed graph $E$. Then the action of $G$ on $E^*$ is level-transitive.
\end{prp}
\begin{proof}
For paths of length one, transitivity follows immediately from recurrence.
Now suppose that for any paths $\mu$ and $\nu$ of length $n$ and any $h  \in
G$ with $d(h)=s(\mu)$ there exists $g  \in G$ with $d(g)=r(\mu)$ such that $g
\cdot \mu=\nu$ and $g|_{\mu}=h$; that is, $G$ acts transitively on paths of
length $n$ with specified restriction as in the definition of recurrence. We
now consider paths $\lambda$ and $\rho$ of length $n+1$ and aim to show that
there exists $g  \in G$ such that $g\cdot \lambda=\rho$. Recurrence implies
that there exists $g_{n+1}  \in G$ with $d(g_{n+1})=r(\lambda_{n+1})$ such
that $g_{n+1}\cdot \lambda_{n+1}=\rho_{n+1}$.  Now the inductive hypothesis
implies that there exists $g  \in G$ with $d(g)=r(\lambda)$ such that $g
\cdot \lambda_1 \cdots \lambda_n=\rho_1 \cdots \rho_n$ and $g|_{\lambda_1
\cdots \lambda_n}=g_{n+1}$. Thus we have
\[
g\cdot \lambda=g \cdot \lambda_1 \cdots \lambda_n \lambda_{n+1}=(\rho_1 \cdots \rho_n)g_{n+1}\cdot \lambda_{n+1}=\rho_1 \cdots \rho_n\rho_{n+1}=\rho,
\]
the desired result.
\end{proof}

Following Nekrashevych, we now look to connectedness of the limit space, but first we will generalise \cite[Proposition 2.11.3]{Nekrashevych:Self-similar}.

\begin{prp}\label{nuc_gen}
Suppose that $(G,E)$ is a contracting, recurrent self-similar groupoid action on a finite directed graph $E$ and that $G$ is finitely generated. Then the nucleus $\Nn$ of $(G,E)$ is a
generating set.
\end{prp}
\begin{proof}
Let $H$ be a finite generating set for $G$. Then there exists $m \in
\mathbb{N}$ such that for every $h \in H$, the set $\{h|_{\mu} \mid |\mu|
\geq m\}$ is contained in $\Nn$. Given $g \in G$, we have $g=h_1 h_2 \cdots
h_n$ with $h_i \in H$ for $1 \leq i \leq n$. Since $(G,E)$ is recurrent, for
each $h_i$, there exists $a_i \in G$ and $\mu_i \in s(a_i)E^m$ such that
$a_i|_{\mu_i}=h_i$. Since $a_i$ is a product of elements of $H$ it follows
that $a_i|_{\mu_i}=h_i$ is a product of elements of $\Nn$. Thus $G$ is
generated by $\Nn$.
\end{proof}

The following proof follows Nekrashevych's \cite[Proposition 3.3.10 and
Theorem 3.5.1]{Nekrashevych:Self-similar}, which he in turn partially
attributes to K. Pilgrim and P. Haissinsky (private communication).

\begin{thm}
Suppose $(G,E)$ is a contracting self-similar groupoid action on a finite directed graph $E$ and that $G$ is finitely generated. Then the limit space $\Jj_{(G,E)}$ is connected if and only if $(G,E)$ level-transitive.
\end{thm}
\begin{proof}
First suppose that $(G,E)$ is level transitive. We suppose that
$\Jj=\Jj_{(G,E)}$ is not connected, and derive a contradiction. Let $H$ be a
finite generating set for $G$. Fix closed, non-empty subsets $A, B \subset
\Jj$ such that $A \cup B=\Jj$ and $A \cap B = \varnothing$. Let
$X_A=q^{-1}(A)$ and $X_B=q^{-1}(B)$. Then $X_A$ and $X_B$ are closed,
non-empty subsets of $E^{-\infty}$ such that $X_A \cup X_B = E^{-\infty}$ and
$X_A \cap X_B=\varnothing$. Define
\[
X_A^n:=\{a_{-n} \cdots a_{-1} \mid a=\ldots a_{-n-1}a_{-n} \ldots a_{-1} \in X_A\}.
\]
Since $X_A$ is open, we write it as a union $X_A = \bigcup_{\mu \in I}
\lZ{\mu}$ of cylinder sets. Since $X_A$ is also compact, there is a finite $F
\subseteq I$ such that $X_A = \bigcup_{\mu \in F} \lZ{\mu}$. Put $N =
\max\{|\mu| : \mu \in F\}$. For each $\mu \in F$, we have $\lZ{\mu} =
\bigcup_{\nu \in E^*\mu \cap E^n} \lZ{\nu}$. So $\lZ{\nu} \subset X_A$
whenever for any $\nu \in X_A^n$.

Since $(G, E)$ is level-transitive and contracting, for each $n \geq N$,
there exist $a_{-n} \ldots a_{-1}$ in $X_A^n$ and $g_n \in \Nn$ such that $g_n
\cdot a_{-n} \cdots a_{-1} \in X_B$. By Lemma~\ref{lem:asymp cylinders},
there exist $\mu_n \in X_A \cap E^* (a_{-n} \cdots a_{-1})$ and $\nu_n \in
X_B \cap E^*(g_n \cdot a_{-n} \cdots a_{-1})$. Since $X_A, X_B$ are compact,
there is an increasing sequence $(n_k)^\infty_{k=1}$ of natural numbers such
that $(\mu_{n_{k}})_k$ and $(\nu_{n_k})_k$ both converge, say to $\mu_{n_k}
\to \mu \in X_A$ and $\nu_{n_k} \to \nu \in X_B$. Since $\Nn$ is finite, $\mu
\sim_\aeq \nu$. So $q_{\aeq}(\mu) = q_{\aeq}(\nu) \in A \cap B$, a
contradiction. Thus $\Jj$ is connected.

Now suppose that $(G, E)$ is not level-transitive. Fix $n \in \NN$ and $a \in
E^{n}$ such that $G\cdot a\neq E^{n}$. Define $A' := \bigcup_{a'\in G\cdot p}
Z(a']$ and $B': = \bigcup_{b'\in E^{n}\setminus G\cdot p}Z(b']$. Then, $A :=
q(A')$ and $B := q(B')$ are disjoint compact sets in $\Jj$ such that $A\cup B
= \Jj$.
\end{proof}

\begin{cor}
Suppose that $(G,E)$ is a contracting and recurrent self-similar groupoid
action on a finite directed graph $E$ such that $G$ is finitely generated. Then the limit space
$\Jj_{(G,E)}$ is connected.
\end{cor}

\begin{exm}\label{s,r clarifying example_schreier}
\begin{figure}
\begin{tikzpicture}
\begin{scope}[xshift=0cm,yshift=0cm]
\def\x{1.0}
\node at (-3,0) {$\Gamma_1$:};
\node[vertex] (vert_00) at (1:\x)   {$1$}
	edge [->,>=latex,out=10,in=-10,loop,thick] node[right,pos=0.5]{$v$} (vert_00);
\node[vertex] (vert_10) at (90:\x)   {$4$}
	edge [->,>=latex,out=100,in=80,loop,thick] node[left,pos=0.5]{$w$} (vert_10)
	edge [<-,>=latex,out=0,in=90,thick] node[above,pos=0.5]{$a$} (vert_00);
\node[vertex] (vert_01) at (180:\x)   {$2$}
	edge [->,>=latex,out=190,in=170,loop,thick] node[left,pos=0.5]{$v$} (vert_01)
	edge [<-,>=latex,out=90,in=180,thick] node[above,pos=0.5]{$b$} (vert_10);
\node[vertex] (vert_11) at (270:\x)   {$3$}
	edge [->,>=latex,out=260,in=280,loop,thick] node[left,pos=0.5]{$w$} (vert_11)
	edge [<-,>=latex,out=180,in=270,thick] node[below,pos=0.5]{$a$} (vert_01)
	edge [->,>=latex,out=0,in=270,thick] node[below,pos=0.5]{$b$} (vert_00);
\end{scope}
\begin{scope}[xshift=7cm,yshift=0cm]
\def\x{2}
\node at (-3.9,0) {$\Gamma_2$:};
\node[vertex] (vert_000) at (1:\x)   {$11$}
	edge [->,>=latex,out=10,in=-10,loop,thick] node[right,pos=0.5]{$v$} (vert_000);
\node[vertex] (vert_100) at (45:\x)   {$41$}
	edge [->,>=latex,out=55,in=35,loop,thick] node[right,pos=0.5]{$w$} (vert_100)
	edge [<-,>=latex,out=-45,in=90,thick] node[right,pos=0.5]{$a$} (vert_000);
\node[vertex] (vert_010) at (90:\x)   {$24$}
	edge [->,>=latex,out=100,in=80,loop,thick] node[left,pos=0.5]{$v$} (vert_010)
	edge [<-,>=latex,out=0,in=135,thick] node[above,pos=0.5]{$b$} (vert_100);
\node[vertex] (vert_110) at (135:\x)   {$32$}
	edge [->,>=latex,out=145,in=125,loop,thick] node[left,pos=0.5]{$w$} (vert_110)
	edge [<-,>=latex,out=45,in=180,thick] node[above,pos=0.5]{$a$} (vert_010);
\node[vertex] (vert_001) at (180:\x)   {$12$}
	edge [->,>=latex,out=190,in=170,loop,thick] node[left,pos=0.5]{$v$} (vert_001)
	edge [<-,>=latex,out=90,in=225,thick] node[left,pos=0.5]{$b$} (vert_110);
\node[vertex] (vert_101) at (225:\x)   {$42$}
	edge [->,>=latex,out=235,in=215,loop,thick] node[left,pos=0.5]{$w$} (vert_101)
	edge [<-,>=latex,out=135,in=270,thick] node[left,pos=0.5]{$b$} (vert_001);
\node[vertex] (vert_011) at (270:\x)   {$23$}
	edge [->,>=latex,out=260,in=280,loop,thick] node[left,pos=0.5]{$v$} (vert_011)
	edge [<-,>=latex,out=180,in=315,thick] node[below,pos=0.5]{$a$} (vert_101);
\node[vertex] (vert_111) at (315:\x)   {$31$}
	edge [->,>=latex,out=305,in=325,loop,thick] node[right,pos=0.5]{$w$} (vert_111)
	edge [<-,>=latex,out=225,in=0,thick] node[below,pos=0.5]{$a$} (vert_011)
	edge [->,>=latex,out=45,in=270,thick] node[right,pos=0.5]{$b$} (vert_000);
\end{scope}
\end{tikzpicture}
\caption{The first two Schreier graphs of Example~\ref{s,r clarifying example_schreier}.}
\label{Schreier_clarifying}
\end{figure}
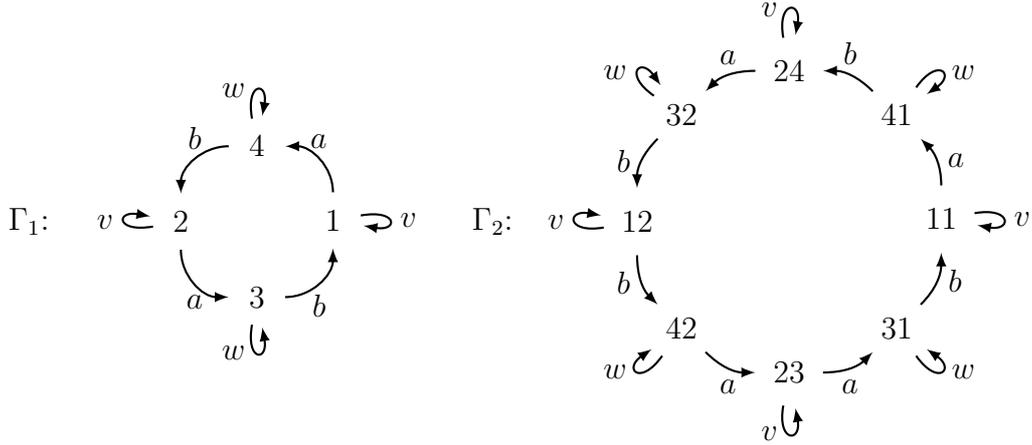

Consider Example~\ref{s,r clarifying example}. We claim that the action is
contracting with nucleus $\Nn=\{v,w,a,b,a^{-1},b^{-1}\}$. Indeed, since all
elements of the automaton appear as restrictions, $v,w,a,b, a^{-1}, b^{-1}
\in \Nn$. To see that this is everything we compute
\begin{align*}
(ab)|_3&=a|_{b \cdot 3}b|_3=v  &(ba)|_1=b|_{a \cdot 1}a|_1=a \\
(ab)|_4&=a|_{b \cdot 4}b|_4=ba   &(ba)|_2=b|_{a \cdot 2}a|_2=b
\end{align*}
and all groupoid elements of length $2$ restrict to the nucleus.

The first two Schreier graphs are depicted in
Figure~\ref{Schreier_clarifying}. More generally, the $n$th Schreier graph is
a cycle of length $2^n$ with a loop at each vertex, showing that the action
is level-transitive. This also suggests that the limit space is homeomorphic
to a circle. One can prove this, by showing inductively that the vertices of
the $n$th Schreier graph can be mapped to the $n$th roots of unity on the
complex circle, metrised so that it has circumference 1, in a way that
extends the pictures in Figure~\ref{Schreier_clarifying}. Specifically, each
vertex is connected in the Schreier graph to its two nearest neighbours on
the circle, and for any infinite path $\mu \in E^{-\infty}$ the images of its
initial segments, regarded as vertices of Schreier graphs, on the unit circle
converge. The map that sends $\mu$ to the limit-point is the desired
homeomorphism: it is continuous because it is a contraction; it is surjective
because its image is both dense and compact; and one checks that it is
injective using the final statement of Proposition~\ref{prp:schreier
approximation}.
\end{exm}

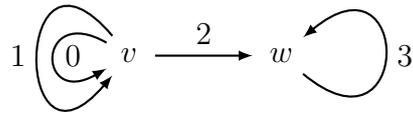
\begin{figure}[h!]
\begin{tikzpicture}
\node[vertex] (vert_x) at (-1,0)   {$v$}
	edge [->,>=latex,out=150,in=210,loop,thick,looseness=8] node[right,pos=0.5]{$0$} (vert_x)
	edge [->,>=latex,out=130,in=230,loop,thick,looseness=10] node[left,pos=0.5]{$1$} (vert_x);
\node[vertex] (vert_y) at (1,0)   {$w$}
	edge [->,>=latex,out=-40,in=40,loop,thick,looseness=10] node[right,pos=0.5]{$3$} (vert_y)
	edge [<-,>=latex,out=180,in=0,thick] node[above,pos=0.5]{$2$} (vert_x);
\end{tikzpicture}
\caption{Graph $E$ for Example~\ref{basilica type example}}
\label{basilica type example graph E}
\end{figure}

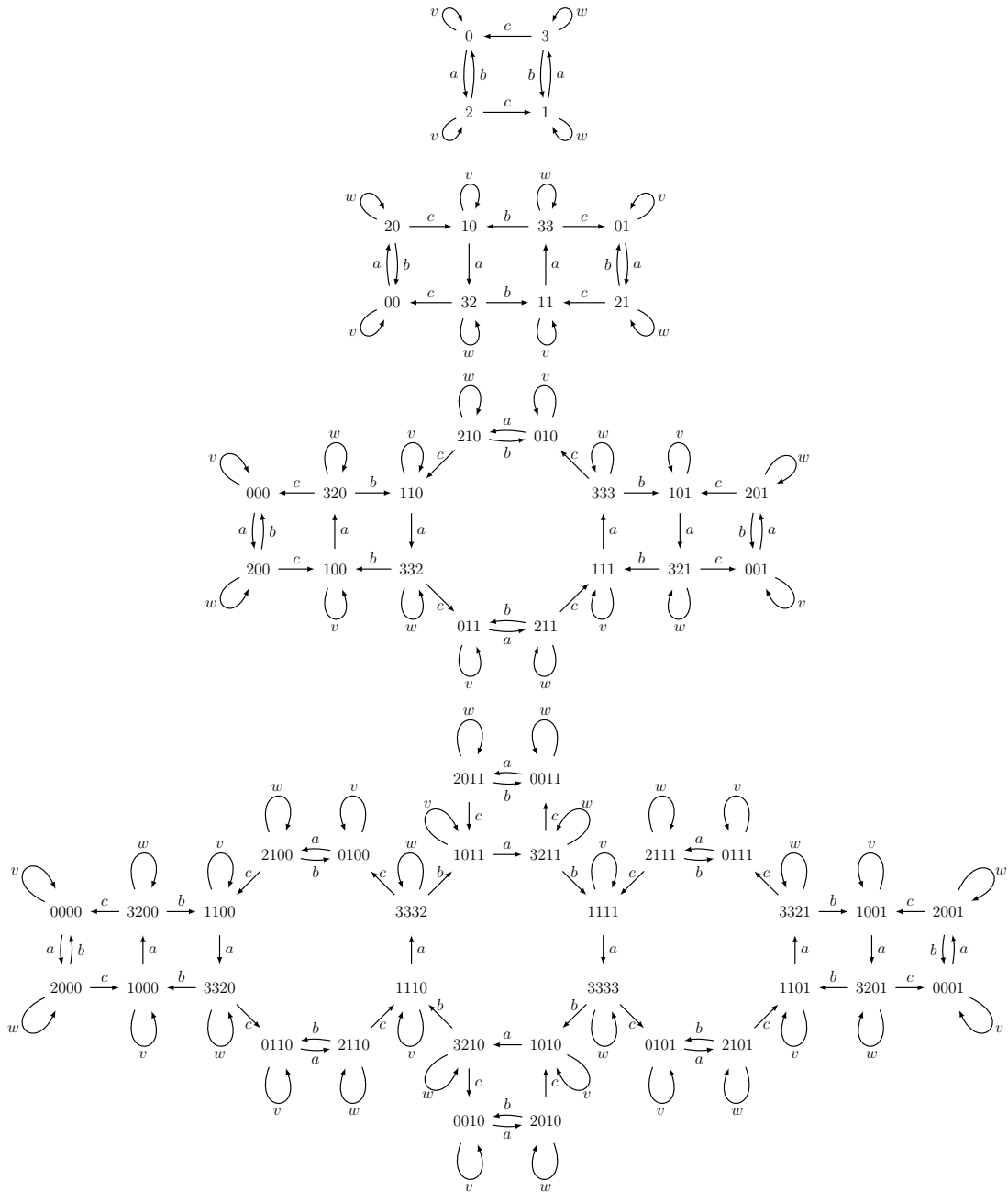
\begin{figure}[h!]\scalebox{0.55}{
\begin{tikzpicture}
\begin{scope}[xshift=0cm,yshift=0cm]
\node[vertex] (vert_0) at (-1,1)   {$0$}
	edge [->,>=latex,out=155,in=115,loop,thick] node[left,pos=0.5]{$v$} (vert_0);
\node[vertex] (vert_3) at (1,1)   {$3$}
	edge [->,>=latex,out=25,in=65,loop,thick] node[right,pos=0.5]{$w$} (vert_3)
	edge [->,>=latex,out=180,in=0,thick] node[above,pos=0.5]{$c$} (vert_0);
\node[vertex] (vert_2) at (-1,-1)   {$2$}
	edge [->,>=latex,out=-155,in=-115,loop,thick] node[left,pos=0.5]{$v$} (vert_2)
	edge [->,>=latex,out=80,in=280,thick] node[right,pos=0.5]{$b$} (vert_0)
	edge [<-,>=latex,out=100,in=260,thick] node[left,pos=0.5]{$a$} (vert_0);
\node[vertex] (vert_1) at (1,-1)   {$1$}
	edge [->,>=latex,out=-25,in=-65,loop,thick] node[right,pos=0.5]{$w$} (vert_1)
	edge [->,>=latex,out=80,in=280,thick] node[right,pos=0.5]{$a$} (vert_3)
	edge [<-,>=latex,out=100,in=260,thick] node[left,pos=0.5]{$b$} (vert_3)
	edge [<-,>=latex,out=180,in=0,thick] node[above,pos=0.5]{$c$} (vert_2);
\end{scope}

\begin{scope}[xshift=0cm,yshift=-5cm]
\node[vertex] (vert_20) at (-3,1)   {$20$}
	edge [->,>=latex,out=155,in=115,loop,thick] node[left,pos=0.5]{$w$} (vert_20);
\node[vertex] (vert_10) at (-1,1)   {$10$}
	edge [<-,>=latex,out=180,in=0,thick] node[above,pos=0.5]{$c$} (vert_20)
	edge [->,>=latex,out=110,in=70,loop,thick] node[above,pos=0.5]{$v$} (vert_10);
\node[vertex] (vert_33) at (1,1)   {$33$}
	edge [->,>=latex,out=180,in=0,thick] node[above,pos=0.5]{$b$} (vert_10)
	edge [->,>=latex,out=110,in=70,loop,thick] node[above,pos=0.5]{$w$} (vert_33);
\node[vertex] (vert_01) at (3,1)   {$01$}
	edge [->,>=latex,out=25,in=65,loop,thick] node[right,pos=0.5]{$v$} (vert_01)
	edge [<-,>=latex,out=180,in=0,thick] node[above,pos=0.5]{$c$} (vert_33);
\node[vertex] (vert_00) at (-3,-1)   {$00$}
	edge [->,>=latex,out=-155,in=-115,loop,thick] node[left,pos=0.5]{$v$} (vert_00)
	edge [<-,>=latex,out=80,in=280,thick] node[right,pos=0.5]{$b$} (vert_20)
	edge [->,>=latex,out=100,in=260,thick] node[left,pos=0.5]{$a$} (vert_20);
\node[vertex] (vert_32) at (-1,-1)   {$32$}
	edge [->,>=latex,out=-110,in=-70,loop,thick] node[below,pos=0.5]{$w$} (vert_32)
	edge [->,>=latex,out=180,in=0,thick] node[above,pos=0.5]{$c$} (vert_00)
	edge [<-,>=latex,out=90,in=270,thick] node[right,pos=0.5]{$a$} (vert_10);
\node[vertex] (vert_11) at (1,-1)   {$11$}
	edge [->,>=latex,out=-110,in=-70,loop,thick] node[below,pos=0.5]{$v$} (vert_11)
	edge [<-,>=latex,out=180,in=0,thick] node[above,pos=0.5]{$b$} (vert_32)
	edge [->,>=latex,out=90,in=270,thick] node[right,pos=0.5]{$a$} (vert_33);
\node[vertex] (vert_21) at (3,-1)   {$21$}
	edge [->,>=latex,out=-25,in=-65,loop,thick] node[right,pos=0.5]{$w$} (vert_21)
	edge [<-,>=latex,out=80,in=280,thick] node[right,pos=0.5]{$a$} (vert_01)
	edge [->,>=latex,out=100,in=260,thick] node[left,pos=0.5]{$b$} (vert_01)
	edge [->,>=latex,out=180,in=0,thick] node[above,pos=0.5]{$c$} (vert_11);
\end{scope}

\begin{scope}[xshift=-0.5cm,yshift=-12cm]
\node[vertex] (vert_000) at (-6,1)   {$000$}
	edge [->,>=latex,out=155,in=115,loop,thick] node[left,pos=0.5]{$v$} (vert_000);
\node[vertex] (vert_320) at (-4,1)   {$320$}
	edge [->,>=latex,out=180,in=0,thick] node[above,pos=0.5]{$c$} (vert_000)
	edge [->,>=latex,out=110,in=70,loop,thick] node[above,pos=0.5]{$w$} (vert_320);
\node[vertex] (vert_110) at (-2,1)   {$110$}
	edge [<-,>=latex,out=180,in=0,thick] node[above,pos=0.5]{$b$} (vert_320)
	edge [->,>=latex,out=110,in=70,loop,thick] node[above,pos=0.5]{$v$} (vert_110);
\node[vertex] (vert_210) at (-0.5,2.5)   {$210$}
	edge [->,>=latex,out=110,in=70,loop,thick] node[above,pos=0.5]{$w$} (vert_210)
	edge [->,>=latex,out=225,in=45,thick] node[above,pos=0.5]{$c$} (vert_110);
\node[vertex] (vert_200) at (-6,-1)   {$200$}
	edge [->,>=latex,out=-155,in=-115,loop,thick] node[left,pos=0.5]{$w$} (vert_200)
	edge [->,>=latex,out=80,in=280,thick] node[right,pos=0.5]{$b$} (vert_000)
	edge [<-,>=latex,out=100,in=260,thick] node[left,pos=0.5]{$a$} (vert_000);
\node[vertex] (vert_100) at (-4,-1)   {$100$}
	edge [->,>=latex,out=-110,in=-70,loop,thick] node[below,pos=0.5]{$v$} (vert_100)
	edge [<-,>=latex,out=180,in=0,thick] node[above,pos=0.5]{$c$} (vert_200)
	edge [->,>=latex,out=90,in=270,thick] node[right,pos=0.5]{$a$} (vert_320);
\node[vertex] (vert_332) at (-2,-1)   {$332$}
	edge [->,>=latex,out=-110,in=-70,loop,thick] node[below,pos=0.5]{$w$} (vert_332)
	edge [->,>=latex,out=180,in=0,thick] node[above,pos=0.5]{$b$} (vert_100)
	edge [<-,>=latex,out=90,in=270,thick] node[right,pos=0.5]{$a$} (vert_110);
\node[vertex] (vert_011) at (-0.5,-2.5)   {$011$}
	edge [->,>=latex,out=-110,in=-70,loop,thick] node[below,pos=0.5]{$v$} (vert_011)
	edge [<-,>=latex,out=135,in=-45,thick] node[below,pos=0.5]{$c$} (vert_332);
\node[vertex] (vert_010) at (1.5,2.5)   {$010$}
	edge [->,>=latex,out=170,in=10,thick] node[above,pos=0.5]{$a$} (vert_210)
	edge [<-,>=latex,out=190,in=-10,thick] node[below,pos=0.5]{$b$} (vert_210)
	edge [->,>=latex,out=70,in=110,loop,thick] node[above,pos=0.5]{$v$} (vert_010);
\node[vertex] (vert_333) at (3,1)   {$333$}
	edge [->,>=latex,out=135,in=-45,thick] node[above,pos=0.5]{$c$} (vert_010)
	edge [->,>=latex,in=110,out=70,loop,thick] node[above,pos=0.5]{$w$} (vert_333);
\node[vertex] (vert_101) at (5,1)   {$101$}
	edge [<-,>=latex,out=180,in=0,thick] node[above,pos=0.5]{$b$} (vert_333)
	edge [->,>=latex,in=110,out=70,loop,thick] node[above,pos=0.5]{$v$} (vert_101);
\node[vertex] (vert_201) at (7,1)   {$201$}
	edge [->,>=latex,in=25,out=65,loop,thick] node[right,pos=0.5]{$w$} (vert_201)
	edge [->,>=latex,out=180,in=0,thick] node[above,pos=0.5]{$c$} (vert_101);
\node[vertex] (vert_211) at (1.5,-2.5)   {$211$}
	edge [->,>=latex,out=170,in=10,thick] node[above,pos=0.5]{$b$} (vert_011)
	edge [<-,>=latex,out=190,in=-10,thick] node[below,pos=0.5]{$a$} (vert_011)
	edge [->,>=latex,out=-70,in=-110,loop,thick] node[below,pos=0.5]{$w$} (vert_211);
\node[vertex] (vert_111) at (3,-1)   {$111$}
	edge [->,>=latex,in=-110,out=-70,loop,thick] node[below,pos=0.5]{$v$} (vert_111)
	edge [<-,>=latex,out=-135,in=45,thick] node[below,pos=0.5]{$c$} (vert_211)
	edge [->,>=latex,out=90,in=270,thick] node[right,pos=0.5]{$a$} (vert_333);
\node[vertex] (vert_321) at (5,-1)   {$321$}
	edge [->,>=latex,in=-110,out=-70,loop,thick] node[below,pos=0.5]{$w$} (vert_321)
	edge [->,>=latex,out=180,in=0,thick] node[above,pos=0.5]{$b$} (vert_111)
	edge [<-,>=latex,out=90,in=270,thick] node[right,pos=0.5]{$a$} (vert_101);
\node[vertex] (vert_001) at (7,-1)   {$001$}
	edge [->,>=latex,out=-25,in=-65,loop,thick] node[right,pos=0.5]{$v$} (vert_001)
	edge [->,>=latex,out=80,in=280,thick] node[right,pos=0.5]{$a$} (vert_201)
	edge [<-,>=latex,out=100,in=260,thick] node[left,pos=0.5]{$b$} (vert_201)
	edge [<-,>=latex,out=180,in=0,thick] node[above,pos=0.5]{$c$} (vert_321);
\end{scope}

\begin{scope}[xshift=0cm,yshift=-23cm]
\node[vertex] (vert_0000) at (-6-5.5,1)   {$0000$}
	edge [->,>=latex,out=155,in=115,loop,thick] node[left,pos=0.5]{$v$} (vert_0000);
\node[vertex] (vert_3200) at (-4-5.5,1)   {$3200$}
	edge [->,>=latex,out=180,in=0,thick] node[above,pos=0.5]{$c$} (vert_0000)
	edge [->,>=latex,out=110,in=70,loop,thick] node[above,pos=0.5]{$w$} (vert_3200);
\node[vertex] (vert_1100) at (-2-5.5,1)   {$1100$}
	edge [<-,>=latex,out=180,in=0,thick] node[above,pos=0.5]{$b$} (vert_3200)
	edge [->,>=latex,out=110,in=70,loop,thick] node[above,pos=0.5]{$v$} (vert_1100);
\node[vertex] (vert_2100) at (-0.5-5.5,2.5)   {$2100$}
	edge [->,>=latex,out=110,in=70,loop,thick] node[above,pos=0.5]{$w$} (vert_2100)
	edge [->,>=latex,out=225,in=45,thick] node[above,pos=0.5]{$c$} (vert_1100);
\node[vertex] (vert_2000) at (-6-5.5,-1)   {$2000$}
	edge [->,>=latex,out=-155,in=-115,loop,thick] node[left,pos=0.5]{$w$} (vert_2000)
	edge [->,>=latex,out=80,in=280,thick] node[right,pos=0.5]{$b$} (vert_0000)
	edge [<-,>=latex,out=100,in=260,thick] node[left,pos=0.5]{$a$} (vert_0000);
\node[vertex] (vert_1000) at (-4-5.5,-1)   {$1000$}
	edge [->,>=latex,out=-110,in=-70,loop,thick] node[below,pos=0.5]{$v$} (vert_1000)
	edge [<-,>=latex,out=180,in=0,thick] node[above,pos=0.5]{$c$} (vert_2000)
	edge [->,>=latex,out=90,in=270,thick] node[right,pos=0.5]{$a$} (vert_3200);
\node[vertex] (vert_3320) at (-2-5.5,-1)   {$3320$}
	edge [->,>=latex,out=-110,in=-70,loop,thick] node[below,pos=0.5]{$w$} (vert_3320)
	edge [->,>=latex,out=180,in=0,thick] node[above,pos=0.5]{$b$} (vert_1000)
	edge [<-,>=latex,out=90,in=270,thick] node[right,pos=0.5]{$a$} (vert_1100);
\node[vertex] (vert_0110) at (-0.5-5.5,-2.5)   {$0110$}
	edge [->,>=latex,out=-110,in=-70,loop,thick] node[below,pos=0.5]{$v$} (vert_0110)
	edge [<-,>=latex,out=135,in=-45,thick] node[below,pos=0.5]{$c$} (vert_3320);
\node[vertex] (vert_0100) at (1.5-5.5,2.5)   {$0100$}
	edge [->,>=latex,out=170,in=10,thick] node[above,pos=0.5]{$a$} (vert_2100)
	edge [<-,>=latex,out=190,in=-10,thick] node[below,pos=0.5]{$b$} (vert_2100)
	edge [->,>=latex,out=70,in=110,loop,thick] node[above,pos=0.5]{$v$} (vert_0100);
\node[vertex] (vert_3332) at (3-5.5,1)   {$3332$}
	edge [->,>=latex,out=135,in=-45,thick] node[above,pos=0.5]{$c$} (vert_0100)
	edge [->,>=latex,in=110,out=70,loop,thick] node[above,pos=0.5]{$w$} (vert_3332);
\node[vertex] (vert_2110) at (1.5-5.5,-2.5)   {$2110$}
	edge [->,>=latex,out=170,in=10,thick] node[above,pos=0.5]{$b$} (vert_0110)
	edge [<-,>=latex,out=190,in=-10,thick] node[below,pos=0.5]{$a$} (vert_0110)
	edge [->,>=latex,out=-70,in=-110,loop,thick] node[below,pos=0.5]{$w$} (vert_2110);
\node[vertex] (vert_1110) at (3-5.5,-1)   {$1110$}
	edge [->,>=latex,in=-110,out=-70,loop,thick] node[below,pos=0.5]{$v$} (vert_1110)
	edge [<-,>=latex,out=-135,in=45,thick] node[below,pos=0.5]{$c$} (vert_2110)
	edge [->,>=latex,out=90,in=270,thick] node[right,pos=0.5]{$a$} (vert_3332);
\node[vertex] (vert_1011) at (-1,2.5)   {$1011$}
	edge [->,>=latex,out=155,in=115,loop,thick] node[above,pos=0.5]{$v$} (vert_1011)
	edge [<-,>=latex,out=225,in=45,thick] node[above,pos=0.5]{$b$} (vert_3332);
\node[vertex] (vert_3211) at (1,2.5)   {$3211$}
	edge [<-,>=latex,out=180,in=0,thick] node[above,pos=0.5]{$a$} (vert_1011)
	edge [->,>=latex,out=25,in=65,loop,thick] node[above,pos=0.5]{$w$} (vert_3211);
\node[vertex] (vert_2011) at (-1,4.5)   {$2011$}
	edge [->,>=latex,out=110,in=70,loop,thick] node[above,pos=0.5]{$w$} (vert_2011)
	edge [->,>=latex,out=270,in=90,thick] node[right,pos=0.5]{$c$} (vert_1011);
\node[vertex] (vert_0011) at (1,4.5)   {$0011$}
	edge [<-,>=latex,out=270,in=90,thick] node[right,pos=0.5]{$c$} (vert_3211)
	edge [->,>=latex,out=170,in=10,thick] node[above,pos=0.5]{$a$} (vert_2011)
	edge [<-,>=latex,out=190,in=-10,thick] node[below,pos=0.5]{$b$} (vert_2011)
	edge [->,>=latex,out=70,in=110,loop,thick] node[above,pos=0.5]{$w$} (vert_0011);
\node[vertex] (vert_3210) at (-1,-2.5)   {$3210$}
	edge [->,>=latex,in=155+90,out=115+90,loop,thick] node[below,pos=0.5]{$w$} (vert_3210)
	edge [->,>=latex,out=-225,in=-45,thick] node[above,pos=0.5]{$b$} (vert_1110);
\node[vertex] (vert_1010) at (1,-2.5)   {$1010$}
	edge [->,>=latex,out=180,in=0,thick] node[above,pos=0.5]{$a$} (vert_3210)
	edge [->,>=latex,in=25-90,out=65-90,loop,thick] node[below,pos=0.5]{$v$} (vert_1010);
\node[vertex] (vert_0010) at (-1,-4.5)   {$0010$}
	edge [->,>=latex,in=110+180,out=70+180,loop,thick] node[below,pos=0.5]{$v$} (vert_0010)
	edge [<-,>=latex,out=90,in=270,thick] node[right,pos=0.5]{$c$} (vert_3210);
\node[vertex] (vert_2010) at (1,-4.5)   {$2010$}
	edge [->,>=latex,out=90,in=270,thick] node[right,pos=0.5]{$c$} (vert_1010)
	edge [->,>=latex,out=170,in=10,thick] node[above,pos=0.5]{$b$} (vert_0010)
	edge [<-,>=latex,out=190,in=-10,thick] node[below,pos=0.5]{$a$} (vert_0010)
	edge [->,>=latex,in=70+180,out=110+180,loop,thick] node[below,pos=0.5]{$w$} (vert_2010);
\node[vertex] (vert_1111) at (-2+4.5,1)   {$1111$}
	edge [<-,>=latex,out=-225,in=-45,thick] node[above,pos=0.5]{$b$} (vert_3211)
	edge [->,>=latex,out=110,in=70,loop,thick] node[above,pos=0.5]{$v$} (vert_1111);
\node[vertex] (vert_2111) at (-0.5+4.5,2.5)   {$2111$}
	edge [->,>=latex,out=110,in=70,loop,thick] node[above,pos=0.5]{$w$} (vert_2111)
	edge [->,>=latex,out=225,in=45,thick] node[above,pos=0.5]{$c$} (vert_1111);
\node[vertex] (vert_3333) at (-2+4.5,-1)   {$3333$}
	edge [->,>=latex,out=225,in=45,thick] node[above,pos=0.5]{$b$} (vert_1010)
	edge [->,>=latex,out=-110,in=-70,loop,thick] node[below,pos=0.5]{$w$} (vert_3333)
	edge [<-,>=latex,out=90,in=270,thick] node[right,pos=0.5]{$a$} (vert_1111);
\node[vertex] (vert_0101) at (-0.5+4.5,-2.5)   {$0101$}
	edge [->,>=latex,out=-110,in=-70,loop,thick] node[below,pos=0.5]{$v$} (vert_0101)
	edge [<-,>=latex,out=135,in=-45,thick] node[below,pos=0.5]{$c$} (vert_3333);
\node[vertex] (vert_0111) at (1.5+4.5,2.5)   {$0111$}
	edge [->,>=latex,out=170,in=10,thick] node[above,pos=0.5]{$a$} (vert_2111)
	edge [<-,>=latex,out=190,in=-10,thick] node[below,pos=0.5]{$b$} (vert_2111)
	edge [->,>=latex,out=70,in=110,loop,thick] node[above,pos=0.5]{$v$} (vert_0111);
\node[vertex] (vert_3321) at (3+4.5,1)   {$3321$}
	edge [->,>=latex,out=135,in=-45,thick] node[above,pos=0.5]{$c$} (vert_0111)
	edge [->,>=latex,in=110,out=70,loop,thick] node[above,pos=0.5]{$w$} (vert_3321);
\node[vertex] (vert_1001) at (5+4.5,1)   {$1001$}
	edge [<-,>=latex,out=180,in=0,thick] node[above,pos=0.5]{$b$} (vert_3321)
	edge [->,>=latex,in=110,out=70,loop,thick] node[above,pos=0.5]{$v$} (vert_1001);
\node[vertex] (vert_2001) at (7+4.5,1)   {$2001$}
	edge [->,>=latex,in=25,out=65,loop,thick] node[right,pos=0.5]{$w$} (vert_2001)
	edge [->,>=latex,out=180,in=0,thick] node[above,pos=0.5]{$c$} (vert_1001);
\node[vertex] (vert_2101) at (1.5+4.5,-2.5)   {$2101$}
	edge [->,>=latex,out=170,in=10,thick] node[above,pos=0.5]{$b$} (vert_0101)
	edge [<-,>=latex,out=190,in=-10,thick] node[below,pos=0.5]{$a$} (vert_0101)
	edge [->,>=latex,out=-70,in=-110,loop,thick] node[below,pos=0.5]{$w$} (vert_2101);
\node[vertex] (vert_1101) at (3+4.5,-1)   {$1101$}
	edge [->,>=latex,in=-110,out=-70,loop,thick] node[below,pos=0.5]{$v$} (vert_1101)
	edge [<-,>=latex,out=-135,in=45,thick] node[below,pos=0.5]{$c$} (vert_2101)
	edge [->,>=latex,out=90,in=270,thick] node[right,pos=0.5]{$a$} (vert_3321);
\node[vertex] (vert_3201) at (5+4.5,-1)   {$3201$}
	edge [->,>=latex,in=-110,out=-70,loop,thick] node[below,pos=0.5]{$w$} (vert_3201)
	edge [->,>=latex,out=180,in=0,thick] node[above,pos=0.5]{$b$} (vert_1101)
	edge [<-,>=latex,out=90,in=270,thick] node[right,pos=0.5]{$a$} (vert_1001);
\node[vertex] (vert_0001) at (7+4.5,-1)   {$0001$}
	edge [->,>=latex,out=-25,in=-65,loop,thick] node[right,pos=0.5]{$v$} (vert_0001)
	edge [->,>=latex,out=80,in=280,thick] node[right,pos=0.5]{$a$} (vert_2001)
	edge [<-,>=latex,out=100,in=260,thick] node[left,pos=0.5]{$b$} (vert_2001)
	edge [<-,>=latex,out=180,in=0,thick] node[above,pos=0.5]{$c$} (vert_3201);
\end{scope}
\end{tikzpicture}}
\caption{The first four Schreier graphs of Example~\ref{basilica type example}.}
\label{Schreier_basilicatype}
\end{figure}

\begin{exm}
\label{basilica type example}
Consider the graph $E$ in Figure~\ref{basilica type example},
and define a self-similar groupoid through the $E$-automaton
\begin{equation}\label{defbasilicatype}
\begin{split}
a \cdot 0=2 \quad a|_{0}=v \qquad b\cdot 2=0 \quad b|_2=v \qquad c\cdot 2=1 \quad c|_2=v \\
a \cdot 1=3 \quad a|_1=a \qquad b\cdot 3=1 \quad b|_3=c \qquad c\cdot 3=0 \quad c|_3=b.
\end{split}
\end{equation}
We claim that this action is contracting with nucleus
\[
\Nn=\{v,w,a,b,c,ba,ca,a^{-1},b^{-1},c^{-1},(ba)^{-1},(ca)^{-1}\}.
\]
To see this we note that all elements of the automaton appear as
restrictions, so $v,w,a,b,c$ and their inverses are in the nucleus. That $ba$
and $ca$ are in the nucleus follow from the computations
\[
(ba)|_{1}=b|_{a \cdot 1} a|_{1}=ca \quad \text{ and } \quad (ca)|_{1}=c|_{a\cdot 1}a|_{1}=ba.
\]
One can now compute that all groupoid elements of length $3$ reduce to one of the elements of the nucleus after restriction to length $2$ words.

The first four Schreier graphs are presented in
Figure~\ref{Schreier_basilicatype}. They suggest that the action is
level-transitive and that the limit-space is homeomorphic to the basilica
fractal (the Julia set of $z^2-1$); one can prove this via an argument of the
sort outlined in Example~\ref{s,r clarifying example_schreier}.
\end{exm}

\section{Dynamics on the limit space}\label{sec:J dynamics}

In this section we describe an action of $\NN$ by locally expansive local homeomorphisms of the
limit space $\Jj$ of a contracting, regular self-similar groupoid action. We will use
this in the next section to construct a Smale space from the self-similar groupoid action.

Let $E$ be a finite directed graph. The \emph{shift map} $\sigma : E^{-\infty} \to E^{-\infty}$ is
defined by $\sigma(\dots x_{-3} x_{-2} x_{-1}) = \dots x_{-3} x_{-2}$; that is, $\sigma$ deletes
the right-most edge of a left-infinite path. This $\sigma$ is a local homeomorphism because it
restricts to a homeomorphism $\sigma : \lZ{\mu e} \to \lZ{\mu}$ for any finite path $\mu$ and any
edge $e$ such that $s(\mu) = r(e)$. The main result in this section is about self-similar groupoid
actions that are \emph{regular} in the following sense, which is based on the regularity condition
used by Nekrashevych in \cite{Nekrashevych:Cstar_selfsimilar}.

\begin{dfn}[cf. {\cite[Definition 6.1]{Nekrashevych:Cstar_selfsimilar}}]\label{def:regular}
Let $E$ be a finite directed graph. Let $(G, E)$ be a self-similar groupoid
action. We say that $(G, E)$ is \emph{regular} if for every $g  \in G$ and
every $y  \in E^\infty$ such that $g \cdot y = y$, there exists $\mu $ in
$E^*$ such that $y \in \rZ{\mu}$, $g \cdot \mu = \mu$ and $g|_\mu = s(\mu)$.
\end{dfn}

\begin{rmk}
Since, by definition, self-similar groupoid actions are faithful, the
regularity condition is equivalent to the condition that if $y \in E^\infty$
and $g \cdot y = y$, then there is a clopen neighbourhood of $y$ that is
pointwise fixed by $g$.
\end{rmk}

Our main theorem in this section says that for contracting, regular
self-similar groupoid actions, the shift map induces a locally expanding
local homeomorphism of $\Jj$.

\begin{thm}\label{thm:Jsig}
Let $E$ be a finite directed graph with no sources. Let $(G, E)$ be a
contracting, regular self-similar groupoid action with limit space $\Jj$ as
in Definition~\ref{dfn:limit space}. Let $\sigma : E^{-\infty} \to
E^{-\infty}$ be the shift map. Then there is a surjective map $\Jsig : \Jj
\to \Jj$ such that $\Jsig([x]) = [\sigma(x)]$ for all $x  \in E^{-\infty}$.
Furthermore, there exist $\varepsilon > 0$, $c>1$ and a metric $d_\Jj$ on $\Jj$ such that
\begin{enumerate}
\item whenever $d_\Jj([x],
[y]) < \varepsilon$, we have $d_\Jj(\Jsig([x]), \Jsig([y])) = c\cdot d_\Jj([x], [y])$, and
\item whenever $\alpha \leq\varepsilon$, we have $\tilde{\sigma}(B([x],\alpha)) = B(\tilde{\sigma}([x]),c\alpha).$
\end{enumerate}
In particular, $\Jsig$ is a locally
expanding local homeomorphism.
\end{thm}

Before proving the theorem, we need to establish some preliminary results. To
get started, observe that if $x \sim_\aeq y$, then there is a sequence
$(g_n)_{n < 0} \in \Nn$ such that $g_n \cdot x_n \dots x_{-1} = y_n \dots
y_{-1}$ for all $n$, and it follows that $g_{n-1} \cdot \sigma(x)_n \dots
\sigma(x)_{-1} = g_{n-1} \cdot x_{n-1} \dots x_{-2} = y_{n-1} \dots y_{-2} =
\sigma(y)_n \dots \sigma(y)_{-1}$. That is,
\begin{equation}\label{eq:Jsig welldef}
    x \sim_\aeq y \Longrightarrow \sigma(x) \sim_\aeq \sigma(y).
\end{equation}
Therefore, there exists a map $\tilde{\sigma}:\mathcal{J}\mapsto\mathcal{J}$
as described in Theorem~\ref{thm:Jsig}.

\begin{lem}\label{lem:magic k}
Let $E$ be a finite directed graph with no sources. Let $(G, E)$ be a regular
self-similar groupoid action. For any finite set $F\subseteq G$, there exists
$k  \in \NN$ such that for all $g,h  \in F$ such that $d(g) = d(h)$ and all
$\mu  \in d(g)E^*$ with $|\mu| \ge k$, if $g \cdot \mu = h \cdot \mu$, then
$g|_\mu = h|_\mu$.
\end{lem}
\begin{proof}
Fix $g,h  \in F$.

For each $y  \in E^\infty$ satisfies $g\cdot y = h \cdot y$, we have $h^{-1}g
\cdot y = y$, and so regularity implies that there exists $\lambda_y \in E^*$
such that $y \in \rZ{\lambda_y}$ and $(h^{-1}g)|_{\lambda_y} = s(\lambda_y)$.
For each $x \in E^\infty$ such that $g \cdot x \not= h \cdot x$, we have
$h^{-1} g \cdot x \not= x$, and so there exists $\lambda_x \in E^*$ such that
$x \in \rZ{\lambda_x}$ and $(h^{-1} g) \cdot \lambda_x \not= \lambda_x$.

Since $E^\infty = \bigcup_{x \in E^\infty} \rZ{\lambda_x}$ and since
$E^\infty$ is compact, there exists a finite $K \subseteq E^\infty$ such that
$E^\infty = \bigcup_{x \in K} \rZ{\lambda_x}$. Let $k_{g,h} :=
\max\{|\lambda_x| : x \in K\}$. Suppose that $\mu \in E^*$ with $|\mu| \geq
k_{g,h}$ and that $g \cdot \mu = h \cdot \mu$. Since $E$ has no sources we
have $Z[\mu)\neq\emptyset$. Since the $\rZ{\lambda_x}$ cover $E^\infty$ we
have $\rZ{\mu} \cap \rZ{\lambda_x} \not= \varnothing$ for some $x  \in K$.
Since $|\mu| \ge |\lambda_x|$, it follows that $\mu = \lambda_x \mu'$ for
some $\mu' $ in $E^\infty$. Since $g \cdot \mu = h \cdot \mu$, we have $g
\cdot \lambda_x = h \cdot \lambda_x$ and therefore $h^{-1} g \cdot \lambda_x
= \lambda_x$. By the choice of $\lambda_x$, we have $h^{-1} g \cdot x = x$ and
$(h^{-1} g)|_{\lambda_x} = s(\lambda_x)$. Hence $(h^{-1} g)|_\mu = (h^{-1}
g)|_{\lambda_x \mu'} = s(\lambda_x)|_{\mu'} = s(\mu') = s(\mu)$. Hence
$g|_\mu = h|_\mu$.

We have now proved that for each $g,h  \in F$ with $d(g) = d(h)$, there
exists $k_{g,h}  \in \NN$ such that whenever $\mu  \in d(g) E^{k_{g,h}}$
satisfies $g\cdot \mu = h \cdot \mu$, we have $g|_\mu = h|_\mu$. So $k =
\max_{g,h \in F} k_{g,h}$ has the required property.
\end{proof}

Our next result is essentially a version of Theorem~\ref{thm:Jsig} in which
the metric balls and $\varepsilon$-approximations are replaced by conditions
in terms of the basic open sets from Corollary~\ref{cor:limit space basis}.
We will bootstrap from this result to prove Theorem~\ref{thm:Jsig}.

\begin{prp}\label{prp:mapping_properties}
Let $E$ be a finite directed graph with no sources. If $(G,E)$ is a contracting and regular self-similar group action, then
\begin{enumerate}
\item for each $z \in \mathcal{J},$ $\sigma$ maps $q^{-1}(z)$ bijectively
    onto $q^{-1}(\tilde{\sigma}(z))$;
\item there exists $k \in \mathbb{N}$ such that for every $n\geq k+1$ and
    every $\mu \in E^{n}$, the map $\tilde{\sigma}$ restricts to a
    bijection of $U_{\mu}$ onto $U_{\sigma(\mu)}$; and
\item for every $n\geq k$, every $\omega \in E^{n}$, every $w \in
    U_{\omega}$, and every $z \in \tilde{\sigma}^{-1}(w)$, there exists
    $\mu \in E^{n+1}$ such that $z \in U_{\mu}$ and $\sigma(\mu) = \omega$.
\end{enumerate}
In particular, $\tilde{\sigma}$ is a local homeomorphism.
\end{prp}
\begin{proof}
Applying Lemma~\ref{lem:magic k} to the finite set $F =
\mathcal{N}^{2}\cup\mathcal{N}\cup E^{0}$ yields $k \in \mathbb{N}$ such that
for all $n_{1}, n_{2} \in F$, if $\mu \in E^*$ with $|\mu| \ge k$ satisfies
$n_{1}\cdot \mu = n_{2}\cdot \mu$, then $n_{1}|_\mu = n_{2}|_\mu$. We fix $k$
with this property for the remainder of the proof.

(1) Since $q \circ \sigma = \tilde\sigma\circ q$, if $q(x) = z$ then
$q(\sigma(x)) = \tilde\sigma(q(x)) = \tilde\sigma(z)$, and so
$\sigma(q^{-1}(z))\subseteq q^{-1}(\tilde{\sigma}(z))$. So we must prove the
reverse inclusion. Suppose that $q(x) = z$ and $y'\sim_{\aeq} \sigma(x)$. Let
$(g_{n})_{n\in\mathbb{N}}$ be a sequence in $\Nn$ such that $d(g_{n}) =
r(x_{-n-1})$ and $g_{n}\cdot x_{-n-1}x_{-n}...x_{-2} =
y'_{-n}y'_{-n+1}...y'_{-1}$ for every $n \in \mathbb{N}$. By the choice of $k$,
for all $n, n'\geq k$, it follows that $g_{n}|_{x_{-n-1}x_{-n}...x_{-2}} =
g_{n'}|_{x_{-n'-1}x_{-n'}...x_{-2}}$. Let $x'_{-1} =
(g_{k}|_{x_{-k-1}x_{-k}...x_{-2}})\cdot x_{-1}$ and $x' = y'x'_{-1}$, and let
$g'_{n} = g_{k}|_{x_{-k-1}x_{-k}...x_{-n-1}}$ for $n < k$, and $g'_{n} =
g_{n-1}$ for $n\geq k$. Then $g'_{n}\cdot x_{-n}x_{-n+1}...x_{-1} =
x'_{-n}x'_{-n+1}...x'_{-1}$ for every $n \in \mathbb{N}$. So, $x\sim_{\aeq}
x'$ and $\sigma(x') = y'$. Therefore, $\sigma(q^{-1}(z)) =
q^{-1}(\tilde{\sigma}(z))$.

Suppose that $x', x'' \in q^{-1}(z)$ satisfy $\sigma(x'') = \sigma(x') = y'$.
Since $x'\sim_{\aeq} x\sim_{\aeq} x''$, there exists $g$ in $\mathcal{N}$
such that $g\cdot x'_{-k-1}x'_{-k}...x'_{-1} =
x''_{-k-1}x''_{-k}...x''_{-1}$. By the choice of $k$ and since
$x'_{-k-1}x'_{-k}...x'_{-2} = x''_{-k-1}x''_{-k}...x''_{-2}$, we have
$g|_{x'_{-k-1}x'_{-k}...x'_{-2}} = s(x'_{-2})$. Hence $x''_{-1} =
(g|_{x'_{-k-1}x'_{-k}...x'_{-2}})\cdot x'_{-1} = x_{-1}$. Therefore, $x' =
x''$.

(2) Fix $\mu \in E^*$ with $|\mu| > k$ and $w \in U_{\sigma(\mu)}$. For each
$y \in q^{-1}(w)$, there exists $g \in \Nn$ such that $y \in \rZ{g\cdot
\sigma(\mu)}$. By the choice of $k$, the element $g|_{\sigma(\mu)}$ does not
depend on the choice of $g$. So there is a unique map $\delta :
q^{-1}(U_{\sigma(\mu)}) \to E^{-\infty}$ such that $\delta(y) = y
\big((g|_{\sigma(\mu)}) \cdot \mu_{-1}\big)$ for any $g \in \Nn$ such that $y
\in \rZ{g\cdot \sigma(\mu)}$. We claim that $\delta$ descends to a map
$\tilde\delta : U_{\sigma(\mu)} \to U_\mu$ that is an inverse for
$\tilde\sigma|_{U_\mu}$.

For this, suppose that $y, y' \in q^{-1}(w)$ are asymptotically equivalent.
Fix $g \in \Nn$ such that $y \in \lZ{g\cdot \sigma(\mu)}$ and $g' \in \Nn$
such that $y' \in \lZ{g'\cdot \sigma(\mu)}$. Then $x := \delta(y)$ satisfies
$x = y\big((g|_{\sigma(\mu)})\cdot \mu_{-1}) \in \lZ{g \cdot \mu}$. By~(1)
there is a unique element $x' \in E^{-\infty}$ such that $x' \sim x$ and
$\sigma(x') = y'$. Since $x \sim_{\aeq} x'$, there exists $h \in \Nn$ such
that $h \cdot(x_{-k-1}\cdots x_{-1}) = x'_{-k-1} x'_{-k} \cdots x'_{-1}$.
Hence $(hg)\cdot\sigma(\mu) = x'_{-k-1} x'_{-k} \cdots x'_{-2} = g' \cdot
\sigma(\mu)$. Since $hg, g' \in F$, by the choice of $k$ we have
$(hg)|_{\sigma(\mu)} = g'|_{\sigma(\mu)}$ and we deduce that $x'_{-1} =
\big((hg)|_{\sigma(\mu)}\big)\cdot \mu_{-1} =
\big(g'|_{\sigma(\mu)}\big)\cdot \mu_{-1}$. By definition, we have
$\delta(y') = y'\big(g'|_{\sigma(\mu)}\big)\cdot \mu_{-1} = y' x'_{-1}$, and
since $\sigma(x') = y'$ by definition of $x'$, we deduce that $\delta(y') =
\sigma(x') x'_{-1} = x'$. So $\delta(y') = x' \sim_{\aeq} x= \delta(y)$, and
it follows that $\delta$ descends to a map $\tilde\delta : U_{\sigma(\mu)}
\to \Jj$.

To see that $\tilde\delta(U_{\sigma(\mu)}) \subseteq U_\mu$, fix $y \in
q^{-1}(w)$ and let $x = \delta(y)$. We must show that $[x] \subseteq
\bigcup_{\nu \in \Nn \cdot \mu} \lZ{\nu}$. Fix $x' \in [x]$, and let $y' =
\sigma(x)$. Applying the argument of the preceding paragraph we obtain $x' =
\delta(y') \in \lZ(g'\cdot \mu)$ for some $g' \in \Nn$ as required.

It remains to show that $\tilde\delta$ is an inverse for
$\tilde\sigma|_{U_\mu}$. By construction, $\sigma\delta(y) = y$ for $y \in
q^{-1}(U_{\sigma(\mu)})$, so $\tilde{\sigma}\tilde{\delta}(q(y)) = q(y)$.
Therefore, $\tilde{\sigma}\tilde{\delta} = \id_{U_{\sigma(\mu)}}$.

We now show that $\tilde{\delta}\tilde{\sigma} = id_{U_{\mu}}$. If
$q^{-1}(z)\subseteq \bigcup_{g\in\mathcal{N}:d(g) = r(\mu)} Z[g\cdot \mu)$,
then $q^{-1}(\tilde{\sigma} (z)) = \sigma(q^{-1}(z))\subseteq
\sigma(\bigcup_{g\in\mathcal{N}:d(g) = r(\mu)} Z[g\cdot \mu)) =
\bigcup_{g\in\mathcal{N}:d(g) = r(\sigma(\mu))} Z[g\cdot \sigma(\mu))$.
Hence, $\tilde{\sigma}(U_{\mu})\subseteq U_{\sigma(\mu)}$, so the composites
$\tilde{\delta}\tilde{\sigma}$ and $\delta\sigma$ are well defined on
$U_{\mu}$ and $q^{-1}(U_{\mu})$ respectively. We have $\delta(y) =
y(g|_{\sigma(\mu)})\cdot\mu_{-1}$ for $y \in q^{-1}(U_{\sigma(\mu)})\cap
Z[g\cdot \sigma(\mu))$, and so $\delta\sigma(x) = x$ for $x$ in
$q^{-1}(U_{\mu})\cap Z[g\cdot \mu)$. Hence,
$\tilde{\delta}\tilde{\sigma}(q(x)) = q(x)$ for $x \in q^{-1}(U_{\mu})$.
Therefore, $\tilde{\delta}\tilde{\sigma} = \id_{U_{\mu}}$, and we have shown
$\tilde{\sigma}$ maps $U_{\mu}$ bijectively onto $U_{\sigma(\mu)}$.

(3) Fix $\omega \in E^*$ with $|\omega| \ge k$ and $w \in U_{\omega}$, and
fix $z \in \tilde{\sigma}^{-1}(w)$. Choose $x,y \in E^{-\infty}$ such that
$q(x) = z$, $q(y) = w$, and $\sigma(x) = y$. Then there exists $g \in G
r(\omega)$ such that $y \in Z[g\cdot \omega)$. Since $\sigma(x) = y$, it
follows that $x_{-n-1}x_{-n}...x_{-1} = (g\cdot\omega)x_{-1}$. We have
$d((g|_{\omega})^{-1}) = c(g|_{\omega}) = s(g\cdot\omega) = r(x_{-1})$. Let
$f = (g|_{\omega})^{-1}\cdot x_{-1}$ and $\mu = \omega f$. Then $x \in
Z(g\cdot\mu)$ and $\sigma(\mu) = \omega$, and the map $\delta :
q^{-1}(U_{\sigma(\mu)}) \to q^{-1}(U_{\mu})$ constructed in the proof of~(2)
satisfies $\delta(y) = x$. Hence $z = q(x) \in U_{\mu}$.
\end{proof}
We now prove the first part of our main result.
\begin{proof}[Proof of Theorem~\ref{thm:Jsig}(1)]
By Corollary~\ref{cor:limit space metrisable}, $\Jj$ is metrisable. Fix a metric $d$ on $\Jj$. By Lemma~\ref{lem:magic k}, there is $k\in\mathbb{N}$ such that for all $n_{1},n_{2}\in F = \Nn^{2}\cup \Nn\cup E^{0}$, if $\mu\in E^{*}$ with $|\mu|\geq k$ satisfies $n_{1}\cdot \mu = n_{2}\cdot\mu$, then $n_{1}|_{\mu} = n_{2}|_{\mu}$. By the Lebesgue Number Lemma, there exists $\delta > 0$ such that every ball of radius $\delta$ in $\Jj$ is contained in $U_\omega$ for some $\omega \in E^{k+2}$. We first prove $\tilde{\sigma}$ is expanding in the sense that if $[x], [y]\in \Jj$ satisfies $d(\tilde{\sigma}^{n}([x]), \tilde{\sigma}^{n}([y])) < \delta$ for all $n\in\mathbb{N}_{0}$, then $[x] = [y]$.
\par
For each $n\in \mathbb{N}_{0}$, since $d(\tilde{\sigma}^{n}([x]), \tilde{\sigma}^{n}([y])) <\delta$, there is $\omega_{n}\in E^{k+2}$ such that $\tilde{\sigma}^{n}([x]), \tilde{\sigma}^{n}([y])\in U_{\omega_{n}}$. Therefore, there are $h_{n,1}, h_{n,2}\in\Nn$ such that $d(h_{n,1})= d(h_{n,2}) = r(\omega_{n})$, $h_{n,1}\cdot\omega_{n} = x_{-n-k-2}...x_{-n-1}$ and $h_{n,2}\cdot\omega_{n} = y_{-n-k-2}...y_{-n-1}$. Hence, $g_{n}:= h_{n,2}h_{n,1}^{-1}$ satisfies $d(g_{n}) = r(x_{-n-k-2})$ and $g_{n}\cdot x_{-n-k-2}...x_{-n-1} = y_{-n-k-2}...y_{-n-1}$. We claim $g_{n}\cdot  x_{-n-k-2}...x_{-1} = y_{-n-k-2}...y_{-1}$ for all $n\in\mathbb{N}_{0}$.
\par
This is trivially true for $n=0$. Suppose we have $g_{n}\cdot x_{-n-k-2}...x_{-1} = y_{-n-k-2}...y_{-1}$ for some $n\geq 0$. Since $g_{n},(g_{n+1}|_{x_{-n-k-3}})\in F$, $g_{n}\cdot x_{-n-k-2}...x_{-n-2} = y_{-n-k-2}...y_{-n-2} = (g_{n+1}|_{x_{-n-k-3}})\cdot x_{-n-k-2}...x_{-n-2}$ and $|x_{-n-k-2}...x_{-n-2}| = k$, it follows from Lemma~\ref{lem:magic k} that
$g_{n+1}|_{x_{-n-k-3}...x_{-n-2}} = g_{n}|_{x_{-n-k-2}...x_{-n-2}}$. Therefore, $g_{n+1}\cdot x_{-n-k-3}...x_{-1} = \\y_{-n-k-3}...y_{-1}$. This proves the claim by induction.
\par
So, if we set $h_{n} = g_{n - k - 2}$ for $n\geq k+2$ and $h_{n} = g_{0}|_{x_{-k-2}...x_{-n-1}}$ for $0< n < k+2$, then $(h_{n})_{n\in\mathbb{N}}\subseteq F$ satisfies $h_{n}\cdot x_{-n}...x_{-1} = y_{-n}...y_{-1}$ for all $n\in\mathbb{N}$. Therefore, $[x] = [y]$, as claimed.
\par
It follows that if $[x]\neq [y]$, there is $n\in\mathbb{N}_{0}$ such that $d(\tilde{\sigma}^{n}([x]), \tilde{\sigma}^{n})([y])) >\delta$. The map $\tilde{\sigma}$ is open by Proposition~\ref{prp:mapping_properties} and, since $E$ has no sources, $\tilde{\sigma}$ is also surjective. Therefore, $\tilde{\sigma}$ is Ruelle expanding (see \cite{FKM}). Hence, by \cite[Theorem~2.2]{FKM}, there is $\epsilon > 0$, $c > 1$ and a metric $d_{\Jj}$ on $\Jj$ such that whenever $d_{\Jj}([x], [y]) <\epsilon$, we have $d_{\Jj}(\tilde{\sigma}([x]),\tilde{\sigma}([y])) = c\cdot d_{\Jj}([x],[y])$.
\end{proof}
We now fix a metric $d_{\Jj}$ satisfying Condition~(1) in Theorem~\ref{thm:Jsig}. To establish Condition~(2), we will use the
following technical lemma, which we will need again in the proof of
Lemma~\ref{lem:unstable_description}.

\begin{lem}\label{lem:delta_property}
Resume the hypotheses of Theorem~\ref{thm:Jsig}. Suppose that $\varepsilon > 0$ and $c > 1$ satisfy statement~(1) of that theorem. Let $k$ be as in Proposition~\ref{prp:mapping_properties}(2). Then there exist $\eta < \varepsilon$
and $n>k$ such that
\begin{enumerate}
    \item[(1)] for every $w \in \mathcal{J}$, and $\alpha\leq\eta$, there
        exists $\omega \in E^*$ such that $|\omega|\geq n - 1$ and
        $B(w,c\alpha)\subseteq U_{\omega}$; and
    \item[(2)] for every $\mu \in E^*$ with $|\mu|\geq n$, the map
        $\tilde{\sigma}$ restricts to a homeomorphism of $U_{\mu}$ onto
        $U_{\sigma(\mu)}$; and if $\tilde{\delta}$ denotes its inverse,
        then for all $z \in U_{\mu}$ and all $\alpha\leq\eta$ such that
        $B(\tilde{\sigma}(z),c\alpha)\subseteq U_{\sigma(\mu)}$, we have
        $B(z,\alpha)\subseteq U_{\mu}$, and $\tilde{\delta}$ restricts to a
        homeomorphism of $B(\tilde{\sigma}(z),c\alpha)$ onto $B(z,\alpha)$.
\end{enumerate}
\end{lem}
\begin{proof}
For $\mu, \omega \in E^*$ such that $s(\mu) = r(\omega)$, we have
$U_{\mu\omega}\subseteq U_{\omega}$, and for any infinite path $x \in
E^{-\infty}$, we have $\bigcap_{n\in\mathbb{N}}U_{x_{-n}...x_{-1}} = [x]$.
Hence $\lim_{n\to\infty}\sup_{\mu\in E^{n}}\text{diam}(U_{\mu}) = 0$. Let $k$
be as in Proposition~\ref{prp:mapping_properties}(2), and fix $n > k$ so that
$\text{diam}(U_{\mu}) < \varepsilon$ whenever $|\mu| \ge n$. Using
compactness of $\mathcal{J}$, fix $K \subseteq \bigcup_{m \ge n-1} E^m$ such
that $\mathcal{J} \subseteq \bigcup_{\omega \in K} U_\omega$. The Lebesgue
Number Lemma yields $0 < \eta < \varepsilon$ such that for every $w$ in
$\mathcal{J}$, there exist $\omega \in K$ such that $B(w,c\eta) \subseteq
U_{\omega}$. These values of $n,\eta$ satisfy~(1) by construction, so we
just have to establish~(2).

For this, let $\mu$ be a path such that $|\mu|\geq n$. Since $n\geq k$, by
Proposition~\ref{prp:mapping_properties}(2), $\tilde{\sigma}$ maps $U_{\mu}$
homeomorphically onto $U_{\sigma(\mu)}$. Suppose that
$B(\tilde{\sigma}(z),c\alpha)\subseteq U_{\sigma(\mu)}$ for $z \in U_{\mu}$
and $\alpha\leq\eta$. By hypothesis, $\varepsilon$ satisfies
Theorem~\ref{thm:Jsig}(1), and so since $\text{diam}(U_{\mu}) < \varepsilon$,
we have
\[
c\cdot d_{\mathcal{J}}(\tilde{\delta}(\tilde{\sigma} z), \tilde{\delta}(w))
    = d_{\mathcal{J}}(\tilde{\sigma}\tilde{\delta}(\tilde{\sigma} z), \tilde{\sigma}\tilde{\delta}(w))
    = d_{\mathcal{J}}(\tilde{\sigma}z, w)
\]
whenever $d_{\mathcal{J}}(\tilde{\sigma}z, w) < c\alpha$. Hence,
$\tilde{\delta}(B(\tilde{\sigma}z,c\alpha))\subseteq B(z,\alpha)\cap
U_{\mu}$, so that
$B(\tilde{\sigma}z,c\alpha)\subseteq\tilde{\sigma}(B(z,\alpha)\cap U_{\mu})$.
Since $\alpha\leq\eta <\varepsilon$, another application of
Theorem~\ref{thm:Jsig}(1) implies that $\tilde{\sigma}$ restricts to a
homeomorphism of $B(z,\alpha)$, and that
$\tilde{\sigma}(B(z,\alpha))\subseteq B(\tilde{\sigma}(z),c\alpha)$.
Therefore, $B(z,\alpha)\cap U_{\mu} = B(z,\alpha)$, and $\tilde{\sigma}(B(z,
\alpha)) = B(\tilde{\sigma}(z),c\alpha).$ Hence, $B(z,\alpha)\subseteq
U_{\mu}$, so that $\tilde{\delta}\circ\tilde{\sigma}|_{B(z,\alpha)}$ is well
defined, and $\tilde{\delta}(B(\tilde{\sigma}(z),c\alpha)) =
\tilde{\delta}\tilde{\sigma}(B(z,\alpha)) = B(z,\alpha)$.
\end{proof}
Now we finish proving the main result of this section.
\begin{proof}[Proof of Theorem~\ref{thm:Jsig}(2)]
Let $\epsilon > 0$ be as in the proof of Theorem~\ref{thm:Jsig}(1).
By \\Lemma~\ref{lem:delta_property}(1), there is $\varepsilon < \epsilon$ and $n\in\mathbb{N}$ such that
for any $z \in \mathcal{J}$ and
$\alpha \leq \varepsilon$, there exists $\omega \in E^*$ with $|\omega|\geq
n-1$ and $B(\tilde{\sigma}(z),c\alpha)\subseteq U_{\omega}$. By
Proposition~\ref{prp:mapping_properties}(3), there exists $f \in E^1$ such
that $s(\omega) = r(f)$ and $z \in U_{\omega f}$.
Lemma~\ref{lem:delta_property}(2) then gives $B(z,\alpha)\subseteq U_{\omega
f}$ and $\tilde{\sigma}(B(z,\alpha)) =
\tilde{\sigma}\tilde{\delta}(B(\tilde{\sigma}(z),c\alpha)) =
B(\tilde{\sigma}(z),c\alpha)$.
\par
This concludes the proof of Theorem~\ref{thm:Jsig} with $\varepsilon > 0$ as above, $c >1$ and $d_{\Jj}$.
\end{proof}
To finish this section, we will show that for strongly-connected graphs, the
regularity hypothesis is necessary in Theorem~\ref{thm:Jsig} (we will need to
restrict to strongly-connected graphs later in order to apply Kaminker,
Putnam and Whittaker's results about $KK$-duality for $C^*$-algebras
associated to Smale spaces). Recall that a directed graph $E$ is
\emph{strongly connected} if it has at least one edge, and if for all $v,w
\in E^0$ the set $v E^* w$ is nonempty.

\begin{lem}\label{lem:discerning mu}
Let $E$ be a finite directed graph with no sources. Let $(G,E)$ be a
contracting self-similar groupoid action with nucleus $\Nn$. Then there
exists $\mu  \in E^*$ such that whenever $g  \in \Nn$ satisfies $d(g) =
r(\mu)$ and $g \cdot \mu = \mu$, we have $g \cdot \mu\nu = \mu\nu$ for all
$\nu $ in $s(\mu)E^*$.
\end{lem}
\begin{proof}
Suppose first that there is $x \in E^{\infty}$ such that $g\cdot x\neq x$ for
all $g \in \mathcal{N}\setminus E^{0}$ satisfying $d(g) = r(x)$. Choose $n\in
\mathbb{N}$ so that $g\cdot x_{1}...x_{n}\neq x_{1}....x_{n}$ for all $g$ in
$\mathcal{N}\setminus E^{0}$ satisfying $d(g) = r(x)$. Then $\mu =
x_{1}...x_{n}$ has the desired property.

Now suppose that for every $x \in E^{\infty}$, there exists
$g\in\mathcal{N}\setminus E^{0}$ such that $d(g) = r(x)$ and $g\cdot x = x$.
Then, $\bigcup_{g\in\mathcal{N}\setminus E^{0}}\{x\in d(g)E^{\infty}:g\cdot x
= x\} = E^{\infty}$. Since a finite intersection of open dense sets is itself
an open dense set, there exists $x \in E^{\infty}$ that does not belong to
the boundary of $\{x\in d(g)E^{\infty}:g\cdot x = x\}$ for any $g \in
\mathcal{N}\setminus E^{0}$. So, there is an $n \in \mathbb{N}$ such that for
all $g$ in $\mathcal{N}\setminus E^{0}$ with $d(g) = r(x)$, either $g\cdot
x_{1}...x_{n}\neq x_{1}...x_{n}$ or $g\cdot y = y$ for all $y$ in
$Z[x_{1}...x_{n})$. Hence, $\mu = x_{1}...x_{n}$ has the desired property.
\end{proof}

\begin{prp}
Let $E$ be a strongly connected finite directed graph. Let $(G,E)$ be a contracting self-similar groupoid
action with nucleus $\Nn$. If $\Jsig : \Jj \to \Jj$ is a local homeomorphism, then $(G, E)$ is
regular. Hence, $(G, E)$ is regular if and only if $\Jsig$ is a local homeomorphism
\end{prp}
\begin{proof}
Suppose that $(G, E)$ is not regular. It suffices to show that there exists
$k  \in \NN$ such that $\Jsig^k$ is not locally injective. Since $(G, E)$ is
not regular, there exist $x  \in E^\infty$ and $g  \in G$ such that $g \cdot
x = x$ but $g$ fixes no neighbourhood of $x$. So there is a strictly
increasing sequence $(n_j)$ in $\NN$ and paths $\alpha_j  \in s(x_{n_j})E^*$
such that $g \cdot x_1 \dots x_{n_j}\alpha_{j} \not= x_1 \dots
x_{n_j}\alpha_{j}$. In particular, the elements $h_j := g|_{x_1 \dots x_j}$
satisfy $h_j \cdot \alpha_j \not= \alpha_j$ for all $j$. By Lemma~\ref{lem:N
a core} we have $h_j \in \Nn$ for large $j$. Since $\Nn$ is finite, by
passing to a subsequence, we can assume that $h_j = h_1 =: h$ for all $j$
(and hence $s(x_{n_j}) = d(h)$ for all $j$). So $\alpha := \alpha_1 \in
d(h)E^*$ satisfies $h \cdot \alpha \not= \alpha$; let $\beta := h \cdot
\alpha$.

For each $j$, fix $y_j  \in \lZ{x_1 \dots x_{n_j}} \subseteq E^{-\infty}$.
Since $E^{-\infty}$ is compact, by passing to a subsequence, we can assume
that $y_j \to y \in E^{-\infty}$. Since $s(y_j) = s(x_{n_j}) = d(h)$ for all
$j$, we have $s(y) = d(h)$, so $y\alpha, y\beta \in E^{-\infty}$. By
definition of convergence in $E^{-\infty}$, for each $N  \in \NN$ there
exists $j$ such that $y_{-N} \dots y_{-1} = x_{n_j - N+1} \dots x_{n_j}$. For
this $j$, the element $g_N := g|_{x_1 \dots x_{n_j - N + 1}}$ satisfies $g_N
\cdot y_{-N} \dots y_{-1} \alpha = y_{-N} \dots y_{-1} \beta$. Hence $y\alpha
\sim_\aeq y_\beta$. That is, $[y_\alpha] = [y_\beta] \in \Jj$. Moreover, $k
:= |\alpha| = |\beta|$ satisfies $\Jsig([y\alpha]) = [y] = \Jsig([y\beta])$.

By Lemma~\ref{lem:discerning mu}, there exists $\mu  \in E^*$ such that every
$g  \in \Nn$ that satisfies $g \cdot\mu = \mu$ pointwise fixes $\rZ{\mu}$.
Fix $z  \in \lZ{r(\mu)}$ so that $z\mu \in E^{-\infty}$. Since $E$ is
strongly connected, for each $n  \in \NN$ there exists $\nu_n  \in E^*$ such
that $r(\nu_n) = s(\mu)$ and $s(\nu_n) = r(y_n)$. So for each $n$ we obtain
elements $z\mu\nu_n y_n \dots y_{-1} \alpha$ and $z\mu\nu_n y_n \dots y_{-1}
\beta$ of $E^{-\infty}$. Since $\alpha \not= \beta$, our choice of $\mu$
ensures that $g \cdot \mu\nu_n y_n \dots y_{-1}\alpha \not= \mu\nu_n y_n
\dots y_{-1}\beta$ for all $g  \in \Nn$, and so Lemma~\ref{lem:ae by nucleus}
shows that $z\mu\nu_n y_n \dots y_{-1} \alpha \not\sim_\aeq z\mu\nu_n y_n
\dots y_{-1} \beta$ for all $n$. We have $z\mu\nu_n y_n \dots y_{-1} \alpha
\to y\alpha$ and $z\mu\nu_n y_n \dots y_{-1} \beta \to y_\beta$, and
$\Jsig^k([z\mu\nu_n y_n \dots y_{-1} \alpha]) = [z\mu\nu_n y_n \dots y_{-1}]
= \Jsig^k([z\mu\nu_n y_n \dots y_{-1} \beta])$ for all $n$, and therefore
$\Jsig^k$ is not locally injective.
\end{proof}
\par

%%%%%%%%%%%%%%%%%%%%%%%%%%HERE%%%%%%%%%%%%%%%%%%%%%%%%%%%%%

\section{The Smale space of a self-similar groupoid action on a graph}\label{sec:smale space}

In this section we describe the Wieler Smale space that arises from the
locally expanding dynamics described in the preceding section, and show that
this Smale space can be realised as the quotient of $E^\ZZ$ by the natural
extension of asymptotic equivalence.

We first recall Wieler's axioms, under which the projective limit of a space $V$ under iterates of a given continuous surjection $V \to V$ becomes a Smale space with totally disconnected stable set. Wieler proves
more, showing that every Smale space with totally disconnected stable set has this form, but we will not need the full power of her theorem.

For the statement of Wieler's Theorem, recall that if $g : X \to X$ is a continuous self-mapping
of a topological space, then $\varprojlim(X, g)$ is the space
\[
\varprojlim(X,g) := \{(x_n)^\infty_{n=1} \mid x_i \in X\text{ and }g(x_{i+1}) = x_i\text{ for all }i\}.
\]
If $\phi: X \to X$ is a continuous self-mapping of a topological space, we
say that a point $x  \in X$ is \emph{non-wandering} if for every
neighbourhood $U$ of $x$ there exists $n \ge 1$ such that $\phi^n(U) \cap U
\not= \varnothing$. Finally, recall that the \emph{forward orbit} of $x  \in
X$ is $\{\phi^n(x) \mid n \ge 0\}$.

\begin{thm}[{Wieler \cite[Theorem~A]{Wieler:Smale}}]\label{thm:Wieler}
Let $(X, d)$ be a compact metric space, and let $\varphi : X \to X$ be a
continuous surjection. Suppose that there exist $\varepsilon > 0, K \in \NN
\setminus \{0\}$, and $\gamma \in (0,1)$ such that
\begin{description}
\item[Axiom 1]\label{WAxiom1} For all $x, y \in X$ such that
    $d(x,y)<\varepsilon$, we have
\[
d(\varphi^K(x),\varphi^K(y))\leq\gamma^K d(\varphi^{2K}(x), \varphi^{2K}(y)).
\]
\item[Axiom 2] For all $x \in X$ and $0 < \alpha < \varepsilon$,
\[
\varphi^K(B(\varphi^K(x),\alpha))\subseteq \varphi^{2K}(B(x,\gamma\alpha)).
\]
\end{description}
Then
\[
d_\infty((x_n), (y_n)) := \sum^K_{n=1} \gamma^{-n} \sup_{m \in \NN} \gamma^m d(x_{m+n}, y_{m+n}),
\]
defines a metric on $X_\infty := \varprojlim(X, \varphi)$, the formula
\[
\varphi_\infty(x_1, x_2, \dots) = (\varphi(x_1), x_1, x_2, \dots)
\]
defines a homeomorphism $\varphi_\infty : X_\infty \to X_\infty$, and
$(X_\infty, \varphi_\infty)$ is a Smale space with totally disconnected
stable set. This Smale space is irreducible if and only if every point in $X$
is nonwandering, and there is a point in $X$ whose forward orbit under
$\varphi_\infty$ is dense.
\end{thm}

The key point for us is that the results of the preceding two sections show that every
contracting, regular self-similar groupoid action on a finite directed graph
with no sources gives rise to a dynamical system satisfying Wieler's axioms.

\begin{lem}
Let $E$ be a finite directed graph with no sources. Let $(G, E)$ be a
contracting, regular self-similar groupoid action. Let $\Jj$ be the limit
space of Definition~\ref{dfn:limit space}, let $d_\Jj$ be the equivalence
relation metric as in Corollary~\ref{cor:limit space metrisable}, and let
$\Jsig$ be the local homeomorphism of Theorem~\ref{thm:Jsig}. Let
$\varepsilon$ and $c$ be as in the statement of Theorem~\ref{thm:Jsig}. Then the pair
$(\Jj, \Jsig)$ satisfies Wieler's axioms for $\gamma = \frac{1}{c}$, $K = 1$
and $\varepsilon' = \frac{\varepsilon}{2}$.
\end{lem}
\begin{proof}
Theorem~\ref{thm:Jsig} (1) shows that if $d_\Jj([x], [y]) < \varepsilon'$ then
\begin{align*}
d_{\Jj}(\Jsig^K([x]), \Jsig^K([y]))
    &= d_{\Jj}(\Jsig([x]), \Jsig([y]))\\
    &= \frac{1}{c} d_{\Jj}(\Jsig^2([x]), \Jsig^2([y]))
    = \gamma d_{\Jj}(\Jsig^{2K}([x]), \Jsig^{2K}([y])),
\end{align*}
establishing Axiom~1.

Theorem~\ref{thm:Jsig}(2) shows that
$\tilde{\sigma}^{2K}(B([x],\gamma\alpha)) =
\tilde{\sigma}^{K}(B(\tilde{\sigma}^{K}([x]),\alpha))$ for all $\alpha\leq
\varepsilon'$ and $[x] \in \mathcal{J}$, establishing Axiom~2.
\end{proof}

We now identify the limit space $\Jj_\infty$ obtained from Theorem~\ref{thm:Wieler} applied to
$(\Jj, \Jsig)$ with a quotient of the bi-infinite path space of $E$.

We define asymptotic equivalence on bi-infinite paths just as we define it
for right-infinite paths. That is, if $(G, E)$ is a self-similar groupoid
action and $x,y \in E^\ZZ$, then $x \sim_\aeq y$ if there exists a
bi-infinite sequence $(g_n)_{n \in \ZZ}$ in $G$ such that $\{g_n \mid n \in
\ZZ\}$ is a finite set, and such that $g_n \cdot x_n x_{n+1} x_{n+2} \dots =
y_n y_{n+1} y_{n+2} \dots$ for all $n  \in \ZZ$.

The argument of Lemma~\ref{lem:ae by nucleus} shows that $x, y  \in E^\ZZ$
are asymptotically equivalent if and only if there is a sequence $(g_n)_{n
\in \ZZ}$ in $\Nn$ such that $g_n \cdot x_n x_{n+1} x_{n+2} \dots = y_n
y_{n+1} y_{n+2} \dots$ for all $n$ and such that $g_n|_{x_n} = g_{n+1}$ for
all $n$.

\begin{dfn}\label{def:S space}
Let $E$ be a finite directed graph with no sources. Let $(G, E)$ be a contracting, self-similar groupoid action. We write $\Ss$ for the quotient space
$E^\ZZ/{\sim_\aeq}$ and call this the \emph{limit solenoid} of $(G, E)$.
\end{dfn}

We will need the following notation. Given a directed graph $E$ with no sinks
or sources, we will write $\tau : E^\mathbb{Z} \to E^\mathbb{Z}$ for the
translation homeomorphism $\tau(x)_n = x_{n-1}$, $n \in \mathbb{Z}$. For $n
\in \ZZ$ and $x  \in E^{-\infty}$ we will write $x(-\infty, n)$ for the
element of $E^{-\infty}$ given by $x(-\infty, n) = \dots x_{n-2} x_{n-1}
x_{n}$.

\begin{prp}\label{prp:S-J conjugate}
Let $E$ be a finite directed graph with no sinks or sources. Let $(G, E)$ be
a contracting, self-similar groupoid action. Let $\Jj$ be the limit space of
$(G, E)$, and let $\Jsig : \Jj \to \Jj$ be the map defined as in
Theorem~\ref{thm:Jsig}. Let $\Jj_\infty := \varprojlim(\Jj, \Jsig)$, and let
$\Jsig_\infty : \Jj_\infty \to \Jj_\infty$ be the homeomorphism of defined as
in Theorem~\ref{thm:Wieler}. Let $\Ss = E^\ZZ/{\sim_\aeq}$ be the limit
solenoid of $(G, E)$. Then there is a homeomorphism $\theta : \Ss \to
\Jj_\infty$ such that $\theta([x]) = ([x(-\infty, -1)], [x(-\infty, 0)],
[x(-\infty, 1)], \dots)$ for all $x  \in E^\ZZ$. We have $\theta([\tau(x)]) =
\Jsig_\infty(\theta([x]))$ for all $x  \in E^\ZZ$.
\end{prp}
\begin{proof}
If $x,y  \in E^{\ZZ}$ satisfy $x \sim_\aeq y$, then $x(-\infty, n) \sim_\aeq
y(-\infty, n)$ for all $n  \in \ZZ$. So the formula $\theta([x]) =
([x(-\infty, -1)], [x(-\infty, 0)], [x(-\infty, 1)], \dots)$ is well defined
and determines a map $\theta : \Ss \to \prod^\infty_{n=1} \Jj$. By
definition of $\Jsig$, we have $\Jsig([x(-\infty, n)]) = [\sigma(x(-\infty,
n))] = [x(-\infty, n-1)]$ for all $n$, and so each $\theta([x]) \in
\Jj_\infty$.

Since $E^\ZZ$ is compact, so is its continuous image $\Ss$. Projective limits of Hausdorff spaces are Hausdorff, so $\Jj_\infty$ is Hausdorff. So, to see that $\theta$ is a homeomorphism, it suffices to show that
it
a continuous bijection.

The maps $E^\ZZ \owns x \mapsto x(-\infty, n)$ indexed by $n  \in \ZZ$ are
clearly continuous, and so the maps $x \mapsto [x(-\infty, n)]$ are also
continuous because the quotient map from $E^{-\infty}$ to $\Jj$ is
continuous. Hence $\bar\theta(x) := ([x(-\infty, -1)], [x(-\infty, 0)],
[x(-\infty, 1)], \dots)$ defines a continuous map $\bar\theta : E^\ZZ \to
\Jj_\infty$. Since $\theta : \Ss \to \Jj_\infty$ is the map induced by
$\bar\theta$, it is also continuous.

To see that $\theta$ is surjective, fix $(\zeta_1, \zeta_2, \zeta_3, \dots)
\in \Jj_\infty$. For each $j$, choose $x_j  \in E^{-\infty}$ such that $[x_j]
= \zeta_j$. Since each $\Jsig(\zeta_j) = \zeta_{j-1}$, we have $\sigma(x_j)
\sim_\aeq x_{j-1}$ for all $j$. For each $j \ge 1$, consider the sequence
$(\sigma^{n-j}(x_n))^\infty_{n=j}$ in $E^{-\infty}$. We just saw that each
$\sigma^{n-j}(x_n) \sim_\aeq x_j$, and so Corollary~\ref{cor:finite orbits}
shows that this sequence contains just finitely many distinct elements of
$E^{-\infty}$. A standard Cantor diagonal argument yields a subsequence
$(x_{n_k})$ such that $\big(\sigma^{n_k -j}(x_{n_k})\big)_{n_k \ge j}$ is a
constant sequence for each $j$. Define $y  \in E^\ZZ$ by $y_j = \lim_k
\sigma^{n_k - j}(x_{n_k})$. By construction, $\sigma(y_j) = y_{j-1}$ for all
$j$. Also, by construction, each $y(-\infty, n) \sim_\aeq x_n$ and so
$\theta([y]) = (\zeta_1, \zeta_2, \zeta_3, \dots)$.

To show that $\theta$ is injective, suppose that $\theta([x]) = \theta([y])$.
We must show that $x \sim_\aeq y$. For this, fix $m  \in \ZZ$. It suffices to
find $g  \in \Nn$ such that $g \cdot x(m,\infty) = y(m,\infty)$. Since
$\theta([x]) = \theta([y])$, we have $x(-\infty, n) \sim_\aeq y(-\infty, n)$
for all $n$. Fix $n \ge m$. Lemma~\ref{lem:ae by nucleus} shows that there is
a sequence $(g_{k,n})_{k \le n}$ in $\Nn$ such that $g_{k,n}\cdot x_k \dots x_n =
y_k \dots y_n$ for all $k$. Taking $k = m \le n$, we obtain $g_n := g_{m,n}
\in \Nn$, such that $g_n \cdot x_m \dots x_n = y_m \dots y_n$. Since $\Nn$ is
finite, the sequence $(g_n)_{n \ge m}$ has a constant subsequence. The
constant value $g$ of this subsequence then satisfies $g \cdot x_m \dots x_n
= x_m \dots y_n$ for infinitely many $n \ge m$. It then follows that $g \cdot
x(m,\infty) = y(m, \infty)$ as required.

It remains to check that $\theta([\tau(x)]) = \Jsig_\infty(\theta([x]))$ for
all $x \in E^\ZZ$. This follows from direct calculation: for $x  \in E^\ZZ$,
\begin{align*}
\theta([\tau(x)])
    &= \big([\tau(x)(-\infty, -1)], [\tau(x)(-\infty, 0)], [\tau(x)(-\infty, 1)], \dots\big)\\
    &= \big([x(-\infty, -2)], [x(-\infty, -1)], [x(-\infty, 0)], \dots\big)\\
    &= \big(\Jsig([x(-\infty, -1)]), [x(-\infty, -1)], [x(-\infty, 0)], \dots\big)\\
    &= \Jsig_\infty(\theta([x])).\qedhere
\end{align*}
\end{proof}

\begin{cor}\label{cor:Smale solenoid}
Let $E$ be a finite directed graph with no sinks or sources. Let $(G, E)$ be
a contracting, regular self-similar groupoid action. Let $\Ss :=
E^\ZZ/{\sim_\aeq}$ be the limit solenoid of $(G, E)$. Then there is a
homeomorphism $\Stau : \Ss \to \Ss$ such that $\Stau([x]) = [\tau(x)]$ for
all $x  \in E^\ZZ$. Let $d_\Jj$ be the quotient metric on $\Jj$ as in
Corollary~\ref{cor:limit space metrisable}. There is a metric
$d_{\mathcal{S}}$ on $\mathcal{S}$ such that
\[
d_{\mathcal{S}}([x],[y]) = \sup_{m\in\mathbb{N}_{0}}(\frac{1}{c})^{m}d_\Jj([x(-\infty,m)],[y(-\infty, m)])
\]
for all $x,y \in E^\ZZ$. There is a constant $\varepsilon_{\mathcal{S}}$ such
that $(\Ss, d_{\mathcal{S}}, \Stau, \varepsilon_{\mathcal{S}}, \frac{1}{c})$
is a Smale space with totally disconnected stable set.

If $E$ is strongly connected, then $(\Ss, \Stau)$ is irreducible. If $E$ is
primitive, then $(\Ss, \Stau)$ is topologically mixing.
\end{cor}
\begin{proof}
Theorem~\ref{thm:Wieler} shows that $(\Jj_\infty, \Jsig_\infty)$ is a Smale
space with totally disconnected stable set, and Proposition~\ref{prp:S-J
conjugate} shows that $(\Ss, \Stau)$ is conjugate to $(\Jj_\infty,
\Jsig_\infty)$. Let $d_{\infty}$ be the metric of Theorem~\ref{thm:Wieler}.
For $x,y \in E^{\mathbb{Z}}$, we have
\[
d_{\infty}(\theta([x]),\theta([y])) = c\cdot \sup_{m\in\mathbb{N}}(\frac{1}{c})^{m}d(\theta([x])_{m+1}, \theta([y])_{m+1}).
\]
Since $\theta([x])_{m+1} = [x(-\infty, m)]$, we have $d_{\infty} =
d_{\mathcal{S}}\circ (\theta\times \theta)$. Hence, there is a constant
$\varepsilon_{\mathcal{S}} > 0$ such that $(\mathcal{S}, d_{\mathcal{S}},
\tilde{\tau}, \varepsilon_{S}, \frac{1}{c})$ is a Smale space.
\par
For the irreducibility, by Wieler's Theorem it suffices to show that every
point in $\Jj$ is non-wandering and that $\Jj$ admits a dense orbit. To see
that every point is non-wandering, first observe that if $x  \in E^{-\infty}$
is periodic, say $\sigma^n(x) = x$, then $[x]  \in \Jj$ satisfies
$\Jsig^n([x]) = [\sigma^n(x)] = [x]$, so $[x]$ is also periodic. Since $E$ is
strongly connected, for each $\lambda  \in E^*$, there exists $\mu $ in
$s(\lambda) E^*r(\lambda)$, and then $x := (\lambda\mu)^\infty$ is a periodic
point in $Z(\lambda)$. So there is a dense set of periodic points. It follows
that the periodic points in $\Jj$ are dense. So for any $[y]  \in \Jj$ and
any open neighbourhood $U$ of $[y]$, we can find $[x]  \in U$ and $n \ge 1$
such that $\Jsig^n([x]) = [x]$, and so $[x] \in \Jsig^n(U) \cap U$. To see
that $\Jj$ has a dense orbit, let $\lambda_1, \lambda_2, \lambda_3 \dots$ be
a listing of $E^*$. For each $i \ge 1$, choose $\mu_i  \in s(\lambda_i) E^*
r(\lambda_{i+1})$. Then $x = \lambda_1\mu_1\lambda_2\mu_2\dots \in E^\infty$.
For each $\lambda  \in E^*$, we have $\lambda = \lambda_i$ for some $i$, and
then $\sigma^{|\lambda_1\mu_1 \dots \lambda_{i-1}\mu_{i-1}|}(x) \in
Z(\lambda)$. Hence $\{\sigma^n(x) \mid n \in \NN\}$ is dense  $E^\infty$.
Hence $\{\Jsig^n([x]) \mid n \in \NN\} = q(\{\sigma^n(x) \mid n \in \NN\}$ is
a dense forward orbit in $\Jj$.
\par
If $E$ is primitive, then $\tau:E^{\ZZ}\mapsto E^{\ZZ}$ is topologically mixing, see \cite[Observation 7.2.2]{Kitchens}. Since the quotient map $q:E^{\ZZ}\mapsto \Ss$ satisfies $q\circ\tau = \tilde{\tau}\circ q$
and is surjective, $\tau$ being topologically mixing implies $\tilde{\tau}$ is topologically mixing.
\end{proof}

\section{The $C^*$-algebra of a self-similar groupoid action on a graph}\label{sec:O(G,E)}

In this section and the next, we will discuss two $C^*$-algebras associated
to self-similar groupoid actions. The first of these is the $C^*$-algebra
$\Oo(G, E)$ described by Laca--Raeburn--Ramagge--Whittaker in
\cite{Laca-Raeburn-Ramagge-Whittaker:Equilibrium}, see also
\cites{lrrw,Nekrashevych:Cuntz--Pimsner,Nekrashevych:Cstar_selfsimilar}. Our
main goal is to provide a groupoid model based on the one developed for
self-similar group actions on graphs by Exel and Pardo
\cite{Exel-Pardo:Self-similar}. This is the subject of the present section.
In the next section, we consider the $C^*$-algebra obtained from the
Deaconu--Renault groupoid of the dynamics $(\Jj,\Jsig)$ of Section~\ref{sec:J
dynamics}. Our main result will establish $KK$-duality between these two
$C^*$-algebras for contracting, regular self-similar actions.

In \cite{Laca-Raeburn-Ramagge-Whittaker:Equilibrium}, the Toeplitz algebra of
a self-similar groupoid action is defined as the Toeplitz algebra of an
associated Hilbert module. Then Proposition~4.4 of
\cite{Laca-Raeburn-Ramagge-Whittaker:Equilibrium} provides an alternative
description as the universal $C^*$-algebra for generators and relations. At
the beginning of Section~8 of
\cite{Laca-Raeburn-Ramagge-Whittaker:Equilibrium}, the Cuntz--Pimsner algebra
of the self-similar action is defined as the quotient of the Toeplitz algebra
by the ideal determined by an additional Cuntz--Krieger-type relation. We
follow \cite{Laca-Raeburn-Ramagge-Whittaker:Equilibrium} and define the
$C^*$-algebra of a self-similar action in terms of generators and relations.

If $G$ is a discrete groupoid, then a \emph{unitary representation} of $G$ is
a function $g \mapsto u_g$ from $G$ to a $C^*$-algebra such that $u_g u_h =
\delta_{d(g), c(h)} u_{gh}$ and $u_{g^{-1}} = u_g^*$ for all $g,h \in G$.
This is equivalent to the definition presented at the start of
\cite[Section~4]{Laca-Raeburn-Ramagge-Whittaker:Equilibrium}.

If $E$ is a finite directed graph with no sources, and $(G, E)$ is a
self-similar groupoid action, then a \emph{covariant representation} of $(G,
E)$ in a $C^*$-algebra $A$ is a triple $(u, p, s)$ consisting of a unitary
representation $u : g \mapsto u_g$ of $G$ in $A$ and a Cuntz--Krieger
$E$-family $(p, s) \in A$ such that $p_v = u_v$ for all $v  \in E^0$, and
such that
\[
u_g s_e = s_{g \cdot e} u_{g|_e}\quad\text{ for all $g \in G$ and $e \in d(g)E^1$.}
\]
We have $p_v u_g = \delta_{v, c(g)} u_g$ and $u_g p_v = \delta_{d(g), v} u_g$
because $p_v = u_v$. If $d(g) \not= r(e)$, then $u_g s_e = u_g p_{d(g)}
p_{r(e)} s_e = 0$. So the relations we have just presented are equivalent to
those of \cite[Proposition~4.4]{Laca-Raeburn-Ramagge-Whittaker:Equilibrium}
combined with the additional relation determining the generators of the ideal
$I$ described in
\cite[Equation~(8.1)]{Laca-Raeburn-Ramagge-Whittaker:Equilibrium}. It follows
from Proposition~4.4 and the definition of $\Oo(G, E)$ in
\cite{Laca-Raeburn-Ramagge-Whittaker:Equilibrium} that the $C^*$-algebra
$\Oo(G, E)$ is the universal $C^*$-algebra generated by a covariant
representation of $(G, E)$.

Our first step is to describe a groupoid model for $\Oo(G, E)$. Our construction is based on that
of \cite{Exel-Pardo:Self-similar}.

\begin{lem}
Let $E$ be a finite graph, and let $(G, E)$ be a self-similar groupoid action.
The set
\[
S_{G, E} := \{0\} \cup \{(\mu, g, \nu) \in E^* \times G \times E^* \mid s(\mu) =
c(g)\text{ and }s(\nu) = d(g)\}
\]
is an inverse semigroup with respect to the multiplication given
by
\[
(\alpha, g, \beta)(\mu, h, \nu)
    = \begin{cases}
        (\alpha (g \cdot \mu'), g|_{\mu'}h, \nu) &\text{ if $\mu = \beta\mu'$}\\
        (\alpha, g h|_{h^{-1}\cdot\beta'}, \nu(h^{-1} \cdot \beta')) &\text{ if $\beta = \mu\beta'$}\\
        0 &\text{ otherwise.}
    \end{cases}
\]
There is an action of $S_{G, E}$ on $E^\infty$ such that $\Dom(\mu, g, \nu) = \rZ{\nu}$,
and
\[
(\mu, g, \nu) \cdot \nu x = \mu g\cdot x\quad\text{ for all $x \in \rZ{s(\nu)}$.}
\]
\end{lem}
\begin{proof}
It is routine, though tedious, to check that this multiplication is
associative. For each $a := (\mu, g, \nu)$, the element $a^{*} := (\nu,
g^{-1}, \mu)$ satisfies $aa^{*}a = a$ and $a^{*}aa^{*} = a^{*}$. Direct
computation shows that the formula $(\mu, g, \nu) \cdot \nu x = \mu g\cdot x$
defines a homeomorphism from $\rZ{\nu}$ to $\rZ{\mu}$. A routine calculation
very similar to the associativity calculation shows that $a \cdot (b \cdot x)
= (ab)\cdot x$ whenever both sides are defined.
\end{proof}

Given any action of an inverse semigroup $S$ on a locally compact Hausdorff
space $X$, we can form the associated groupoid of germs $S \ltimes X$ as
follows \cite[Section~4.3]{Paterson}: we define an equivalence relation on
$\{(s, x) \mid s\in S, x \in \Dom(s)\}$ by $(s, x) \sim (t, y)$ if $x = y =:
z$ and there is an idempotent $e
 \in S$ such that $z \in \Dom(e)$, and $se = te$. The topology has basic open sets $W(s, V) :=
\{[s, x] \mid x \in V\}$ indexed by pairs $(s, V)$ consisting of an element
$s  \in S$ and an open set $V \subseteq \Dom(s)$. The unit space of this
groupoid is $X$, and the groupoid operations are given by
\[
s([t, x]) = x,\quad r([t, x]) = t \cdot x,\quad [t,u\cdot x][u, x] = [tu, x],
    \quad\text{ and }\quad [t, x]^{-1} = [t^*, t \cdot x].
\]
Though this groupoid need not be Hausdorff, it is always \'etale with Hausdorff unit space $X$,
and hence locally Hausdorff with a basis of open bisections. The $C^*$-algebra of this groupoid is
the completion of the $^*$-algebra
\[
\Cc(S \ltimes X) = \lsp\{C_c(U) \mid U \text{ is an open bisection}\}
\]
in a universal norm. A very nice account of this construction can be found in
\cite{Exel:Inverse_combinatorial}.

\begin{dfn}\label{def:germ gpd}
Let $E$ be a finite directed graph, and let $(G, E)$ be a self-similar groupoid action.
The groupoid of $(G, E)$, denoted $S_{G, E} \ltimes E^\infty$ is defined to be the groupoid of
germs for the action of $S_{G, E}$ on $E^\infty$ as above.
\end{dfn}

To establish our duality theorem later, we will describe a groupoid
equivalence between the groupoid $S_{G, E} \ltimes E^\infty$ and the stable
groupoid of the Smale space constructed in Section~\ref{sec:smale space}. To
do this, it will be helpful first to establish a description of $S_{G, E}
\ltimes E^\infty$ as a kind of lag groupoid. A related description for
self-similar group actions appears in
\cite[Section~8]{Exel-Pardo:Self-similar}, though there the lag takes values
in the ``sequence group" $\prod^\infty_{i=1} G/\bigoplus^\infty_{i=1} G$. We
will give yet another description, which is particularly well suited to our
application to $KK$-duality later. We also characterise exactly when this
groupoid is Hausdorff, by characterising exactly which pairs of elements (if
any) cannot be separated by disjoint open neighbourhoods.

Our groupoid is based on the left-shift map on $E^\infty$ given by $x_1 x_2 x_3 \dots \mapsto x_2
x_3 \dots$. In the graph-algebra literature, it is standard to denote this shift map by $\sigma$,
but we have already used that symbol for the right-shift on $E^{-\infty}$. We will instead use
$\varsigma$ for the left-shift map.

\begin{lem}\label{lem:groupoid def}
Let $E$ be a finite directed graph, and let $(G, E)$ be a self-similar
groupoid action. There is an equivalence relation $\sim$ on
\[
\{(x, m, g, n, y) \in E^\infty \times \NN \times G \times \NN \times E^\infty
 \mid d(g) = r(\varsigma^n(y))\text{ and } \varsigma^m(x) = g \cdot \varsigma^n(y)\}
\]
such that $(x, m, g, n, y) \sim (w, p, h, q, z)$ if and only if
\begin{itemize}
\item $x = w$, $y = z$ and $m-n = p-q$, and
\item there exists $l \ge \max\{n,q\}$ such that $g|_{y(n, l)} = h|_{z(q,l)}$.
\end{itemize}
We write $[x,m,g,n,y]$ for the equivalence class of $(x, m, g, n, y)$ under $\sim$. The set
\[
\Gg_{G, E} := \{[x, m, g, n, y] \mid d(g) = r(\varsigma^n(y))\text{ and } \varsigma^m(x) = g \cdot \varsigma^n(y)\}
\]
is an algebraic groupoid with unit space $\{[x, 0, r(x), 0, x] \mid x \in E^\infty\}$ identified
with $E^\infty$, range and source maps $r([x,m,g,n,y]) = x$ and $s([x,m,g,n,y]) = y$, and
operations
\begin{gather*}
    [x,m,g,n,y][y,p,h,q,z] = [x, m+p, g|_{y(n, n+p)} h|_{z(q, q+n)}, n+q, z],\quad\text{ and}\\
    [x, m, g, n, y]^{-1} = [y, n, g^{-1}, m, x].
\end{gather*}
There is an injective homomorphism $\iota$ of the graph groupoid $\Gg_E$ into
$\Gg_{G, E}$ given by $\iota(x, m-n, y) = [x, m, s(x_m), m, y]$ whenever $x,y
\in E^\infty$ satisfy $\varsigma^m(x) = \varsigma^n(y)$.
\end{lem}
\begin{proof}
Reflexivity and symmetry of the relation $\sim$ are clear. For transitivity, suppose that
$(x,m,g,n,y) \sim (x',m',g',n',y')$ and $(x',m',g',n',y') \sim (x'',m'',g'',n'',y'')$. Then $x =
x' = x''$, $y = y' = y''$ and $m-n = m'-n' = m''-n''$. Choose $l \ge n,n'$ and $l' \ge n', n''$
with $g|_{y(n,l)} = g'|_{y(n',l)}$ and $g'|_{y(n',l')} = g''|_{y(n'',l')}$, and put $L = \max\{l,
l'\}$. Then $L \ge n, n''$, and we have
\begin{align*}
g|_{y(n,L)}
    &= \big(g|_{y(n,l)}\big)|_{y(l, L)}
    = \big(g'|_{y(n',l)}\big)_{y(l, L)}
    = g'|_{y(n',L)}\\
    &= \big(g'|_{y(n',l')}\big)|_{y(l',L)}
    = \big(g''|_{y(n'',l')}\big)|_{y(l',L)}
    = g''|_{y(n'',L)}.
\end{align*}

To show that $\Gg_{G, E}$ is a groupoid, we first check that if $[x,m,g,n,y]$ and $[y,p,h,q,z]$
belong to $\Gg_{G, E}$, then so does $[x, m+p, g|_{y(n,n+p)} h|_{z(q,q+n)}, n+q, z]$. For this,
just check:
\begin{align*}
(g|_{y(n,n+p)})(h|_{z(q,q+n)})\cdot \varsigma^{q+n}(z)
    &= (g|_{y(n,n+p)})\cdot \varsigma^n(h\cdot \varsigma^q(z))\\
    &= (g|_{y(n,n+p)})\cdot \varsigma^{n+p}(y)
    = \varsigma^p(g\cdot\varsigma^n(y))
    = \varsigma^{m+p}(x).
\end{align*}
The range and source maps are well-defined by definition. We must check that
multiplication is well-defined. First suppose that $[x,m,g,n,y] =
[x',m',g',n',y']$; so $x = x'$, $y = y'$ and $m - n = m'-n'$, and there
exists $l \ge n, n'$ such that $g|_{y(n,l)} = g'|_{y(n',l)}$. Fix
$[y,p,h,q,z]  \in \Gg_{G, E}$. We must show that
\[
[x, m+p, g|_{y(n,n+p)} h|_{z(q,q+n)}, n+q, z]
    = [x, m'+p, g'|_{y(n',n'+p)} h|_{z(q,q+n')}, n'+q, z].
\]
We have $m+p - (m'+p) = n+q - (n'+q)$. We will show that
\[
\big(g|_{y(n,n+p)} h|_{z(q,q+n)}\big)|_{z(q+n, q+l)}
    = \big(g'|_{y(n',n'+p)} h|_{z(q,q+n')}\big)|_{z(q+n', q+l)}.
\]
We have
\[
\big(g|_{y(n,n+p)} h|_{z(q,q+n)}\big)|_{z(q+n, q+l)}
    = \big((g|_{y(n,n+p)})|_{h|_{z(q,q+n)} \cdot z(q+n, q+l)}\big)h|_{z(q, q+l)},
\]
and similarly
\[
\big(g'|_{y(n',n'+p)} h|_{z(q,q+n')}\big)|_{z(q+n', q+l)}
    = \big((g'|_{y(n',n'+p)})|_{h|_{z(q,q+n')} \cdot z(q+n', q+l)}\big)h|_{z(q, q+l)}.
\]
So we need to check that $(g|_{y(n,n+p)})|_{h|_{z(q,q+n)} \cdot z(q+n, q+l)}
= (g'|_{y(n',n'+p)})|_{h|_{z(q,q+n')} \cdot z(q+n', q+l)}$. For this, we
observe that $h|_{z(q,q+n)} \cdot z(q+n, q+l) = (h \cdot
\varsigma^q(z))(n,l)$, which is equal to $\varsigma^p(y)(n,l)$ because
$[y,p,h,q,z] \in \Gg_{G, E}$. Hence
\begin{align*}
(g|_{y(n,n+p)})|_{h|_{z(q,q+n)} \cdot z(q+n, q+l)}
    &= (g|_{y(n,n+p)})|_{y(n+p, l+p)}
    = (g|_{y(n, l)})|_{y(l, l+p)}
    = (g'|_{y(n', l)})|_{y(l, l+p)}\\
    &= (g'|_{y(n',n'+p)})|_{y(n'+p, l+p)}
    = (g'|_{y(n',n'+p)})|_{h|_{z(q,q+n')} \cdot z(q+n', q+l)}
\end{align*}
as required. A very similar calculation shows that if $[y,p,h,q,z] = [y',p',h',q',z']$
and $[x,m,p,n,y] \in \Gg_{G, E}$, then
\[
[x, m+p, g|_{y(n,n+p)} h|_{z(q,q+n)}, n+q,z]
    = [x, m+p', g|_{y(n,n+p')} h'|_{z(q',q'+n)}, n+q',z],
\]
so multiplication in $\Gg_{G, E}$ is well-defined. It is routine that
$[x,0,r(x),0,x][x,m,g,n,y] = [x,m,g,n,y] = [x,m,g,n,y][y,0,r(y),0,y]$ for all
$[x,m,g,n,y]  \in \Gg_{G, E}$, so that $\Gg_{G, E}$ admits units.

We have
\[
[x,m,g,n,y][y,n,g^{-1},m,x]
    = [x,m+n,g|_{y(n,2n)} g^{-1}|_{x(m,m+n)}, m+n,x].
\]
We calculate:
\[
g^{-1}|_{x(m,m+n)}
    = g^{-1}|_{\varsigma^m(x)(0,n)}
    = g^{-1}|_{g\cdot(\varsigma^n(y)(0,n))}
    = \big(g|_{\varsigma^n(y)(0,n)}\big)^{-1}.
\]
So $[x,m+n,g|_{y(n,2n)} g^{-1}|_{x(m,m+n)}, m+n,x] = [x,m+n,r(x),m+n,x]$. We now have
that $(x,p,r(x),p,x) \sim (x,q,r(x),q,x)$ for all $x,p,q$, and we deduce that
$[y,n,g^{-1},m,x]$ is an inverse for $[x,m,g,n,y]$. Associativity of the multiplication
described follows from straightforward calculations like those above, and we deduce that
$\Gg_{G, E}$ is a groupoid. Using the definition of $\sim$, we see that $[x,m,r(x),n,y]
\sim [x',m',r(x'),n',y']$ if and only if $x = x'$, $y = y'$ and $m-n = m'-n'$, and it
follows that $\iota(x,m-n,y) = [x,m,e,n,y]$ defines a groupoid homomorphism $\Gg_E \to
\Gg_{G, E}$, which is injective by definition of $\sim$.
\end{proof}

We now describe an algebraic isomorphism of $S_{G, E} \ltimes E^\infty$ onto the groupoid $\Gg_{G,
E}$ of Lemma~\ref{lem:groupoid def}, and use it to define an \'etale topology on $\Gg_{G, E}$.

\begin{lem}\label{lem:gpd isomorphism}
Let $E$ be a finite directed graph, and let $(G,E)$ be a self-similar
groupoid action. Let $\Gg_{G, E}$ be the groupoid of Lemma~\ref{lem:groupoid
def}, and let $S_{G, E} \ltimes E^\infty$ be the groupoid of germs described
in Definition~\ref{def:germ gpd}. There is an algebraic isomorphism $\psi :
S_{G, E} \ltimes E^\infty \to \Gg_{G, E}$ such that $\psi([(\mu,g,\nu), \nu
x]) = [\mu (g\cdot x), |\mu|, g, |\nu|, \nu x]$ for all $(\mu, g,\nu)$ in
$S_{G, E}$ and $x \in \rZ{s(\nu)}$. The sets
\[
Z(\mu, g, \nu) := \{[\mu (g\cdot y), |\mu|, g, |\nu|, \nu y] \mid y \in \rZ{s(\nu)}\}
\]
indexed by triples $(\mu, g, \nu)  \in E^* \times G \times E^*$ such that
$s(\mu) = g\cdot s(\nu)$ constitute a basis of compact open sets for a
locally Hausdorff topology on $\Gg_{G, E}$ on which the range and source maps
are homeomorphisms. Under this topology, $\Gg_{G, E}$ is an \'etale groupoid.
\end{lem}
\begin{proof}
Define
\[
\psi^0 : \{\big((\mu,g,\nu),\nu x\big) \mid (\mu,g,\nu) \in S_{E, G}, x \in s(\nu)E^\infty\} \to \Gg_{G, E}
\]
by $\psi^0\big(((\mu,g,\nu), \nu x)\big) = [\mu g\cdot x, |\mu|, g, |\nu|, \nu x]$. We
claim that
\begin{equation}\label{eq:psi def'd}
[(\mu, g, \nu),\nu x] = [(\alpha,h,\beta),\beta y]
    \quad\Longleftrightarrow\quad
        \psi^0\big((\mu,g,\nu),\nu x\big) = \psi^0\big((\alpha,h,\beta),\beta y\big).
\end{equation}
For this, first suppose that $[(\mu,g,\nu),\nu x] = [(\alpha,h,\beta),\beta
y]$. Then $\nu x = \beta y$, and there is an idempotent $(\lambda, e,
\lambda)$ of $S_{G, E}$ such that $x \in \rZ{\lambda}$ and
$(\mu,g,\nu)(\lambda, e, \lambda) = (\alpha,h,\beta)(\lambda, e, \lambda)$.
Without loss of generality, we may assume that $\lambda = x(0,n)$ with $n \ge
|\nu|, |\beta|$.  So $\lambda = \nu\nu' = \beta\beta'$. Since $(\lambda, e,
\lambda) = (\lambda,e,\lambda)(\lambda,e,\lambda) = (\lambda,e^{2},\lambda)$,
we have $e^{2} = e$, so $e = s(\lambda)$. Hence
\[
(\mu g\cdot \nu', g|_{\nu'}, \nu\nu')
    = (\mu,g,\nu)(\lambda, e, \lambda)
    = (\alpha,h,\beta)(\lambda, e, \lambda)
    = (\alpha h\cdot \beta', h|_{\beta'}, \beta\beta').
\]
In particular, $g|_{(\nu x)(|\nu|, |\lambda|)} = g|_{\nu'} = h|_{\beta'} = h|_{(\alpha
y)(|\alpha|,|\lambda|)}$. Also, $\mu (g\cdot x) = \mu (g\cdot \nu') g|_{\nu'}\cdot
\varsigma^{|\nu'|}(x) = \alpha (h\cdot \beta') h|_{\beta'} \cdot \varsigma^{|\beta'|}(y)
= \alpha (h\cdot y)$. Hence
\begin{align*}
\psi^0\big((\mu, g, \nu),\nu x\big)
    &= [\mu (g\cdot x), |\mu|, g, |\nu|, \nu x]\\
    &= [\mu (g\cdot x), |\mu| + |\nu'|, g|_{(\nu x)(|\nu|, |\nu| + |\nu'|)}, |\nu| + |\nu'|, \nu x]\\
    &= [\alpha (h\cdot y), |\alpha| + |\beta'|, h|_{(\beta y)(|\beta|, |\beta| + |\beta'|)}, |\beta| + |\beta'|, \beta y]\\
    &= \psi^0\big((\alpha,h,\beta), \beta y\big).
\end{align*}

Conversely suppose that $\psi^0\big((\mu,g,\nu),\nu x\big) =
\psi^0\big((\alpha,h,\beta),\beta y\big)$. Then $|\mu| - |\nu| = |\alpha| -
|\beta|$, $\nu x = \beta y$, $\mu g\cdot x = \alpha h \cdot y$, and there
exists $l \ge |\nu|, |\beta|$ such that $g|_{x(0,l-|\nu|)} = h|_{y(0,
l-|\beta|)}$. Let $e = s(y_{l-|\beta|}) = s(x_{l-|\nu|})$. Then
\begin{align*}
(\alpha, h, \beta)&(\beta y(0, l-|\beta|), e, \beta y(0, l-|\beta|))\\
    &= (\alpha h\cdot y(0, l-|\beta|), h|_{y(0, l-|\beta|)}, (\beta y)(0,l))\\
    &= (\mu g\cdot x(0, l-|\nu|), g|_{x(0, l-|\nu|)}, (\nu x)(0,l))\\
    &= (\mu, g, \nu)(\nu x(0, l-|\nu|), e, \nu x(0, l-|\nu|)).
\end{align*}
Since $\nu x(0, l-|\nu|) = (\nu x)(0,l) = (\beta y)(0,l) = \beta y(0, l-|\beta|)$, we obtain
\[
(\alpha, h, \beta)(\beta y(0, l-|\beta|), e, \beta y(0, l-|\beta|))
    = (\mu, g, \nu)(\beta y(0, l-|\beta|), e, \beta y(0, l-|\beta|)).
\]
Hence $[(\mu,g,\nu),\nu x] = [(\alpha,h,\beta),\beta y]$. This completes the
proof of~\eqref{eq:psi def'd}.

It follows that $\psi^0$ descends to an injective map $\psi : S_{G, E} \ltimes E^\infty
\to \Gg_{G, E}$. To see that $\psi$ is surjective, fix $[x,m,g,n,y] \in \Gg_{G, E}$, let
$z := \varsigma^{n}(y)$, $\nu = y(0,n)$ and $\mu = x(0,m)$, and observe that $[x,m,g,n,y]
= [\mu g\cdot z, |\mu|, g, |\nu|, \nu z] = \psi([(\mu, g, \nu), \nu z])$. It is routine
to check that $\psi$ is multiplicative, and hence an algebraic isomorphism of groupoids
as claimed.

For $(\mu, g, \nu)$ with $s(\mu) = g\cdot s(\nu)$ and an open set $U \subseteq \rZ{\nu}$ in
$E^\infty$, let
\[
Z(\mu,g,\nu, U) := \{[\mu g\cdot x, |\mu|, g, |\nu|, \nu x] \mid \nu x \in U\}.
\]
Proposition~4.14 of \cite{Exel:Inverse_combinatorial} combined with the
algebraic isomorphism $\psi$ above shows that the sets $Z(\mu, g, \nu, U)$
are a basis for a topology on $\Gg_{G, E}$ under which it becomes a
topological groupoid. If $[\mu g\cdot x, |\mu|, g, |\nu|, \nu x]$ is in
$Z(\mu, g, \nu, U)$, then by definition of the topology on $E^\infty$ there
exists $n \in \mathbb{N}$ for which $\nu' := x(0, n)$ is a path such that
$\nu x \in \rZ{\nu\nu'} \subseteq U$. Then, we have
\[
[\mu g\cdot x, |\mu|, g, |\nu|, \nu x]
    \in Z(\mu g\cdot \nu', g|_{\nu'}, \nu\nu')
    = Z(\mu, g, \nu, Z(\nu'))
    \subseteq Z(\mu, g, \nu, U).
\]
So the $Z(\mu, g, \nu)$ are a basis for the same topology as the $Z(\mu, g, \nu, U)$. Now
Proposition~4.15 of \cite{Exel:Inverse_combinatorial} shows that the range and source maps
restrict to homeomorphisms $r : Z(\mu, g, \nu) \to \rZ{\mu}$ and $s : Z(\mu, g, \nu) \to
\rZ{\nu}$. Since the $\rZ{\nu}$ are compact and Hausdorff, we deduce that the $Z(\mu, g, \nu)$ are
also compact and Hausdorff. It follows that $\Gg_{G, E}$ is locally Hausdorff and \'etale as
claimed.
\end{proof}

We now show that the $C^*$-algebra of the groupoid $\Gg_{G, E}$ just constructed coincides with
the $C^*$-algebra of the self-similar action $(G, E)$.

Note that for each $\mu  \in E^*$ and $g  \in G$ with $c(g) = s(\mu)$, we
have $(\mu, g, d(g)) \in S_{G, E}$, and $\Dom((\mu, g, d(g))) = \rZ{d(g)}$.
Hence $W((\mu, g, d(g)), \rZ{\lambda})$ is a compact open subset of $S_{G, E}
\ltimes E^\infty$ for each $\lambda  \in d(g)E^*$.

\begin{prp}\label{prp:O,C*G isomorphism}
Let $E$ be a finite directed graph with no sources, and let $(G, E)$ be a self-similar
groupoid action. Let $\Gg_{G, E}$ be the groupoid described in Lemmas
\ref{lem:groupoid def}~and~\ref{lem:gpd isomorphism}. There is an isomorphism
$\pi : \Oo(G, E) \to C^*(\Gg_{G, E})$ such that $\pi(u_g) = 1_{Z(c(g), g,
d(g))}$ and $\pi(s_e) = 1_{Z(e, s(e), s(e))}$ for all $g \in G$ and $e \in
E^1$.
\end{prp}
\begin{proof}
By Lemma~\ref{lem:gpd isomorphism}, it suffices to construct an isomorphism $\pi : \Oo(G, E) \to
C^*(S_{G, E} \ltimes E^\infty)$ such that each $\pi(u_g) = 1_{W((c(g), g, d(g)), \rZ{d(g)})}$, and
each $\pi(s_e) = 1_{W((e, s(e), s(e)), \rZ{s(e)})}$.

For $g  \in G$, $e  \in E^1$ and $v  \in E^0$, define
\[
U_g := 1_{W((c(g), g, d(g)), \rZ{d(g)})}, \quad
    S_e := 1_{W((e, s(e), s(e)), \rZ{s(e)})}, \quad\text{ and }\quad
    P_v := 1_{W((v, v, v), \rZ{v})}.
\]
Elementary calculations using the definition of multiplication in $S_{G, E}$
shows that $(U, P, S)$ is a covariant representation of $(G, E) \in C^*(S_{G,
E} \ltimes E^\infty)$. It follows that there is a homomorphism $\pi : \Oo(G,
E) \to C^*(S_{G, E} \ltimes E^\infty)$ satisfying the given formulas. For
each $\lambda  \in E^*$, write $\lambda = \lambda_1 \lambda_2 \dots
\lambda_n$ as a concatenation of edges, and then define $S_\lambda :=
S_{\lambda_1} \dots S_{\lambda_n} $. Then
\[
1_{\rZ{\lambda}} = S_\lambda S^*_\lambda \in \pi(\Oo(G, E)).
\]
Since the $\rZ{\lambda}$ constitute a basis for the topology on $E^\infty$, it follows that
$C_0(E^\infty) \subseteq \pi(\Oo(G, E))$. If $V$ is a compact open bisection in $S_{G, E} \ltimes
E^\infty$, we can write it as a finite disjoint union of bisections of the form $W((\mu, g, \nu),
\rZ{\nu\alpha})$. We have
\[
1_{W((\mu,g,\nu), \rZ{\nu\alpha})}
    = 1_{W((\mu g \cdot\alpha, g|_\alpha, \nu\alpha), \rZ{s(\alpha)})}
    = S_{\mu g\cdot \alpha} U_{g|_\alpha} S^*_{\nu\alpha}
    \in \pi(\Oo(G, E)),
\]
and we deduce that the indicator function of each compact open bisection belongs to the range of
$\pi$. For each compact open bisection $V$, indicator functions of this form linearly span a dense sub-algebra of $C_{c}(V)$. It follows that $\pi$ is
surjective.

It remains to show that $\pi$ is injective. To do this, it suffices to construct a right
inverse $\rho : C^*(S_{G, E} \ltimes E^\infty) \to \Oo(G, E)$ for $\pi$. Observe that
since $S_{G, E} \ltimes E^\infty$ is the groupoid of germs of the action $\theta$ of the
inverse semigroup $S_{G, E}$ on $E^\infty$,
\cite[Theorem~8.5]{Exel:Inverse_combinatorial} shows that $C^*(S_{G, E} \ltimes
E^\infty)$ is universal for representations, as defined in
\cite[Definition~8.1]{Exel:Inverse_combinatorial} of $(\theta, S_{G, E}, E^\infty)$.
Since the $s_e$ and $p_v$ constitute a Cuntz--Krieger $E$-family in $\Oo(G, E)$, there is
a homomorphism $\rho_0 : C_0(E^\infty) \to \Oo(G, E)$ such that $\rho_0(1_{\rZ{\lambda}})
= s_\lambda s^*_\lambda$ for each $\lambda$. For each $(\mu, g, \nu)  \in S_{G, E}$
define $S_{(\mu,g,\nu)} := s_\mu u_g s^*_\nu$. If $\lambda = \nu\lambda'$ then the
relations in $\Oo(G, E)$ give
\[
S_{(\mu,g,\nu)} \rho_0(1_{\rZ{\lambda}}) S_{(\mu,g,\nu)}^*
    = s_\mu u_g s^*_\nu s_\lambda s^*_\lambda s_\nu u_{g^{-1}} s^*_\mu
    = s_{\mu g\cdot\lambda'} s^*_{\mu g\cdot\lambda'},
\]
and then linearity and continuity imply that $S_{(\mu,g,\nu)} \rho_0(f)
S_{(\mu,g,\nu)}^* = \rho_0(f \circ \theta_{(\mu,g,\nu)^*})$ whenever $f$ is
supported on $\Dom(\theta_{(\mu,g,\nu)})$. Routine calculations show that
$S_a S_b = S_{ab}$ and $S_{a^*} = S^*_a$ for all $a,b  \in S_{G, E}$. So
$(\rho_0, S)$ is a representation of $(\theta, S_{G, E}, E^\infty)$ and it
follows that there is a homomorphism $\rho : C^*(S_{G, E} \ltimes E^\infty)
\to \Oo(G, E)$ such that $\rho(f) = \rho_0(f)$ for all $f $ in
$C_0(E^{\infty})$ and $\rho(1_{W((\mu,g,\nu), \rZ{\nu})}) = s_\mu u_g
s^*_\nu$ for all $(\mu,g,\nu)  \in S_{G, E}$. In particular, $\rho \circ \pi$
fixes the generators of $\Oo(G, E)$ and therefore $\rho \circ \pi =
\id_{\Oo(G, E)}$ as required.
\end{proof}

We conclude this section by characterising exactly when $\Gg_{G, E}$ is Hausdorff.

\begin{prp}\label{prp:non-Hausdorff points}
Let $E$ be a finite directed graph, and let $(G, E)$ be a self-similar
groupoid action. Let $\Gg_{G, E}$ be the groupoid of Lemmas \ref{lem:groupoid
def}~and~\ref{lem:gpd isomorphism}. Points $[x, m, g, n, y]$ and $[w, p, h,
q, z] \in \Gg_{G, E}$ are distinct but cannot be separated by disjoint open
sets if and only if both of the following hold
\begin{enumerate}
    \item\label{it:nonHausdorff1} $x = w$, $y = z$, $m-n = p-q$ and $g|_{y(n,l)} \not=
        h|_{y(q,l)}$ for all $l \ge n,q$; and
    \item\label{it:nonHausdorff2} for every $l \ge n,q$ there exists
        $\lambda  \in s(y_{l})E^*$ such that $g|_{y(n,l)} \cdot \lambda =
        h|_{y(q, l)} \cdot \lambda$ and $g|_{y(n,l)\lambda} =
        h|_{y(q,l)\lambda}$.
\end{enumerate}
\end{prp}
\begin{proof}
First, suppose that $[x,m,g,n,y]$ and $[w,p,h,q,z]$ are distinct but cannot
be separated by open neighbourhoods.
\par
The range map $r$, the source map $s$ and the co-cycle map
$c:\mathcal{G}_{(G,E)}\mapsto\mathbb{Z}$, defined by $c([x,m,g,n,y]) = m-n$,
are continuous mappings onto Hausdorff spaces, so since $[x,m,g,n,y]$ and
$[w,p,h,q,z]$ cannot be separated by open neighbourhoods, their images under
$r,s$ and $c$ coincide. Hence $x = w$, $y = z$ and $m-n = p - q$. Since
$[x,m,g,n,y] \not= [w,p,h,q,z]$, the definition of the equivalence relation
of Lemma~\ref{lem:groupoid def} forces $g|_{y(n,l)} \neq h|_{y(q,l)}$ for all
$l\geq \text{max}(n,q)$. Therefore, $(1)$ is satisfied.

For~$(2)$, let $\xi = y(n,l)$, and $\tau = y(q,l)$. Then $[x,m,g,n,y] \in Z(x(0,m) g
\cdot\xi, g|_\xi, y(0,n)\xi)$, and $[x, p, h, q, y]$ is in $Z(x(0,p) h\cdot\tau, h|_\tau,
y(0,q)\tau)$. Let $\gamma := y(0,n)\xi = y(0,q)\tau$ and $\omega := x(0,m) g\cdot\xi =
x(0,p) h\cdot \tau$. By assumption,  $Z(\omega, g|_\xi, \gamma)\cap Z(\omega, h|_\tau,
\gamma)\neq\emptyset,$ so there exist $u \in s(\gamma)E^{\infty}$ and $v \in
s(\omega)E^{\infty}$ such that $g|_{\xi}\cdot u = h|_{\tau}\cdot u = v$ and $[\omega v,
|\omega|, g|_{\xi}, |\gamma|, \gamma u] = [\omega v, |\omega|, h|_{\tau}, |\gamma|,
\gamma u].$ Hence, there exists $k$ in $\mathbb{N}$ such that $(g|_{\xi})|_{u(0,k)} =
(h|_{\tau}|)_{u(0,k)}$. Thus $\lambda := u(0,k) \in s(y_{l})E^{*}$ satisfies~$(2)$.

Now suppose that (\ref{it:nonHausdorff1})~and~(\ref{it:nonHausdorff2}) hold. The last
part of~(\ref{it:nonHausdorff1}) and the definition of $\sim$ imply that $[x,m,g,n,y]
\not= [w,p,h,q,z]$. Fix neighbourhoods $U \owns [x,m,g,n,y]$ and $V \owns [w,p,h,q,z]$.
We must show that $U \cap V \not= \varnothing$. By definition of the topology, there
exists $l \ge n,q$ such that $Z(x(0,m-n+l), g|_{y(n,l)}, y(0,l)) \subseteq U$  and
$Z(x(0, p-q+l), h|_{y(q,l)}, y(0,l)) \subseteq V$. Fix $\lambda$ as in
condition~(\ref{it:nonHausdorff2}) for this $l$. Then
\begin{align*}
U &\owns \big[x(0, m - n + l) g|_{y(n,l)} \cdot (\lambda z), m-n+l + |\lambda|, g|_{y(n,l)\lambda}, l + |\lambda|, y(0,l)\lambda z\big]\\
    &= \big[x(0, p-q+l)h|_{y(q,l)} \cdot (\lambda z), p-q+l + |\lambda|, h|_{y(q,l)\lambda)}, l + |\lambda|, y(0,l)\lambda z\big] \in V
\end{align*}
for all $z  \in \rZ{s(\lambda)}$, so $U \cap V \not= \emptyset$.
\end{proof}

For the following corollary, we use the following terminology adapted from
\cite[Definition~5.2]{Exel-Pardo:Self-similar}: if $(G,E)$ is a self-similar groupoid
action on a finite graph $E$, we say that a path $\lambda \in E^*$ is \emph{strongly
fixed} by an element $g \in G r(\lambda)$ if $g\cdot\lambda = \lambda$ and $g|_{\lambda}
= s(\lambda)$.

\begin{cor}\label{cor:Hausdorffness}
Let $E$ be a finite directed graph, and let $(G, E)$ be a self-similar
groupoid action. Then the following are equivalent.
\begin{enumerate}
    \item\label{it:Hchar1} The groupoid $\Gg_{G, E}$ is Hausdorff.
    \item\label{it:Hchar2} The subgroupoid $\Gg_{G, E}^\TT := \{[x,m,g,n,y] \in \Gg_{G, E} \mid
        m = n\}$ is Hausdorff.
    \item\label{it:Hchar3} The subgroupoid $\{[g \cdot x, 0, g, 0, x] \mid x \in E^\infty\text{
        and }d(g) = r(x)\}$ is Hausdorff.
    \item\label{it:Hchar4} If $g  \in G$ and $y  \in E^\infty$ satisfy $g \cdot y = y$
        and $g|_{y(0,n)} \not= s(y_n)$ for all $n$, then there exists $\lambda  \in
        E^*$ such that $y \in \rZ{\lambda}$ and no element of $\lambda E^*$ is strongly
        fixed by $g$.
\end{enumerate}
In particular, if $(G, E)$ is regular then $\Gg_{G, E}$ is Hausdorff.
\end{cor}

\begin{proof}
Since $\mathcal{G}_{G,E}^{\mathbb{T}}\subseteq\mathcal{G}_{G,E}$ and $\mathcal{G}_{G,E}^{\mathbb{T}}$ contains the subgroupoid of $(\ref{it:Hchar3})$, we have
(\ref{it:Hchar1})$\;\implies\;$(\ref{it:Hchar2})$\;\implies\;$(\ref{it:Hchar3}).

For (\ref{it:Hchar3})$\;\implies\;$(\ref{it:Hchar4}), we prove the contrapositive. So
suppose that~(\ref{it:Hchar4}) fails with respect to $g \in G$ and $y  \in \rZ{d(g)}$.
That is $g \cdot y = y$ and $g|_{y(0,n)} \not= s(y_n)$ for all $n$, but for every
$\lambda$ such that $y \in \rZ{\lambda}$ there exists $\mu \in s(\lambda)E^*$ such that
$\lambda\mu$ is strongly fixed by $g$. Equivalently, for every $l \ge 0$, there exists
$\mu \in s(y_l)E^*$ such that $g\cdot (y(0,l)\mu) = y(0,l)\mu$ and $g|_{y(0,l)\mu} =
s(\mu)$. Hence Proposition~\ref{prp:non-Hausdorff points} implies that $[y, 0, g, 0, y]$
and $[y, 0, d(g), 0, y]$ are distinct but cannot be separated by open sets. Hence the
open subgroupoid $\{[g \cdot x, 0, g, 0, x] \mid x \in E^\infty\text{ and }d(g) = r(x)\}$
is not Hausdorff.

For (\ref{it:Hchar4})$\;\implies\;$(\ref{it:Hchar1}), suppose
that~(\ref{it:Hchar4}) holds. Fix $[x, m, g, n, y]$ and $[w, p, h, q, z] \in
\Gg_{G, E}$. It suffices to show that the conditions of
Proposition~\ref{prp:non-Hausdorff points} do not both hold for these points.
To do this, we suppose that Proposition~\ref{prp:non-Hausdorff
points}(\ref{it:nonHausdorff1}) holds, and show that
Proposition~\ref{prp:non-Hausdorff points}(\ref{it:nonHausdorff2}) fails. Let
$l := \max{(n,q)}$, let $k := (g|_{y(n,l)})^{-1} h|_{y(q,l)}$, and put $y' =
\varsigma^l(y)$. Since $x = w$, we have $(g|_{y(n,l)}) \cdot y' = h|_{y(q,l)}
\cdot y'$ and so $k \cdot y' = y'$. Moreover, $k|_{y'(0,a)} =
\big((g|_{y(n,l)})^{-1}h|_{y(q,l)}\big)|_{y'(0,a)} =
\big((g|_{y(n,l+a)})^{-1}h|_{y(q,l+a)}\big) \neq s(y_{(l+a)})$ for all $a$.
So~(\ref{it:Hchar4}) implies that for large $a  \in \NN$ and $\lambda \in
s(y'_{a})E^*$, either $k|_{y'(0,a)} \cdot \lambda \not= \lambda$ or
$k|_{y'(0,a)\lambda} \not= s(\lambda)$. That is, for large $a$, for every
$\lambda \in s(y_{a})E^*$, either $g|_{y(n,l+a)} \cdot \lambda \not= h|_{y(q,
l+a)}\cdot \lambda$, or $g|_{y(n,l+a)\lambda} \not= h|_{y(q, l+a)\lambda}$.
Thus condition~(\ref{it:nonHausdorff2}) of Proposition~\ref{prp:non-Hausdorff
points} fails for $[x, m, g, n, y]$ and $[w, p, h, q, z]$ as required.

For the final statement, we show that if $(G, E)$ is regular,
then~(\ref{it:Hchar4}) holds. Suppose that $g  \in G$ and $y  \in E^\infty$
satisfy $g \cdot y = y$. Regularity gives $n \in \NN$ such that $g$ pointwise
fixes $\rZ{y(0,n)}$. Hence $g|_{y(0,n)}$ pointwise fixes $\rZ{y(n)}$. Since
self-similar groupoid actions are, by definition, faithful this implies that
$g|_{y(0,n)} = y(n) \in G^{(0)}$. So the hypothesis of~(\ref{it:Hchar4}) is
never satisfied, and so~(\ref{it:Hchar4}) holds vacuously.
\end{proof}

We can characterise regularity of $(G,E)$ in terms of of $\mathcal{G}_{G,E}$.
Recall that a groupoid $\mathcal{G}$ is \textit{principal} if its isotropy
bundle $\text{Iso}(\mathcal{G}) = \{g\in\mathcal{G}:r(g) = s(g)\}$ is equal
to the unit space $\mathcal{G}^{(0)}$.
\begin{prp}
Let $(G,E)$ be a self-similar groupoid action on a finite directed graph $E$. Then, $(G,E)$ is regular if and only if $\mathcal{G}^{\mathbb{T}}_{G,E}$ is principal.
\end{prp}
\begin{proof}
Suppose that $(G,E)$ is regular. Fix $\gamma \in
\text{Iso}(\mathcal{G}_{G,E})$. Then $\gamma = [x,n,g,n,x]$ for some $x \in
E^{\infty}$, $n \ge 0$ and $g \in G$ such that $d(g) = r(\varsigma^{n}(x))$
and $g\cdot \varsigma^{n}(x) = \varsigma^{n}(x)$. By regularity, there exists
$k \in \mathbb{N}$ such that $g|_{x(n+1, n+k)}= s(x_{n+k})$. By definition of
the equivalence relation $\sim$ defining $\mathcal{G}_{G, E}$ (see
Lemma~\ref{lem:groupoid def}), we have $[x,n,g,n,x] = [x,n+k,s(x_{n+k}), n+k,
x] \in \mathcal{G}_{G,E}^{(0)}$. Therefore, $\mathcal{G}_{G,E}^{\mathbb{T}}$
is principal.

Now, suppose that $\mathcal{G}_{G,E}^{\mathbb{T}}$ is principal. If $g \in G$
and $x \in E^{\infty}$ satisfy $d(g) = r(x)$ and $g\cdot x = x$, then
$[x,0,g,0,x] \in \text{Iso}(\mathcal{G}_{G,E})$. Therefore, $(x,0,g,0,x) \sim
(x,n, s(x_{n}),n,x)$ for some $n$. Hence there exists $k\geq n$ such that
$g|_{x(0,k)} = s(x_{n})|_{x(n,k)} = s(x_{k})$. Therefore, $(G,E)$ is regular.
\end{proof}

\section{The dual algebra of a self-similar graph}\label{sec:hO(G,E)}

We now describe a second $C^*$-algebra associated to a contracting, regular
self-similar groupoid action on a finite directed graph with no sources; namely the $C^*$-algebra of the
Deaconu--Renault groupoid of the homeomorphism $\Jsig : \Jj \to \Jj$ of Section~\ref{sec:J
dynamics}.
%We draw heavily on the analysis and techniques of
%\cite[Section~6]{Nekrashevych:Self-similar}.

Recall that if $X$ is a locally compact Hausdorff space, and $T : X \to X$ is a local
homeomorphism, then $\Gg_{X, T}$ is the set
\[
\Gg_{X, T} := \{(x, m-n, y) \in X \times \ZZ \times X \mid m,n\in\mathbb{N}_{0}, T^m(x) = T^n(y)\},
\]
endowed with the topology arising from the basic open sets
\[
Z(U, m, n, V) := \{(x, m-n, y) \in U
\times \{m-n\} \times V \mid T^m(x) = T^n(y)\}.
\]
The unit space is $\Gg^{(0)}_{X, T} := \{(x, 0, x) \mid x \in X\}$ and is
identified with $X$. The groupoid structure is given by
\begin{gather*}
r(x, p, y) := x, \qquad s(x, p, y) := y,\\
(x,p,y)(y,q,z) := x(p+q,z)\qquad\text{and}\qquad (x, p, y)^{-1} := (y, -p, x).
\end{gather*}

It is not hard to check that
\[
\{Z(U, m, n, V) \mid T^m|_U\text{ and }T^n|_V\text{ are homeomorphisms and } T^m(U) = T^n(V)\}
\]
is a basis of open bisections for the topology, so $\Gg_{X, T}$ is \'etale. It is easy to see that
it is Hausdorff, and it is locally compact because $X$ is.

\begin{dfn}\label{dfn:dual alg}
Let $E$ be a finite directed graph with no sources. Let $(G, E)$ be a contracting,
regular self-similar groupoid action. Let $\Jj$ and $\Jsig$ be the space and local homeomorphism
of Definition~\ref{dfn:limit space} and Proposition~\ref{thm:Jsig}. We define $\hG_{G, E}$ to be
the Deaconu--Renault groupoid $\hG_{G, E} := \Gg_{\Jj, \Jsig}$, and we define $\hO(G, E) :=
C^*(\hG_{G, E})$, and call this the \emph{dual $C^*$-algebra} of $(G, E)$.
\end{dfn}

Let $\mathcal{G}_{\sigma}$ be the Deaconu--Renault groupoid associated to
$\sigma:E^{-\infty}\mapsto E^{-\infty}$. Recall from Section~\ref{sec:limit
space} the quotient map $q:E^{-\infty}\mapsto\mathcal{J}$. Since
$q\circ\tilde{\sigma} = \sigma\circ q$, we see that $q$ extends to a groupoid
homomorphism $q:\mathcal{G}_{\sigma}\mapsto \hat{\mathcal{G}}_{G,E}$ which
sends $(x,k,y) \in \mathcal{G}_{\sigma}$ to $q(x,k,y) = (q(x), k, q(y))$. The
next result will allow us to deduce properties of $\hat{\mathcal{G}}_{G, E}$
from those of $\mathcal{G}_\sigma$.

\begin{prp}\label{prp:s-bijective}
Let $E$ be a finite directed graph with no sources and $(G,E)$ a contracting
regular self-similar groupoid action. For every $y \in E^{-\infty}$,
$q:(\mathcal{G}_{\sigma})y\mapsto (\hat{\mathcal{G}}_{G,E})q(y)$ is a
bijection, and $q:\mathcal{G}_{\sigma}\mapsto \hat{\mathcal{G}}_{G,E}$ is
proper.
\end{prp}
\begin{proof}
Suppose $z \in \mathcal{J}$ and $m,n \in \mathbb{N}$ satisfy
$\tilde{\sigma}^{n}(q(y)) = \tilde{\sigma}^{m}(z):= w$.
Proposition~\ref{prp:mapping_properties}$(1)$ implies that
$\sigma^{m+k}:q^{-1}(z)\mapsto q^{-1}(\tilde{\sigma}^{k}(w))$ and
$\sigma^{n+k}:q^{-1}(q(y))\mapsto q^{-1}(\tilde{\sigma}^{k}(w))$ are
bijective for all $k \ge 0$. Thus there is a unique $x'$ in $q^{-1}(z)$ such
that $\sigma^{m+k}(x') = \sigma^{n+k}(y)$, for all $k \ge 0$. Hence,
$(x',m-n,y)$ is the unique element of $(\mathcal{G}_{\sigma})y$ such that
$q((x',m-n,y)) = (z,m-n,q(y))$.

The same uniqueness property shows that if $x,y \in E^{-\infty}$ and $m,n,k
\ge 0$ satisfy $\sigma^{m+k}(x) = \sigma^{n+k}(y)$ and
$\tilde{\sigma}^{m}(q(x)) = \tilde{\sigma}^{n}(q(y))$, then $\sigma^{m}(x) =
\sigma^{n}(y)$. It follows that $q^{-1}(Z(\mathcal{J},m,n,\mathcal{J})) =
Z(E^{-\infty},m,n,E^{-\infty})$ for any $m,n \ge 0$. Since these sets form
compact open coverings of the respective groupoids, $q$ is proper.
\end{proof}

We investigate when $\hO(G, E)$ is simple. A groupoid $\Gg$ is \emph{minimal}
if $\{r(\gamma) \mid s(\gamma) = x\}$ is dense in $\Gg^{(0)}$ for every $x
\in \Gg^{(0)}$, and that an \'etale groupoid $\Gg$ is \emph{effective} if the
interior of $\text{Iso}(\mathcal{G}) =\{\gamma \in \Gg \mid r(\gamma) =
s(\gamma)\}$ is $\mathcal{G}^{(0)}$.

\begin{lem}
Let $E$ be a finite directed graph. Let $(G, E)$ be a contracting, regular
self-similar groupoid action. If $E$ is strongly connected and not a simple
cycle, then $\hG_{G, E}$ is minimal and effective, and $\hO(G, E)$ is simple.
\end{lem}
\begin{proof}
Fix $x \in E^{-\infty}$. Then $\{x\mu\mid\mu\in E^{*}\text{ and }s(x) =
r(\mu)\}\subseteq\{r(\gamma)\mid\gamma\in\mathcal{G}_{\sigma}\text{ and
}s(\gamma) = x\}$. Since $E$ is strongly connected, for every $\nu \in
E^{*}$, there is $\mu \in E^{*}$ such that $s(\mu) = r(\nu)$ and $r(\mu) =
s(x)$. Hence $x\mu\nu \in E^{-\infty}$. Therefore, $Z(\nu)\cap\{r(\gamma)\mid
s(\gamma) = x\}\neq\emptyset$ and consequently $\mathcal{G}_{\sigma}$ is
minimal.

By Proposition~\ref{prp:s-bijective}, for every $x \in E^{-\infty}$,
$q:(\mathcal{G}_{\sigma})x\mapsto (\hat{\mathcal{G}}_{G,E})q(x)$ is
surjective. Hence, $q(r((\mathcal{G}_{\sigma})x) =
r((\hat{\mathcal{G}}_{G,E})q(x))$. Since $q$ is surjective and continuous,
$r((\hat{\mathcal{G}}_{G,E})q(x))$ is dense whenever
$r((\mathcal{G}_{\sigma})x)$ is dense. Therefore, minimality of
$\mathcal{G}_{\sigma}$ implies minimality of $\hat{\mathcal{G}}_{G,E}$.

To see that $\hat{\mathcal{G}}_{G,E}$ is effective, suppose that $w \in
\mathcal{J}$ satisfies $\tilde{\sigma}^{m}(w) = w$. By
Proposition~\ref{prp:mapping_properties}(1), $\sigma^{m}$ maps $q^{-1}(w)$
bijectively onto $q^{-1}(w)$. Since $q^{-1}(w)$ is finite, there exists $k
\in \mathbb{N}$ such that $\sigma^{mk}(x) = x$ for all $x \in q^{-1}(w)$. Let
$P_{\sigma}$ and $P_{\tilde{\sigma}}$ denote the sets of periodic points for
$\sigma$ and $\tilde{\sigma}$ respectively. Then $q^{-1}(P_{\tilde{\sigma}})
= P_{\sigma}$. Hence, $q^{-1}(\bigcup_{n\geq
0}\tilde{\sigma}^{-n}(P_{\tilde{\sigma}})) = \bigcup_{n\geq
0}\sigma^{-n}(P_{\sigma}).$ We have $P_\sigma = \bigcup_{\lambda \in E^*}
\bigcup_{\mu \in s(\lambda)E^* s(\lambda)\setminus\{\lambda\}}
\{\lambda\mu^\infty\}$, and so $P_\sigma$, and hence $\bigcup_{n \ge 0}
\tilde\sigma^{-n}(P_{\tilde\sigma})$, is countable.

If $g \in \hat{\mathcal{G}}_{G,E}\setminus\mathcal{J}$ satisfies $r(g) =
s(g)$, then $r(g) \in \bigcup_{n\geq
0}\tilde{\sigma}^{-n}(P_{\tilde{\sigma}})$. Thus
$r\big(\text{Iso}(\hat{\mathcal{G}}_{G,E})\setminus\mathcal{J}\big)$ is
countable. Since $r$ is an open map, to show
$\text{Iso}(\hat{\mathcal{G}}_{G,E})\setminus\mathcal{J}$ has empty interior,
it suffices to show no countable set in $\mathcal{J}$ is open. By the Baire
Category Theorem, it suffices to show that $\mathcal{J}$ has no isolated
points. Since $E$ is strongly connected and not a simple cycle, every open
subset of $E^{-\infty}$ is infinite. By continuity and surjectivity of $q$,
the preimage of every nonempty open subset of $\Jj$ is open and hence
infinite. Since $q$ is finite-to-one, it follows that no singleton in $\Jj$
is open. Hence $\mathcal{J}$ has no isolated points, and consequently
$\hat{\mathcal{G}}_{G,E}$ is effective.

It now follows from \cite[Theorem~5.1]{BCFS} that $\hO(G, E) = C^*(\hG_{G,
E})$ is simple.
\end{proof}

\section{$KK$-duality via Smale spaces}\label{sec:poincare}

In this section we establish our $KK$-duality result. We do this using a
general result of Kaminker--Putnam--Whittaker
\cite{Kaminker-Putnam-Whittaker:K-duality}, which says that the stable and
unstable Ruelle algebras of any irreducible Smale space are $KK$-dual. We
show that the stable Ruelle algebra of the Smale space of
Section~\ref{sec:smale space} is Morita equivalent to the $C^*$-algebra
$\hO(G, E)$ of Section~\ref{sec:hO(G,E)}, and that the unstable Ruelle
algebra is Morita equivalent to the $C^*$-algebra $\Oo(G, E)$ of
Section~\ref{sec:O(G,E)}. This, combined with the duality of the Ruelle
algebras, gives our main result.

\begin{thm}\label{thm:main duality thm}
Let $E$ be a strongly connected finite directed graph. Let $(G, E)$ be a contracting, regular
self-similar groupoid action. Then $\Oo(G, E)$ and $\hO(G, E)$ are $KK$-dual
in the sense that there are classes $\mu  \in KK^1(\Oo(G, E) \otimes \hO(G,
E), \CC)$ and $\beta \in KK^1(\CC, \Oo(G, E) \otimes \hO(G, E))$ such that
\[
\beta \mathbin{\widehat{\otimes}}_{\Oo(G, E)} \mu = \id_{KK(\hO(G, E), \hO(G, E))}
    \quad\text{ and }\quad
\beta \mathbin{\widehat{\otimes}}_{\hO(G, E)} \mu = -\id_{KK(\Oo(G, E), \Oo(G, E))}.
\]
In particular, $K^*(\Oo(G, E)) \cong K_{*+1}(\hO(G, E))$ and $K_*(\Oo(G, E)) \cong K^{*+1}(\hO(G,
E))$.
\end{thm}

\subsection{The stable algebra}\label{subsec:stable ME}

We will show that the groupoid $\hG_{G, E}$ of Lemmas \ref{lem:groupoid def}~and~\ref{lem:gpd
isomorphism} is equivalent to the stable Ruelle groupoid $G^s \rtimes \ZZ$ of the Smale space
$\Ss$ of Section~\ref{sec:smale space}. The idea is to show that in fact $G^s \rtimes \ZZ$ is
equal to the amplification of $\hG_{G, E}$ with respect to the surjection $\tilde\pi : \Ss \to \Jj$
induced by the natural surjection of $E^\ZZ$ onto $E^{-\infty}$. For this, we first need to show
that this $\tilde\pi$ makes sense and is an open map.

\begin{lem}\label{lem:pi tilde}
Let $E$ be a finite directed graph with no sources. Let $(G, E)$ be a
contracting regular self-similar groupoid action. Let $\Jj$ be the limit
space of Definition~\ref{dfn:limit space}, and let $\Ss$ be the limit
solenoid of Definition~\ref{def:S space}. There is a continuous open
surjection $\tilde\pi : \Ss \to \Jj$ such that $\tilde\pi([x]) = [\dots
x_{-3} x_{-2} x_{-1}]$ for all $x \in E^\ZZ$.
\end{lem}
\begin{proof}
Let $\theta : \Ss \to \Jj_\infty$ be the homeomorphism of
Proposition~\ref{prp:S-J conjugate}. Let $P_1$ be the projection map $P_1 :
\Jj_\infty \to \Jj$ given by $P_1([x_1], [x_2], [x_3], \dots) = [x_1]$. Then
$P_1$ is continuous by definition of the projective-limit topology, and
surjective because $\Jsig$ is surjective. By definition of the topology on
$\Jj_\infty$, the sets $Z(W, n) := \{([x_1], [x_2], [x_3], \dots) \mid [x_n]
\in W\}$ indexed by pairs $(W, n)$ consisting of an open $W \subseteq J$ and
an element $n  \in \NN$ constitute a basis for the topology on $\Jj_\infty$.
Since $\Jsig$ is surjective, each $P_1(Z(W, n)) = \Jsig^n(W)$, which is open.
So $P_1$ is an open map. Hence $\tilde\pi := P_1 \circ \theta$ is a
continuous open surjection from $\Ss$ to $\Jj$. It satisfies $\tilde\pi([x])
= [\dots x_{-3} x_{-2} x_{-1}]$ by definition.
\end{proof}

If $X$ is a locally compact Hausdorff space, $\Gg$ is an \'etale groupoid and
$\pi : X \to \Gg^{(0)}$ is a continuous open surjection, then we can form the
\emph{amplification} of $\Gg$ by $\pi$ which, as a topological space, is
\[
\Gg^\pi := \{(x, \gamma, y) \in X \times \Gg \times X \mid r(\gamma) = \pi(x)\text{ and }s(\gamma) = \pi(y)\}
\]
under the topology inherited from the product topology. Its unit space is
$(\Gg^\pi)^{(0)} = \{(x, \pi(x), x) \mid x \in X\}$ which we identify with
$X$. The range and source maps are given by $r(x, \gamma, y) = x$ and $s(x,
\gamma, y) = y$. The multiplication and inversion are given by $(x, \gamma,
y)(y, \eta, z) = (x, \gamma\eta, z)$ and $(x, \gamma, y)^{-1} = (y,
\gamma^{-1}, x)$. By, for example, \cite[(4)$\implies$(1) of
Proposition~3.10]{FKPS}, the groupoids $\Gg$ and $\Gg^\pi$ are equivalent
groupoids, and hence $C^*(\Gg)$ and $C^*(\Gg^\pi)$ are Morita equivalent by
\cite[Theorem~2.8]{Muhly-Renault-Williams:Equivalence}.

Recall that the stable equivalence relation associated to a Smale space
$(\Ss, d, \tau, \varepsilon_{S}, \lambda)$ is the equivalence relation
\begin{align*}
G^s := \{(\xi,\eta) \in \Ss \times \Ss \mid {}\text{ }\lim_{m\to\infty}d(\tau^m(\xi), \tau^m(\eta)) = 0\}.
\end{align*}
By \cite[pg. 179]{Putnam:Algebras}, there exists $\varepsilon'_{\mathcal{S}}
\leq\varepsilon_{\mathcal{S}}$ such that for any $\delta \leq
\varepsilon'_{\mathcal{S}}$,
\begin{equation}\label{eq:eps'_S}
\begin{split}
G^s := \{(\xi,\eta) \in \Ss \times \Ss \mid {}&\text{there exists } M\in\mathbb{N}\text{ such that }\\
        &d(\tau^m(\xi), \tau^m(\eta)) < \delta \text{ for all } m \ge M\}
\end{split}
\end{equation}
For $M \ge 0$ we define
\[
G^{s}_{\varepsilon, M} = \{(\xi,\eta) \mid  d(\tau^m(\xi), \tau^m(\eta)) < \varepsilon\text{ for all }m \ge M\},
\]
endowed with the subspace topology inherited from $\Ss \times \Ss$. We endow
$G^{s}$ with the inductive-limit topology obtained from the inductive limit
decomposition $G^s = \bigcup_M(G^{s}_{\varepsilon, M})$. It is
straightforward to check this agrees with the topology on $G^{s}$ described
on \cite[pg.~282]{Putnam-Spielberg:Structure}.

We now give a description of the stable equivalence relation and its topology
for the Smale space $(\Ss, d_{\mathcal{S}}, \Stau, \varepsilon_{\mathcal{S}},
\frac{1}{c})$ that will help us in proving the amplification
$\hG^{\tilde{\pi}}_{G,E}$ and the stable Ruelle groupoid $G^s
\rtimes_{\tilde\tau} \ZZ$ are isomorphic.

\begin{lem}\label{lem:stable_description}
Let $E$ be a finite directed graph with no sinks or sources and $(G,E)$ be a
contracting, regular self-similar groupoid. Let $(\Ss, d_{\mathcal{S}},
\Stau, \varepsilon_{\mathcal{S}}, \frac{1}{c})$ be the Smale space of
Corollary~\ref{cor:Smale solenoid}. Let $\varepsilon$ be as in
Theorem~\ref{thm:Jsig} and let $\varepsilon'_\Ss$ be a constant such
that~\eqref{eq:eps'_S} holds for all $\delta < \varepsilon'_\Ss$. Let $\beta
= \min\{\varepsilon,\varepsilon'_{\Ss}\}$. For each $m \in \NN$, let
\[
    G^{s}_{m} = \{([x],[y])\in \Ss\times \Ss \mid\text{ }[x(-\infty, -m)] = [y(-\infty, -m)]\}.
\]
Then there exists $k \in \mathbb{N}$ such that, for every $m \in \mathbb{N}$,
we have $G^{s}_{\beta, m}\subseteq G^{s}_{m}\subseteq G^{s}_{\beta, m+k}$.
Points $[x], [y] \in \Ss$ are stably equivalent if and only if there exists
$m \in \mathbb{N}$ such that $[x(-\infty, -m)] = [y(-\infty, -m)]$. The
topology on $G^{s}$ is equal to the inductive limit topology for the
decomposition $G^{s} = \bigcup_{m}G^{s}_{m}$.
\end{lem}
\begin{proof}
To see that $G^{s}_{\beta, m}\subseteq G^{s}_{m}$, fix $x,y \in
E^{\mathbb{Z}}$ such that $d_{\Ss}([\tau^{n}(x)], [\tau^{n}(y)]) < \beta$ for
all $n\geq m$. Fix $n \ge m$. Then
\begin{align*}
d([x(-\infty, -n)], [x(-\infty, -n)]) &= d([\tau^{n}(x)(-\infty, 0)], [\tau^{n}(x)(-\infty, 0)]) \\
&\leq d_{\Ss}([\tau^{n}(x)], [\tau^{n}(y)]) < \beta <\varepsilon.
\end{align*}
By definition of $\epsilon$,
\begin{align*}
c\cdot d([x(-\infty, -n)], [x(-\infty, -n)]) &= d(\tilde{\sigma}([x(-\infty, -n)]),\tilde{\sigma}([y(-\infty, -n)])) \\
&= d([x(-\infty, -(n+1))], [x(-\infty, -(n+1))]).
\end{align*}
Hence
\[
d([x(-\infty, -m)], [x(-\infty, -m)]) = c^{m-n}d([x(-\infty, -n)], [x(-\infty, -n)])\leq c^{m-n}.
\]
Since $n \ge m$ was arbitrary, we deduce that $[x(-\infty,-m)] = [y(-\infty,
-m)]$.

Fix $k \in \mathbb{N}$ such that $(\frac{1}{c})^{k} < \beta$. We show that
$G^{s}_{m}\subseteq G^{s}_{\beta, m+k}$. Fix $x,y \in E^{\mathbb{Z}}$ such
that $[x(-\infty,-m)] = [y(-\infty, -m)]$. Then, $[\tau^{n}(x)(-\infty, l)] =
[\tau^{n}(y)(-\infty, l)]$ for all $l-n\leq -m$. Therefore,
\[
d_{\Ss}(\tilde{\tau}^{n}[x], \tilde{\tau}^{n}[y]) =
    \sup_{l > n-m}(c^{-l}d([\tau^{n}(x)(-\infty, l)],[\tau^{n}(y)(-\infty, l)])\leq c^{-(n-m)}.
\]
So, whenever $n\geq m +k$, we have $d_{\Ss}(\tilde{\tau}^{n}[x],
\tilde{\tau}^{n}[y]) < c^{-k} <\beta$. Therefore, $([x],[y]) \in
G^{s}_{\beta, m+k}$.
\end{proof}

Recall that the stable Ruelle groupoid is the skew groupoid for the action of
$\ZZ$ on the unit space of $G^s$. That is,
\[
G^s \rtimes_{\tilde{\tau}} \ZZ
    = \{(\xi, n, \eta) \in \Ss \times \ZZ \times \Ss \mid (\tilde{\tau}^{n}(\xi), \eta) \in G^s\}.
\]

\begin{thm}\label{thm:stable gpd iso}
Let $E$ be a finite directed graph with no sinks or sources. Let $(G, E)$ be
a contracting, regular self-similar groupoid action. Let $\tilde\pi : \Ss \to
\Jj$ be the continuous open surjection of Lemma~\ref{lem:pi tilde}, and let
$\Stau : \Ss \to \Ss$ be the homeomorphism of Corollary~\ref{cor:Smale
solenoid}. Then there is an isomorphism $\kappa$ of the stable Ruelle
groupoid $G^s \rtimes_{\Stau} \ZZ$ onto the amplification $\hG^{\tilde\pi}_{G,
E}$ of the dual groupoid of $(G, E)$ by $\tilde\pi$ satisfying $\kappa([x],
n, [y]) = ([x], (\tilde\pi([x]), n, \tilde\pi([y])), [y])$ for all $([x], n,
[y])  \in G^s \rtimes_{\Stau} \ZZ$.
\end{thm}
\begin{proof}
For $z \in E^{\mathbb{Z}}$ and $n \in \mathbb{N}$, we have
$\tilde{\pi}(\tilde{\tau}^{n}([z])) = [z(-\infty, -n)]$. Hence
Lemma~\ref{lem:stable_description} gives $([x], n, [y]) \in G^s$ if and only
if $\tilde\pi(\Stau^{m + n}([x])) = \tilde\pi(\Stau^m([y]))$ for some $m \in
\mathbb{N}$ such that $m+n\geq 0$. Since $\tilde{\pi}\circ\tilde{\tau}^{m} =
\tilde{\sigma}^{m} \circ \tilde{\pi}$ for any $m \in \mathbb{N}$,
\begin{equation}\label{eq:stable equivalence}
([x], n, [y]) \in G^s \quad\Longleftrightarrow\quad
    ([x], (\tilde\pi([x]), n, \tilde\pi([y])), [y]) \in \hG^{\tilde\pi}_{G, E}.
\end{equation}
Hence there is a bijection $\kappa : G^s \rtimes_{\Stau} \ZZ \to
\hG^{\tilde\pi}_{G, E}$ satisfying the desired formula.
\par
We show $\kappa$ is continuous. Extend of $\kappa$ to a map
$\tilde{\kappa}:\Ss\times\mathbb{Z}\times\Ss\to
\Ss\times\Jj\times\mathbb{Z}\times\Jj\times\Ss$ by $\tilde{\kappa}((\xi,
n,\eta)) = (\xi, (\tilde{\pi}(\xi), n,\tilde{\pi}(\eta)), \eta)$. Then
$\tilde\kappa$ is continuous. For $M \in \mathbb{N}$ and $n \in \mathbb{Z}$
such that $M+n\geq 0$ the restriction of $\kappa$ to $G^s_M \rtimes_{\Stau}
\{n\} = \{([x],n,[y])\mid (\tilde{\tau}^{n}([x]),[y])\in G^{s}_{M}\}$ has
co-domain contained in
\[
\Ss\ast\hG^{(M+n, M)}_{G,E}\ast\Ss := \{([x],(\tilde{\pi}([x]), n, \tilde{\pi}([y])), [y])\mid\text{ }\tilde{\sigma}^{M+n}(\tilde{\pi}([x])) = \tilde{\sigma}^{M}(\tilde{\pi}([y]))\}.
\]
The subspace topology of $ G^s_M \rtimes_{\Stau} \{n\}$ relative to $ G^s
\rtimes_{\Stau} \ZZ$ is equal to the subspace topology relative to
$\Ss\times\mathbb{Z}\times\Ss$, and the subspace topology of
$\Ss\ast\hG^{(M+n, M)}_{G,E}\ast\Ss$ relative to $\hG_{G,E}$ is equal to the
subspace topology relative to $\Ss\times
\Jj\times\mathbb{Z}\times\Jj\times\Ss$. So, continuity of $\tilde\kappa$
implies that $\kappa: G^s_M \rtimes_{\Stau} \{n\}\to \Ss\ast\hG^{(M+n,
M)}_{G,E}\ast\Ss$ is continuous. Fix $n \in \mathbb{Z}$. The universal
property of the inductive limit topology on $ G^s\rtimes_{\Stau} \{n\} =
\bigcup_{M+n\geq 0} G^s_M \rtimes_{\Stau} \{n\}$ implies that $\kappa$ is
continuous on the clopen subspace $ G^s\rtimes_{\Stau} \{n\} \subseteq
G^s\rtimes_{\Stau}\mathbb{Z}$, for each $n$ in $\mathbb{Z}$. Hence, $\kappa$
is continuous.
\par
Since $\Ss\ast\hG^{(M+n, M)}_{G,E}\ast\Ss$ and $ G^s_M \rtimes_{\Stau} \{n\}$
are compact, and since $\kappa^{-1}(\Ss\ast\hG^{(M+n,
M)}_{G,E}\ast\Ss)\subseteq G^s_M \rtimes_{\Stau} \{n\}$, and
$\hG_{G,E}^{\tilde{\pi}} = \bigcup_{M+n\geq 0}\Ss\ast\hG^{(M+n,
M)}_{G,E}\ast\Ss$, the map $\kappa$ is proper. Since proper continuous maps
between locally compact Hausdorff spaces are closed, $\kappa$ is a
homeomorphism.
\end{proof}

\begin{cor}\label{cor:dual alg Me}
Let $E$ be a finite directed graph with no sinks or sources. Let $(G, E)$ be a contracting, regular self-similar groupoid action. Then the dual $C^*$-algebra $\hO(G, E)$ of
Definition~\ref{dfn:dual alg} is Morita equivalent to the stable Ruelle algebra $C^*( G^s \rtimes
\ZZ)$ of the Smale space $(\Ss, \Stau)$ of Corollary~\ref{cor:Smale solenoid}.
\end{cor}
\begin{proof}
Since $\tilde\pi : \Ss \to \Jj$ is an open map by Lemma~\ref{lem:pi tilde},
\cite[Proposition~3.10]{FKPS} shows that $\hG^{\tilde\pi}_{G, E}$ is groupoid equivalent to
$\hG_{G, E}$. Therefore \cite[Theorem~2.8]{Muhly-Renault-Williams:Equivalence} shows that $C^*(\hG_{G,
E})$ is Morita equivalent to $C^*(\hG^{\tilde\pi}_{G, E})$. Theorem~\ref{thm:stable gpd iso} shows
that $C^*(\hG^{\tilde\pi}_{G, E})$ is isomorphic to $C^*( G^s \rtimes_{\Stau} \ZZ)$, which is
precisely the stable Ruelle algebra of $(\Ss, \Stau)$.
\end{proof}

\subsection{The unstable algebra}\label{subsec:unstable ME}

We now need to show that the unstable Ruelle algebra of the Smale space $(\Ss, \Stau)$ is Morita
equivalent to the $C^*$-algebra $\Oo(G, E)$. Our approach again is via groupoid equivalence. We
use Putnam and Spielberg's construction of an \'etale groupoid $ G^{u}_{[x]} \rtimes_{\Stau} \ZZ$
corresponding to a choice of orbit in $\Ss$. We will show that this groupoid is isomorphic to a
suitable amplification of the groupoid $\Gg_{G, E}$ of Lemma~\ref{lem:groupoid def}. To do this, we shall need an alternative description of the unstable equivalence relation and its topology, which we provide in the
next lemma.

For $M \in \NN$ and $\varepsilon > 0$, let $G^{u}_{\varepsilon,M} =
\{(\eta,\xi)\in\mathcal{S}\times\mathcal{S} \mid
d_{\Ss}(\tilde{\tau}^{-m}(\eta),\tilde{\tau}^{-m}(\xi)) <\varepsilon \text{
for all } m\geq M\}$. By \cite[pg. 179]{Putnam:Algebras}, there exists
$\varepsilon_{\Ss}'\leq\varepsilon_{\Ss}$ such that for every $\varepsilon <
\varepsilon'_\Ss$, we have
\[
G^{u} = \bigcup_{M} G^{u}_{\varepsilon,M}
\]
in the inductive-limit topology. This agrees with the topology on $G^{u}$ on
\cite[pg. 282]{Putnam-Spielberg:Structure}.

\begin{lem}\label{lem:unstable_description}
Let $E$ be a finite directed graph with no sinks or sources. Let $(G,E)$ be a
contracting, regular self-similar groupoid action. Let $\varepsilon'_\Ss$ be
as above, and let $\varepsilon$ be as in Theorem~\ref{thm:Jsig}. Let $\beta =
\min\{\varepsilon,\varepsilon_{\Ss}'\}$. Let $F$ be the smallest finite set
containing $\Nn\cup\Nn^{2}$ that is closed under restriction. Then, there is
an $l \in \mathbb{N}$ such that for every $M \in \mathbb{N}$, we have
\\$ G^{u}_{\beta, M}\subseteq  G^{u}_{M}\subseteq G^{u}_{\beta, M+l}$, where
$$ G^{u}_{M} = \{([y],[x])\in\Ss\times\Ss\mid\text{ }\exists\text{} g\in F:g\cdot x(M+1,\infty) = y(M+1,\infty)\}.$$ More precisely, if $(\eta,\xi) \in  G^{u}_{\beta, M}$, then for any representatives $y,x$
such that $[y] = \eta$, $[x] = \xi$, there is an element $g \in F$ such that
$g\cdot x(M+1,\infty) = y(M+1,\infty)$.\par In particular, $[x]$ and $[y]$ in
$\Ss$ are unstably equivalent if and only if there is an $M \in \mathbb{N}$
and $g \in G^s(x_{M})$ such that $g\cdot x(M+1,\infty) = y(M+1,\infty)$, and
the topology on $ G^{u}$ is equal to the inductive limit topology provided by
the decomposition $ G^{u} = \bigcup_{M} G^{u}_{M}$.
\end{lem}
\begin{proof}
Fix $M \in \mathbb{N}$. We first show $ G^{u}_{\beta, M}\subseteq G^{u}_{M}$.
Suppose $([x],[y]) \in  G^{u}_{\beta, M}$. For $m \in \mathbb{Z}$, let $x^{m}
:= [x(-\infty,m)]$ and $y^{m} := [y(-\infty, m)]$. By definition of the
metric $d_{\Ss}$, for every $m\geq M$,
\[
d(x^{m},y^{m}) = d([\tau^{-m}(x)(-\infty,0)],[\tau^{-m}(y)(-\infty,0)])\leq d_{\Ss}(\tilde{\tau}^{-m}([x]),\tilde{\tau}^{-m}([y]))\ <\beta.
\]
Since $\beta\leq\varepsilon$, Theorem~\ref{thm:Jsig}(1) gives $d(x^{m},y^{m})
= d(\tilde{\sigma}(x^{m+1}), \tilde{\sigma}(y^{m+1})) = c\cdot d(x^{m+1}, y^{m+1})$
for all $m\geq M$. So $d(x^{M+s}, y^{M+s}) = (\frac{1}{c})^{s}d(x^{M},
y^{M})$ for all $s \ge 0$. Hence $\alpha := \frac{d(x^{M}, y^{M})
+\beta}{2}$, satisfies $y^{M+s} \in B(x^{M+s}, (\frac{1}{c})^{s}\alpha)$ for
every $s\ge 0$.

Fix $n \in \mathbb{N}$ as in Lemma~\ref{lem:delta_property} and $k \in
\mathbb{N}$ as in Proposition~\ref{prp:mapping_properties}. Since $\alpha
<\varepsilon$, Lemma~\ref{lem:delta_property}(1) yields $\omega \in E^*$ such
that $|\omega|\geq n-1$ and $B(x^{M},\alpha)\subseteq U_{\omega}$.

Since $\tilde{\sigma}(x^{M+1}) = x^{M}$ and $|\omega|\geq n-1\geq k-1$,
Proposition~\ref{prp:mapping_properties}(3) gives $e \in E^1$ such that
$s(\omega) = r(e)$ and $x^{M+1} \in U_{\omega e}$. Since
$B(x^{M},\alpha)\subseteq U_{\omega}$, Lemma~\ref{lem:delta_property} implies
that $B(x^{M+1},\frac{1}{c}\alpha)\subseteq U_{\omega e}$.

Applying the argument in the above paragraph inductively gives paths
$\{\mu_{s}\}_{s\in\mathbb{N}}$ such that $\mu_{s+1} = \mu_{s}e_{s+1}$, for
some edge $e_{s+1}$, and that $B(x^{M+s},(\frac{1}{c})^{s}\alpha)\subseteq
U_{\omega \mu_{s}}$, for all $s \in \mathbb{N}$. Since $y^{M+s} \in
B(x^{M+s},(\frac{1}{c})^{s}\alpha)$ for each $s \in \mathbb{N}$, we obtain
$x^{M+s}, y^{M+s} \in U_{\omega \mu_{s}}$ for each $s \in \mathbb{N}$.
Therefore, there exist sequences $\{g_{s}\}_{s\in\mathbb{N}}$,
$\{h_{s}\}_{s\in\mathbb{N}}$ in $\Nn$ such that $d(g_{s}) = d(h_{s}) =
r(\mu_{s})$, $g_{s}\cdot \mu_{s} = x(M+1, M+s)$, and $h_{s}\cdot \mu_{s} =
y(M+1, M+s)$ for all $s \in \mathbb{N}$. Let $z \in E^\infty$ be the unique
element of $\bigcap_s \rZ{\mu_s}$. Choose an increasing subsequence
$\{n_{s}\}_{s\in\mathbb{N}}$ such that $(g_{n_{s}})_s$ and $(h_{n_s})_s$ are
constant sequence, with constant values $g$ and $h$, say. Then $g\cdot z =
x(M+1,\infty)$ and $h\cdot z = y(M+1,\infty)$. So $t := hg^{-1} \in
\mathcal{N}^{2}\subseteq F$ satisfies $t\cdot x(M+1,\infty) = y(M+1,\infty)$.
Therefore, $([y], [x]) \in G^{u}_{M}$.

Take $l_{1} \in \mathbb{N}$ such that for all $g \in F$ and $\mu \in d(g)E^*$
such that $|\mu|\geq l_{1}$ we have that $g|_{\mu} \in \mathcal{N}$. Fix
$l_{2} \in \mathbb{N}$ such that $\text{diam}(U_{\nu}) < \frac{1}{c}\beta$
for every path $\nu$ with $|\nu|\geq l_{2}$. Let $l := l_{1}+l_{2}$. We show
that $ G^{u}_{M}\subseteq  G^{u}_{\beta, M+l}$.

Take $[x],[y] \in \Ss$ such that there exists $g \in F$ such that $d(g) =
s(x_{M})$ and $g\cdot x(M+1,\infty) = y(M+1,\infty)$. As above, let $x^{m} :=
[x(-\infty,m)]$ and $y^{m} := [y(-\infty, m)]$ for all $m \in \mathbb{Z}$.
Let $h = g|_{x(M+1,M+l_{1})}$. By the choice of $l_{1}$, we have $h \in
\mathcal{N}$. Since $h\cdot x(M+l_{1},\infty) = y(M+l_{1},\infty)$, it
follows that $x^{M+l_{1}+m}, y^{M+l_{1}+m} \in U_{x(M+l_{1}+1, M+l_{1}+m)}$
for all $m \in \mathbb{N}$. By the choice of $l_{2}$, we have
$d(x^{M+l_{1}+m},y^{M+l_{1}+m}) <\frac{1}{c}\beta$ for all $m\geq l_{2}$. So,
for all $s\geq M+l$ and $k \ge 0$, we have $(\frac{1}{c})^{k}d(x^{k+s},
y^{k+s}) < \frac{1}{c}\beta$. Therefore,
$d_{\Ss}(\tilde{\tau}^{-s}([x]),\tilde{\tau}^{-s}([y])) =
\sup_{k\in\mathbb{N}_{0}}(\frac{1}{c})^{k}d(x^{k+s},
y^{k+s})\leq\frac{1}{c}\beta < \beta$ for all $s\geq M+l$, forcing $([y],
[x]) \in G^{u}_{\beta, M+l}$. Hence $ G^{u}_{M}\subseteq G^{u}_{\beta, M+l}$.

We have established that $[x]$ and $[y]$ are unstably equivalent if and only
if there exist $M \in \mathbb{N}$ and $g \in Fs(x_{M})$ such that $g\cdot
x(M+1,\infty) = y(M+1,\infty)$. If $[x], [y] \in \mathcal{S}$, $N$ in
$\mathbb{N}$ and $g \in Gr(x_{N})$ satisfy $g\cdot x(N+1,\infty) =
y(N+1,\infty)$, then, since $(G, E)$ is contracting, there exists $M\geq N$
such that $h:= g|_{x(N+1, M)} \in \mathcal{N}\subseteq F$ and $h\cdot
x(M+1,\infty) = y(M+1,\infty)$. This proves the penultimate statement of the
lemma.

Since $F$ is closed under restriction, $ G^{u}_{M}\subseteq G^{u}_{M+1}$ for
all $M \in \mathbb{N}$. Hence the inductive limit topology with respect to
the decomposition $ G^{u} = \bigcup_{M} G^{u}_{M}$ is well defined. Since $
G^{u}_{\beta, M}\subseteq  G^{u}_{M}\subseteq G^{u}_{\beta, M+l}$ for all $M
\in \mathbb{N}$, this topology is equal to the one provide by the
decomposition $ G^{u} = \bigcup_{M} G^{u}_{\beta, M}$.
\end{proof}

Let $(X,\tau)$ be an irreducible Smale space, and recall that $
G^{u}\rtimes\ZZ = \{(\eta,l,\xi)\in X\times\ZZ\times
X:(\tau^{-l}(\eta),\xi)\in  G^{u}\}$. In \cite{Putnam-Spielberg:Structure},
Putnam and Spielberg show that, given any point $x \in X$, the groupoid $
G^{u}_{x} \rtimes \ZZ$ defined as
\begin{equation}\label{eq:PutnamSpielberg}
 G^{u}_{x} \rtimes \ZZ := \{(\eta, n, \varepsilon) \in  G^u \rtimes \ZZ \mid (x, \eta), (x, \varepsilon) \in  G^s\},
\end{equation}
endowed with a suitable topology, is an \'etale groupoid that is equivalent to $ G^u\rtimes\ZZ$ when $(X,\tau)$ is mixing, and use
this to study the unstable $C^*$-algebra of a Smale space up to Morita equivalence. We will make
use of the same technique here.

Let $E$ be a finite directed graph with no sinks or sources, and let $(G, E)$
be a contracting, regular self-similar groupoid action. Let $(\Ss, d_{\Ss},
\Stau, \varepsilon_{S}, \frac{1}{c})$ be the Smale space of
Corollary~\ref{cor:Smale solenoid}. Let $\beta >0$ and $k \in \mathbb{N}$ be
as in Lemma~\ref{lem:stable_description}. Consider $x \in E^\ZZ$. In line
with \cite{Putnam-Spielberg:Structure}, the global stable equivalence class
\begin{equation}\label{eq:Ss[x] top}
\Ss^s([x]) := \{\xi \in \Ss : ([x], \xi) \in  G^s\}
\end{equation}
is endowed with the inductive-limit topology coming from the decomposition
$\Ss^s([x]) = \bigcup_{M}\Ss_{\beta, M}^{s}([x])$, where
\[
\Ss_{\beta, M}^{s}([x]) = \{\xi\in\Ss \mid
d_{\Ss}(\tilde{\tau}^{n}([x]),\tilde{\tau}^{n}(\xi)) <\beta,\text{
}\forall\text{ }n\geq M\}
\]
is given the subspace topology relative to $\Ss$. For $M \in \NN$, define
\[
\Ss^{s}_{M}([x]) = \{[y]\in\Ss : [x(-\infty, M)] = [y(-\infty, M)]\}.
\]
Then Lemma~\ref{lem:stable_description} implies that $\Ss_{\beta,
M}^{s}([x])\subseteq\Ss_{M}^{s}([x])\subseteq\Ss_{\beta, M+k}^{s}([x])$.
Hence, the inductive-limit topology on $\Ss^s([x])$ is equivalent to the
inductive-limit topology for the decomposition $\Ss^s([x]) =
\bigcup_{M}\Ss^{s}_{M}([x])$. Note that $\Ss^s([x])$ is not compact in this
topology even though $\Ss$ is compact.

We equip the groupoid $G^{u}_{[x]} \rtimes \ZZ = ( G^u \rtimes \ZZ) \cap
(\Ss^s([x]) \times\ZZ\times \Ss^s([x]))$ with the topology with sub-basis
\[
\big\{U\cap r^{-1}(V)\cap s^{-1}(W) \mid U \subseteq G^{u}\rtimes\mathbb{Z}\text{ and } V,W \subseteq \Ss^{s}([x]) \text{ are open}\big\}.
\]

Fix a periodic orbit $P = \{\tilde\tau^l(p_0) : l \in \ZZ\} =
\{\tilde\tau^l(p_0) : 0 \le l < N\}$ of $\tilde{\tau}$. Then $\Ss^{s}(p)\cap
S^{s}(q) = \emptyset$ for distinct $p,q \in P$. So $\Ss^{s}(P):=
\bigcup_{p\in P}\Ss^{s}(p)$ is the topological disjoint union of the sets
$\Ss^s(p)$. Consider the groupoid $ G^u(P) \rtimes \ZZ := ( G^u \rtimes \ZZ)
\cap (\Ss^s(P) \times\ZZ\times \Ss^s(P))$. As in the preceding paragraph, we
give $G^u(P) \rtimes \ZZ$ the topology with sub-basis
\[
\big\{U\cap r^{-1}(V)\cap s^{-1}(W) \mid U \subseteq G^{u}\rtimes\mathbb{Z}\text{ and } V,W \subseteq \Ss^{s}(P) \text{ are open}\big\}.
\]
Then $G^{u}(P)\rtimes\ZZ$ is a locally compact \'etale Hausdorff groupoid
\cite[Section~3]{Kaminker-Putnam-Whittaker:K-duality}, and
$G^{u}(P)\rtimes\ZZ$ is groupoid equivalent to $G^{u}\rtimes\ZZ$
\cite[Section~2]{Putnam:Functoriality}.

We claim that for any $p \in P$, the subset $G^{u}_{p}\rtimes\ZZ$ is also a
locally compact Hausdorff groupoid that is equivalent to $G^u \rtimes \ZZ$.
Indeed, the open subgroupoid $G^{u}(P)\rtimes\ZZ|_{\Ss^{s}(p)} = \{g\in
G^{u}(P)\rtimes\ZZ:r(g),s(g)\in \Ss^{s}(p)\}$ is equal to $G^{u}_{p}\rtimes
\ZZ$, and the relative topology of $G^{u}_{p}\rtimes\ZZ$ inherited from
$G^{u}(P)\rtimes\ZZ$ is equal to the topology on $G^{u}_{p}\rtimes\ZZ$
described above. Hence $G^{u}_{p}\rtimes\ZZ$ is a  locally compact \'etale
Hausdorff groupoid.

We show $G^{u}_{p}\rtimes\ZZ$ is equivalent to $G^{u}(P)\rtimes\ZZ$, and
hence to $G^{u}\rtimes\ZZ$. Since $H = \{g\in G^{u}(P)\rtimes\ZZ:r(g)\in
\Ss^{s}(p)\}$ is a clopen subset, it follows from
\cite[Example~2.7]{Muhly-Renault-Williams:Equivalence} that $H$ implements a
groupoid equivalence between $G^{u}_{p}\rtimes\ZZ$ and $G^{u}(P)\rtimes\ZZ$
if and only if the source map restricted to $H$ surjects onto $\Ss^{s}(P)$.
Fox $q \in P$ and $z \in \Ss^{s}(q)$. There exists $n \in \mathbb{N}$ such
that $\tilde{\tau}^{n}(q) = p$. So $(\tilde{\tau}^{n}(z), n, z) \in H$,
proving surjectivity.

We describe an amplification of $\mathcal{G}_{G,E}$ which we will prove in
Theorem~\ref{thm:unstable gpd iso} is Morita equivalent to $G^{u}\rtimes\ZZ$.

For $x \in E^{\mathbb{Z}}$, we write
\begin{equation}\label{eq:EZx}
    E^\ZZ_x := \{y \in E^\ZZ \mid y(-\infty, -M) = x(-\infty, -M)\text{ for some }M\in\mathbb{N}\}.
\end{equation}
For $M  \in \NN$, let
\[
    E^\ZZ_x(M) := \{y \in E^\ZZ : y(-\infty, -M) = x(-\infty, -M)\},
\]
endowed with the relative topology inherited from $E^\ZZ$. We endow $E^\ZZ_x$ with the
inductive-limit topology determined by this decomposition.

The map $\pi_{x}:E_{x}^{\ZZ}\mapsto E^{\infty}$ sending $z \in E_{x}^{\ZZ}$
to $\pi_{x}(z) = z(1,\infty)$ is an open continuous map (in fact, it is a
local homeomorphism onto an open set in $E^{\infty}$). It is a surjection
whenever $E$ is strongly connected. We show that the amplification
$\mathcal{G}_{G,E}^{\pi_{x}}$ is isomorphic to $G_{[x]}^{u}\rtimes\ZZ$. We
start by analysing the space $E_{x}^{\mathbb{Z}}$.

We first prove that an element $x  \in E^\ZZ$ is completely determined by its
class $[x]$ in the limit solenoid $\Ss$ of Definition~\ref{def:S space}
together with the tail $\dots x_{n-3} x_{n-2} x_{n-1}$ for any $n  \in \ZZ$.

\begin{lem}\label{lem:section}
Let $E$ be a finite directed graph with no sinks or sources. Let $(G, E)$ be
a contracting, regular self-similar groupoid action. Let $\Ss$ be the limit
solenoid of Definition~\ref{def:S space}. Suppose that $x, y  \in E^\ZZ$
satisfy $[x] = [y]  \in \Ss$, and suppose that there exists $n $ in $\ZZ$
such that $x_m = y_m$ for all $m \le n$. Then $x = y$.
\end{lem}
\begin{proof}
Fix $n$ satisfying $x_m = y_m$ for all $m \le n$. Let $k$ be as in
Lemma~\ref{lem:magic k}, with respect to the finite set $F' = \Nn\cup G^{0}$.
Since $x \sim_\aeq y$, Lemma~\ref{lem:ae by nucleus} shows that there is a
sequence $(g_m)_{m \in \ZZ} $ in $\Nn$ such that $g_m \cdot x_m x_{m+1} \dots
= y_m y_{m+1} \dots$ for all $m \in \mathbb{Z}$ and $g_m|_{x_m \dots x_{l-1}}
= g_l$ for all $m \le l$. In particular,
\[
g_m \cdot x_{-n-k} \dots x_{-n-1} = y_{-n-k} \dots y_{-n-1} = x_{-n-k} \dots x_{-n-1}.
\]
Since $v := r(x_{-n-k}) \in G^{(0)}\subseteq F'$ and also satisfies $v \cdot
x_{-n-k} \dots x_{-n-1} = x_{-n-k} \dots x_{-n-1}$, the choice of $k$
guarantees that $g_m|_{x_{-n-k} \dots x_{-n-1}} = v|_{x_{-n-k} \dots
x_{-n-1}} = r(x_n)$. We then have
\begin{align*}
y_{n+1} y_{n+2} \dots
    &= g_{n-1} \cdot x_{n+1} x_{n+2} \dots\\
    &= g_m|_{x_{-n-k} \dots x_{-n-1}} \cdot x_{n+1} x_{n+2} \dots
    = v \cdot x_{n+1} x_{n+2} \dots
    = x_{n+1} x_{n+2} \dots
\end{align*}
Since we already have $x_m = y_m$ for $m \le n$, we conclude that $x = y$.
\end{proof}

Lemma~\ref{lem:section} shows that for $x \in E^\ZZ$ the quotient map $y
\mapsto [y]$ from $E^\ZZ$ to $\Ss$ restricts to an injection $E^\ZZ_x \to
\Ss^s([x])$. We show that this is a homeomorphism with respect to the
inductive-limit topologies.

\begin{lem}\label{lem:x gives section}
Let $E$ be a finite directed graph with no sinks or sources, and let $(G, E)$
be a contracting, regular self-similar groupoid action. Fix $x  \in E^{\ZZ}$.
The map $y \mapsto [y]$ is a homeomorphism from $E^\ZZ_x$ onto $\Ss^s([x])$
with respect to the inductive-limit topologies described at \eqref{eq:Ss[x]
top}~and~\eqref{eq:EZx}.
\end{lem}
\begin{proof}
Fix $n  \in \NN$. Then $y \mapsto [y]$ restricts to a bijection of
$E^\ZZ_x(n)$ onto $\Ss^s_n([x])$. By definition of the inductive-limit
topologies it suffices to show that this map is continuous and open. For
fixed $n$ the set $E_{x}^{\mathbb{Z}}(n)$ is compact because it is closed in
$E^\ZZ$, and the set $\Ss^s_n([x])$ is Hausdorff because $\Ss$ is. Since $y
\mapsto [y]$ is the quotient map, it is continuous on $E_{x}^{\ZZ}(n)$, and
we deduce that it is a homeomorphism.
\end{proof}

We now analyse the topology of $ G^{u}_{[x]}\rtimes\ZZ$, for any $x \in
E^\mathbb{Z}$.

\begin{rmk}\label{rmk:G-sets open}
For any $m,M \in \mathbb{N}$ such that $m\leq M$, we have
\[
    G^{u}_{\beta, m} = G^{u}_{\beta, M}\bigcap_{k: m\leq k\leq M} \{(\eta,\xi)\in\Ss\times\Ss: d_{\Ss}(\tau^{-k}(\eta),\tau^{-k}(\xi))<\beta\}.
\]
Therefore, $G^{u}_{\beta,m}$ is open in $ G^{u}$, and consequently $
G^{u}_{\beta,m}\rtimes\{l\}:=\{(\eta,l,\xi):(\tilde{\tau}^{-l}(\eta), \xi)\in
G^{u}_{\beta, m}\}$ is open in $ G^{u}\rtimes\ZZ$.
\end{rmk}

\begin{lem}\label{lem:unstable_top}
Let $E$ be a finite directed graph with no sources, and let $(G,E)$ be a
contracting, regular self-similar action. Fix $x \in E^{\ZZ}$. For $k,m \in
\mathbb{N}$ and $l \in \mathbb{Z}$, let
\[
X_{k,l,m} := \{([z],l,[y])\in G^{u}_{\beta, m}\rtimes\{l\}: x(-\infty,-m) =
y(-\infty,-m) = z(-\infty, -m) \}.
\]
Then, $X_{k,l,m}$ is an open subset of $G^{u}_{[x]}\rtimes\ZZ$, and the
relative topology it inherits from $G^{u}_{[x]}\rtimes\ZZ$ coincides with the
relative topology inherited from $\Ss\times\ZZ\times\Ss$.
\end{lem}
\begin{proof}
We have $X_{k,l,m} =  G^{u}_{\beta,k}\rtimes\{l\}\bigcap
r^{-1}(q(E_{x}^{\ZZ}(m)))\cap s^{-1}(q(E_{x}^{\ZZ}(m)))$. The first factor is
open in $ G^{u}\rtimes\ZZ$ by Remark~\ref{rmk:G-sets open}. The image
$q(E_{x}^{\ZZ}(m))$ is open in $\Ss^{s}([x])$ by Lemma~\ref{lem:x gives
section}. Therefore, $X_{k,l,m}$ is open in $  G^{u}_{[x]}\rtimes\ZZ$. The
topologies on $G^{u}_{\beta,k}\rtimes\{l\}$ and $q(E_{x}^{\ZZ}(m))$ are the
subspace topologies relative to $\Ss\times\ZZ\times\Ss$ and to $\Ss$
respectively. Hence, for all triples of open sets $U \subseteq
G^{u}\rtimes\ZZ$ and $V,W \subseteq \Ss^{s}([x])$, there exist open sets $U'
\subseteq \Ss\times\ZZ\times\Ss$ and $V',W' \subseteq \Ss$ such that
\begin{align*}
X_{k,l,m}&\cap U\cap r^{-1}(V)\cap s^{-1}(W)\\
    &= \big(U\cap G^{u}_{\beta,k}\rtimes\{l\}\big) \cap \big(r^{-1}(V)\cap r^{-1}(q(E_{x}^{\ZZ}(m)))\big) \cap \big(s^{-1}(W)\cap s^{-1}(q(E_{x}^{\ZZ}(m)))\big)\\
    &= X_{k,l,m}\cap U'\cap r^{-1}(V')\cap s^{-1}(W').
\end{align*}
Since $r,s$ are continuous with respect to the subspace topologies, $U'\cap
r^{-1}(V')\cap s^{-1}(W')$ is open in $\Ss\times\ZZ\times\Ss$. Hence
$X_{k,l,m}$ has the subspace topology relative to $\Ss\times\ZZ\times\Ss$.
\end{proof}

To lighten notation, for all $z,y \in E^\ZZ_x$, all $m,n \in \NN$ and all $g
\in G$ such that $g\cdot \sigma^n(\pi_x(y)) = \sigma^m(\pi_x(z))$, we define
\begin{equation}\label{eq:notational shortcut}
    [z, m, g, n, y] := (z,[\pi_{x}(z),m,g,n,\pi_{x}(y)],y) \in \mathcal{G}_{G,E}^{\pi_{x}}.
\end{equation}

\begin{ntn}\label{ntn:calZs}
Give finite paths $\mu_{-},\mu_{+},\nu_{-},\nu_{+} \in E^{*}$ such that
$s(\mu_{-}) = r(\mu_{+})$, $s(\nu_{-}) = r(\nu_{+})$, $|\mu_{-}| = |\nu_{-}|$
and an element $g \in G$ with $s(\nu_{+}) = d(g)$, $s(\mu_{+}) = c(g)$, we
define
\begin{align*}
\mathcal{Z}(\mu_{-}\mu_{+}) = \{z\in E_{x}^{\ZZ}: {}&z(-|\mu_{-}|,|\mu_{+}|) = \mu_{-}\mu_{+}\\
    &\text{ and } z(-\infty, -|\mu_{-}| - 1) = x(-\infty, -|\mu_{-}| - 1)\},
\end{align*}
and
\[
\mathcal{Z}(\mu_{-}\mu_{+}, g, \nu_{-}\nu_{+} ) = \{[z,|\mu_{+}|, g, |\nu_{-}|, y]\in\mathcal{G}_{G,E}^{\pi_{x}}:z\in \mathcal{Z}(\mu_{-}\mu_{+}), y\in  \mathcal{Z}(\nu_{-}\nu_{+})\}.
\]
These two collections form a basis for the topologies on $E_{x}^{\ZZ}$, $\mathcal{G}_{G,E}^{\pi_{x}}$, respectively.
\end{ntn}

We show that $\mathcal{G}_{G,E}^{\pi_{x}}$ and $  G^{u}_{[x]}\rtimes\ZZ$ are
isomorphic.

\begin{thm}\label{thm:unstable gpd iso}
Let $E$ be a finite directed graph with no sinks or sources, and let $(G,E)$
be a contracting, regular self-similar groupoid action. Fix $x \in E^{\ZZ}$.
The map $\theta : \mathcal{G}_{G,E}^{\pi_{x}} \to G^{u}_{[x]}\rtimes\ZZ$ such
that $\theta([z,m,g,n,y]) = ([z],m-n,[y])$ is an isomorphism of topological
groupoids.
\end{thm}
\begin{proof}
If $[z, m, g, n, y] = [z', m', g', n' y']$, then $z = z'$ and $y = y'$
by~\eqref{eq:notational shortcut}, and $m - n = m' - n'$ by definition of the
equivalence relation defining $\mathcal{G}_{G, E}$. So there is a
well-defined map $\theta$ satisfying $\theta([z,m,g,n,y]) = ([z],m-n,[y])$.

If $[z,m,h,n,y] \in \mathcal{Z}(\mu_{-}\mu_{+}, g, \nu_{-}\nu_{+} )$ for $g$
in the nucleus $\mathcal{N}$, then $g\cdot y(|\nu_{+}|+1,\infty) =
z(|\nu_{+}|+1 +(m-n),\infty)$. Lemma~\ref{lem:unstable_description} yields $l
\in \NN$ such that $([z], m-n, [y]) \in  G^{u}_{\beta,
|\nu_{+}|+l}\rtimes\{n-m\}$. We have $([z],m-n,[y]) \in
r^{-1}(q(\mathcal{Z}(\mu_{-}\mu_{+})))\cap
s^{-1}(q(\mathcal{Z}(\nu_{-}\nu_{+})))$, so $([z],m-n,[y]) \in X_{|\nu_{+}| +
l, |\mu_{+}|-|\nu_{+}|, |\nu_{-}|}$. Hence,
\begin{equation}\label{eq:theta mapping property}
\theta(\mathcal{Z}(\mu_{-}\mu_{+}, g, \nu_{-}\nu_{+} ))\subseteq X_{|\nu_{+}| + l, |\mu_{+}|-|\nu_{+}|, |\nu_{-}|}, \text{ }g\in\mathcal{N}.
\end{equation}
Hence $\theta(\Gg^{\pi_x}_{G, E}) \subseteq G^{u}_{[x]}\rtimes\ZZ$.

It is straightforward that $\theta$ is a groupoid homomorphism.

We show that $\theta$ is continuous. It is enough to show that $\theta$
restricts to a continuous map from $\mathcal{Z}(\mu_{-}\mu_{+}, g,
\nu_{-}\nu_{+} )$ to $X_{|\nu_{+}| + l, |\mu_{+}|-|\nu_{+}|, |\nu_{-}|}$  for
all $\mu_+, \mu_-, \nu_+,\nu_-$ as in Notation~\ref{ntn:calZs}. By
Lemma~\ref{lem:unstable_top}, $X_{|\nu_{+}| + l, |\mu_{+}|-|\nu_{+}|,
|\nu_{-}|}$ has the subspace topology inherited from $\Ss\times\ZZ\times\Ss$,
so it suffices to show that $\theta: \mathcal{Z}(\mu_{-}\mu_{+}, g,
\nu_{-}\nu_{+} )\mapsto\Ss\times\ZZ\times\Ss$ is continuous.

Let $([z_{\lambda}, m_{\lambda}, h_{\lambda}, n_{\lambda},
y_{\lambda}])_{\lambda\in\Lambda}$ be a net in $\mathcal{Z}(\mu_{-}\mu_{+},
g, \nu_{-}\nu_{+} )$ converging to $[z,m,h,n,y]$. Since $r,s$ are continuous,
it follows that $z_\lambda \to z$ in $\mathcal{Z}(\mu_{-}\mu_{+})$ and
$y_{\lambda} \to y$ in $\mathcal{Z}(\nu_{-}\nu_{+})$. Since
$\mathcal{Z}(\mu_{-}\mu_{+})$ and $\mathcal{Z}(\nu_{-}\nu_{+})$ carry the
subspace topologies relative to $E^{\ZZ}$, we have $y_n \to y$ and $z_n \to
z$ in $E^{\ZZ}$. Since the quotient map $q:E^{\ZZ}\mapsto \Ss$ is continuous,
\[
([z_{\lambda}], m_{\lambda} - n_{\lambda}, [y_{\lambda}]) \to ([z],
m_{\lambda}-n_{\lambda}, [y]) = ([z], m-n, [y])
\]
in $\Ss\times\ZZ\times\Ss$. Hence, $\theta$ is continuous.

Now, we show that $\theta$ is an open map. It suffices to show that each
$\theta(\mathcal{Z}(\mu_{-}\mu_{+}, g, \nu_{-}\nu_{+} ))$ is open.

By \eqref{eq:theta mapping property} and Lemma~\ref{lem:unstable_top}, it is
enough to show that $\theta(\mathcal{Z}(\mu_{-}\mu_{+}, g, \nu_{-}\nu_{+}))$
is an open subset of $X_{|\nu_{+}| + l, |\mu_{+}|-|\nu_{+}|, |\nu_{-}|}$.
Write $n = |\nu_{+}|$, $m = |\mu_{+}|$ and $k = |\nu_{-}|$. Let
\begin{align*}
\tilde{X} = \{([z],m-n,[y])\mid \exists h\in F: h\cdot z( m +l+1,\infty) = y(n+l+1,\infty)\\
&\hskip-27em\text{ and } x(-\infty, -k) = z(-\infty, -k) = y(-\infty, -k)\}.
\end{align*}
Then Lemma~\ref{lem:unstable_description} gives $X_{n+l, m-n, k}\subseteq
\tilde{X}$. So it suffices to show that $\theta(\mathcal{Z}(\mu_{-}\mu_{+},
g, \nu_{-}\nu_{+} ))$ is open in the relative topology
inherited from $\Ss^{s}_{k}([x])\times\{m-n\}\times\Ss^{s}_{k}([x])$.

Let $j$ be the number from Lemma~\ref{lem:magic k} applied to the set $F =
E^{0}\cup\mathcal{N}\cup\mathcal{N}^{2}$. Let $L = \max\{(m+l+1 +j), k,
(n+l+1+j)\}$. Fix $(u, m, g,n, v) \in \mathcal{Z}(\mu_{-}\mu_{+}, g,
\nu_{-}\nu_{+} )$. Let $W = q(\mathcal{Z}(u(-L,-1)u(1,L)))\times\{m-n\}\times
q(\mathcal{Z}(v(-L,-1)v(l,L)))$. Then $W$ is open in
$\Ss^{s}_{k}([x])\times\{m-n\}\times\Ss^{s}_{k}([x])$. We show that
$W\cap\tilde{X}\subseteq \theta(\mathcal{Z}(\mu_{-}\mu_{+}, g,
\nu_{-}\nu_{+}))$.

Fix $a,b \in E_{x}^{\mathbb{Z}}$, $m,n \in \NN$ and $h \in F$ such that
$h\cdot a(m +l+1,\infty) = b(n+l+1,\infty)$ and $x(-\infty, -k) = a(-\infty,
-k) = b(-\infty, -k)$. Suppose that $([a], m-n, [b]) \in W\cap\tilde{X}$. For
$L\geq m,n,k$, the paths $\mu_{-}\mu_{+}, \nu_{-}\nu_{+}$ are subwords of
$u(-L,-1)u(1,L)$, $v(-L,-1)v(1,L)$, respectively. Therefore, since
$q:E^{\mathbb{Z}}_{x}\mapsto\Ss^{s}([x])$ is bijective (Lemma \ref{lem:x
gives section}), that $([a], m-n, [b]) \in W$ implies that $a \in
\mathcal{Z}(\mu_{-}\mu_{+})$ and $b \in \mathcal{Z}(\nu_{-}\nu_{+})$. It only
remains to show that $g\cdot a(m+1,\infty) = b(n+1,\infty)$.

By the choice of $L$, we have $a(m+l+1, m+l+j) = u(m+l+1, m+l+j) =: p$ and
$g|_{u(m+1,m+l+1)}\cdot p = h\cdot p$. By the choice of $j$, we have
$g|_{u(m+1,m+l+j)} = h|_{a(m+l+1, m+l+j)}$. Therefore,
\begin{align*}
g\cdot a(m+1,\infty)
    &= v(n+1,n+l+j)(g|_{u(m+1,m+l+j)})\cdot a(m+l+1+j,\infty)\\
    &= b(n+1,n+l+j)(h|_{a(m+l+1, m+l+j)})\cdot a(m+l+1+j,\infty)\\
    &= b(n+1,\infty).
\end{align*}
We have shown that the open neighbourhood $W\cap\tilde{X}$ of
$\theta([u,m,g,n,v])$ is contained in $\theta(\mathcal{Z}(\mu_{-}\mu_{+}, g,
\nu_{-}\nu_{+} ))$. Therefore, $\theta(\mathcal{Z}(\mu_{-}\mu_{+}, g,
\nu_{-}\nu_{+} ))$ is open, and hence $\theta$ is an open map.

Now, we show that $\theta$ is a bijection. We first show injectivity. Since
$q:E_{x}^{\ZZ}\mapsto \Ss^{s}([x])$ is a bijection, $\theta([z,m,g,n,y]) =
\theta([z',m',g',n',y'])$ implies $z = z'$ and $y = y'$. Since $m - n = m'-n'
= l$, for $k$ large enough, we have $g|_{y(n,k)}\cdot y(k+1,\infty)=
g'|_{y(n',k)}\cdot y(k+1,\infty)$. By regularity, $g|_{y(n,K)} =
g'|_{y(n,K)}$ for some $K\geq k$. Therefore, $[z,m,g,n,y] = [z',
m',g',n',y'].$
\par
Finally, we show surjectivity. Fix $([z],l,[y]) \in G^{u}_{[x]}\rtimes\ZZ$.
By Lemma~\ref{lem:unstable_description}, there exist $M \in \mathbb{N}$ and
$g \in F$ such that $M+l\geq 0$ and $g\cdot y(M+1,\infty) = z(M+l+1,\infty)$.
Hence $[z,M+l,g,M,y] \in \mathcal{G}^{\pi_{x}}_{G,E}$ satisfies
$\theta([z,M+l,g,M,y]) = ([z],l,[y])$.
\end{proof}

\begin{cor}\label{cor:base alg Me}
Let $E$ be a finite strongly connected directed graph, and let $(G, E)$ be a
contracting, regular self-similar groupoid action. Then the $C^*$-algebra
$\Oo(G, E)$ of Section~\ref{sec:O(G,E)} is Morita equivalent to the unstable
Ruelle algebra $C^*( G^u \rtimes \ZZ)$ of the Smale space $(\Ss, \Stau)$ of
Corollary~\ref{cor:Smale solenoid}.
\end{cor}
\begin{proof}
Fix $x \in E^\ZZ$ such that $[x]$ is periodic under the action of
$\tilde{\tau}$. Since $E$ is assumed strongly connected,
Corollary~\ref{cor:Smale solenoid} implies $(\Ss,\tilde{\tau})$ is
irreducible. Thus, as shown above Lemma~\ref{lem:unstable_top}, $G^u \rtimes
\ZZ$ is groupoid equivalent to $G^{u}_{[x]} \rtimes \ZZ$. Hence $C^*( G^u
\rtimes \ZZ)$ is Morita equivalent to $C^*( G^{u}_{[x]} \rtimes \ZZ)$.
Theorem~\ref{thm:unstable gpd iso} implies that $C^*( G^{u}_{[x]} \rtimes
\ZZ) \cong C^*\big(\Gg^{\pi_{x}}_{G, E})$. Since $E$ is strongly connected,
$\pi_{x}$ is an open surjection, so \cite[Proposition~3.10]{FKPS} implies
that $C^*(\Gg_{G, E})$ is Morita equivalent to $C^*(\Gg^{\pi_{x}}_{G, E})$,
and Proposition~\ref{prp:O,C*G isomorphism} shows that $\Oo(G, E) \cong
C^*(\Gg_{G,E})$. Stringing these isomorphisms and Morita equivalences
together gives the desired Morita equivalence $C^*( G^u \rtimes \ZZ)
\sim_{\operatorname{Me}} \Oo(G, E)$.
\end{proof}

\begin{rmk}
If $E$ is not strongly connected, then $\pi_{x}$ is not necessarily
surjective, and $\mathcal{G}^{\pi_{x}}_{G,E}$ is only groupoid equivalent to
the reduction $\mathcal{G}_{G,E}|_{\pi_{x}(E^{\mathbb{Z}}_{x})}$ of
$\mathcal{G}_{G,E}$ to the image of $\pi_{x}$, which is open. However, if
$\mathcal{G}_{G,E}$ is minimal, then
$\mathcal{G}_{G,E}|_{\pi_{x}(E^{\mathbb{Z}}_{x})}$ is still groupoid
equivalent to $\mathcal{G}_{G,E}$.
\end{rmk}

We can now prove our main theorem.

\begin{proof}[Proof of Theorem~\ref{thm:main duality thm}]
Corollary~\ref{cor:Smale solenoid} shows that $(\Ss, \Stau)$ is an irreducible Smale space. So
Theorem~1.1 of \cite{Kaminker-Putnam-Whittaker:K-duality} shows that there are classes $\delta$ in
$KK^1\big(\CC, C^*( G^s \rtimes \ZZ) \otimes C^*( G^u \rtimes \ZZ)\big)$ and $\Delta$ in
$KK^1\big(C^*( G^s \rtimes \ZZ) \otimes C^*( G^u \rtimes \ZZ), \CC\big)$ such that $\delta
\widehat{\otimes}_{C^*( G^u \rtimes \ZZ)} \Delta = \id_{\KK(C^*( G^s \rtimes \ZZ), C^*( G^s
\rtimes \ZZ))}$ and $\delta \widehat{\otimes}_{C^*( G^s \rtimes \ZZ)} \Delta =
-\id_{\KK(C^*( G^u \rtimes \ZZ), C^*( G^u \rtimes \ZZ))}$.

Corollary~\ref{cor:dual alg Me} gives a Morita equivalence bimodule between
$\hO(G, E)$ and $C^*( G^s \rtimes \ZZ)$, which induces a $KK$-equivalence
$\hat{\alpha}  \in KK^0\big(\hO(G, E), C^*( G^s \rtimes \ZZ)\big)$. Likewise
Corollary~\ref{cor:base alg Me} gives a $KK$-equivalence $\alpha $ in
$KK\big(\Oo(G, E), C^*( G^u \rtimes \ZZ)\big)$. So the Kasparov products
\[
\beta := (\hat{\alpha} \otimes \alpha)
    \mathbin{\widehat{\otimes}} \delta
    \mathbin{\widehat{\otimes}} (\hat{\alpha} \otimes \alpha)^{-1}\quad\text{ and}\quad
\mu := (\hat{\alpha} \otimes \alpha)
    \mathbin{\widehat{\otimes}} \Delta
    \mathbin{\widehat{\otimes}} (\hat{\alpha} \otimes \alpha)^{-1}
\]
implement the desired duality.
\end{proof}

\begin{exm}\label{Ex:Katsura}
Unital Katsura algebras \cite{Katsura:class_II} come from self-similar groupoids that are not recurrent unless the graph has only a single vertex. A Katsura groupoid is defined by two matrices $A$ and $B$, where $A$ is the adjacency matrix of the graph and $B$ dictates the restriction map.

Let $(G,E)$ be the Katsura groupoid action defined by
\begin{equation}\label{Katsura_matrices}
A=\left(\begin{matrix} 2 & 1 \\ 2 & 2 \end{matrix}\right) \quad \text{ and } \quad B=\left(\begin{matrix} 1 & 0 \\ 1 & 1 \end{matrix}\right),
\end{equation}
see \cite[Definition 18.1]{Exel-Pardo:Self-similar} or \cite[Example 7.7]{Laca-Raeburn-Ramagge-Whittaker:Equilibrium}. The graph $E$ associated to $A$ is depicted in Figure \ref{Katsura_graph}.

\begin{center}
\begin{figure}
\begin{tikzpicture}[scale=1.0]
\node at (0,0) {$2$};
\node[vertex] (vertexe) at (0,0)   {$\quad$}
	edge [->,>=latex,out=-30,in=30,loop] node[left,pos=0.5]{$\scriptstyle e_{2,2,0}$} (vertexe)
	edge [->,>=latex,out=-50,in=50,loop,looseness=10] node[right,pos=0.5]{$\scriptstyle e_{2,2,1}$} (vertexe);
\node at (-3,0) {$1$};
\node[vertex] (vertex-a) at (-3,0)   {$\quad$}
	edge [->,>=latex,out=35,in=145] node[below,swap,pos=0.5]{$\scriptstyle e_{2,1,0}$} (vertexe)
	edge [->,>=latex,out=50,in=130] node[above,swap,pos=0.5]{$\scriptstyle e_{2,1,1}$} (vertexe)
	edge [->,>=latex,out=210,in=150,loop] node[right,pos=0.5]{$\scriptstyle e_{1,1,1}$} (vertex-a)
	edge [->,>=latex,out=230,in=130,loop,looseness=10] node[left,pos=0.5]{$\scriptstyle e_{1,1,0}$} (vertex-a)
	edge [<-,>=latex,out=310,in=230] node[below,swap,pos=0.5]{$\scriptstyle e_{1,2,0}$} (vertexe);
\end{tikzpicture}
\begin{tikzpicture}[scale=1.0]
\node at (0,0) {$w$};
\node[vertex] (vertexe) at (0,0)   {$\quad$}
	edge [->,>=latex,out=-30,in=30,loop] node[left,pos=0.5]{$5$} (vertexe)
	edge [->,>=latex,out=-50,in=50,loop,looseness=10] node[right,pos=0.5]{$6$} (vertexe);
\node at (-3,0) {$v$};
\node[vertex] (vertex-a) at (-3,0)   {$\quad$}
	edge [->,>=latex,out=35,in=145] node[below,swap,pos=0.5]{$3$} (vertexe)
	edge [->,>=latex,out=50,in=130] node[above,swap,pos=0.5]{$4$} (vertexe)
	edge [->,>=latex,out=210,in=150,loop] node[right,pos=0.5]{$1$} (vertex-a)
	edge [->,>=latex,out=230,in=130,loop,looseness=10] node[left,pos=0.5]{$0$} (vertex-a)
	edge [<-,>=latex,out=310,in=230] node[below,swap,pos=0.5]{$2$} (vertexe);
\end{tikzpicture}
\caption{The adjacency matrix $A$ from \eqref{Katsura_matrices} defines the graph $E$ on the left. On the right we relabel the graph to streamline the labelling of the Schreier graphs below. }
\label{Katsura_graph}
\end{figure}
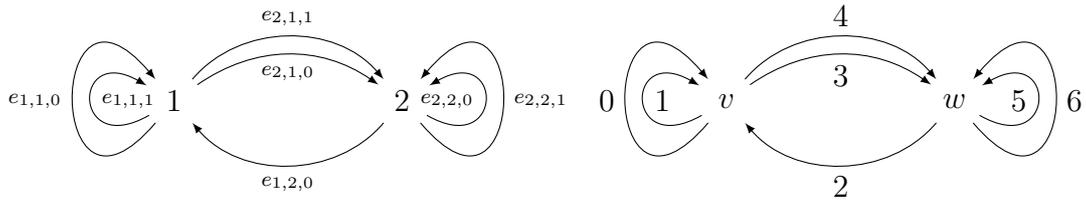
\end{center}

Using the hypotheses of \cite[Remark 18.3]{Exel-Pardo:Self-similar}, the
groupoid $G$ is $E^0 \times \ZZ$ with unit space $E^0$ and a copy of $\ZZ$ as
the isotropy group over each unit. We denote the generator of the copy of
$\ZZ$ over the vertex $i$ of the graph $E$ on the left-hand side of
Figure~\ref{Katsura_graph} by $a_i$. The action and restriction maps are
defined, for $\mu \in jE^*$, by the formulae
\begin{equation}
\label{Katsura_action}
a_i \cdot e_{i,j,m}\mu=e_{i,j,n}(a_j^l \cdot \mu) \quad \text{where} \quad
b_{ij}+m=l{a_{ij}}+n \text{ and }0\leq n< a_i.
\end{equation}
Thus, for the matrices in \eqref{Katsura_matrices}, in the notation of the
relabelled graph on the right of Figure~\ref{Katsura_graph}, and re-labelling
$a\coloneqq a_1$, $b \coloneqq a_2$, we obtain:
\begin{align*}
a\cdot 0=1,\ a|_0=v; \quad\quad&b\cdot 3=4, \ b|_3=v;\\
a\cdot 1=0,\ a|_1=a; \quad\quad&b\cdot 4=3,\ b|_4=a; \\
a\cdot 2=2,\ a|_2=w;\quad\quad&b\cdot 5=6,\ b|_5=w; \\
&b\cdot 6=5,\ b|_6=b;\notag
\end{align*}

Using [7, Remark 18.3], the $K$-theory of the Kirchberg algebra $\Oo(G,E)$ given by this Katsura groupoid is
\begin{align*}
K_0(\Oo(G,E))&= \text{coker}(I-A) \oplus \text{ker}(I-B)= \mathbb{Z} \quad \text{ and } \\ K_1(\Oo(G,E))&= \text{coker}(I-B) \oplus \text{ker}(I-A)= \mathbb{Z}.
\end{align*}

The first three Schreier graphs, $\Gamma_0$, $\Gamma_1$, and $\Gamma_2$, are given in Figure \ref{Katsura_Schreier}, with the levels separated by dashed lines. For a contracting, self-similar group(oid), the combination of dashed and solid lines is a self-similarity graph in the sense of Nekrashevych \cite[Section 3.7.1]{Nekrashevych:Self-similar}. The Gromov boundary of this self-similarity graph is the limit space \cite[Theorem 3.7.8]{Nekrashevych:Self-similar}. Note that we have omitted the unit loops at the vertices.

\begin{center}
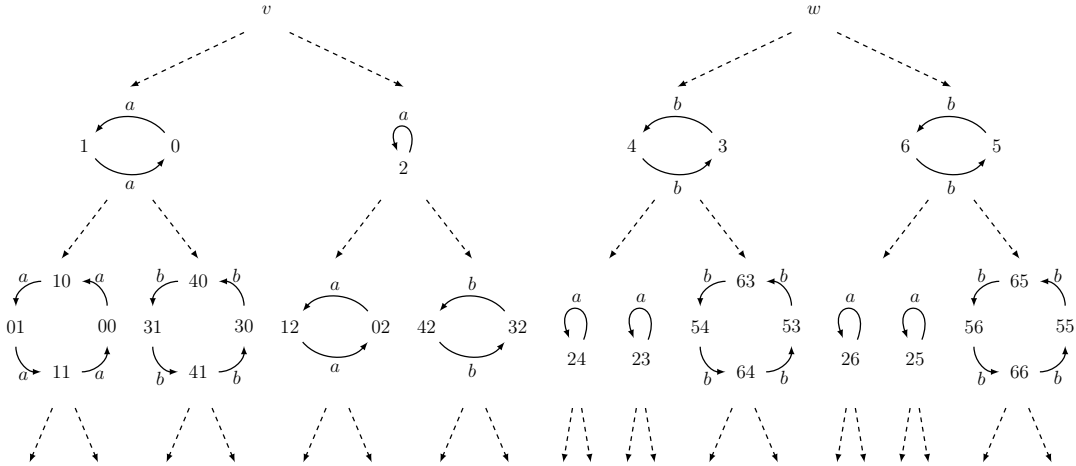
\begin{figure}
\scalebox{0.6}{
\begin{tikzpicture}
%\node at (-9,1) {$\Gamma_0$:};
%\node at (-9,0) {$\Gamma_1$:};
%\node at (-9,-3) {$\Gamma_2$:};
\node at (-6,3) {$v$};
\node at (6,3) {$w$};
\draw[->,dashed,>=latex,thick] (-6.5,2.5) -- (-9,1.3);
\draw[->,dashed,>=latex,thick] (-5.5,2.5) -- (-3,1.3);
\draw[->,dashed,>=latex,thick] (6.5,2.5) -- (9,1.3);
\draw[->,dashed,>=latex,thick] (5.5,2.5) -- (3,1.3);
\draw[->,dashed,>=latex,thick] (-9.5,-1.2) -- (-10.5,-2.5);
\draw[->,dashed,>=latex,thick] (-8.5,-1.2) -- (-7.5,-2.5);
\draw[->,dashed,>=latex,thick] (-3.5,-1.2) -- (-4.5,-2.5);
\draw[->,dashed,>=latex,thick] (-2.5,-1.2) -- (-1.5,-2.5);
\draw[->,dashed,>=latex,thick] (9.5,-1.2) -- (10.5,-2.5);
\draw[->,dashed,>=latex,thick] (8.5,-1.2) -- (7.5,-2.5);
\draw[->,dashed,>=latex,thick] (3.5,-1.2) -- (4.5,-2.5);
\draw[->,dashed,>=latex,thick] (2.5,-1.2) -- (1.5,-2.5);
\draw[->,dashed,>=latex,thick] (0.7,-5.8) -- (0.5,-7);
x\draw[->,dashed,>=latex,thick] (0.9,-5.8) -- (1.1,-7);
\draw[->,dashed,>=latex,thick] (4.2,-5.8) -- (3.7,-7);
\draw[->,dashed,>=latex,thick] (6.7,-5.8) -- (6.5,-7);
\draw[->,dashed,>=latex,thick] (6.9,-5.8) -- (7.1,-7);
\draw[->,dashed,>=latex,thick] (10.2,-5.8) -- (9.7,-7);
\draw[->,dashed,>=latex,thick] (2.1,-5.8) -- (1.9,-7);
\draw[->,dashed,>=latex,thick] (2.3,-5.8) -- (2.5,-7);
\draw[->,dashed,>=latex,thick] (4.7,-5.8) -- (5.2,-7);
\draw[->,dashed,>=latex,thick] (8.1,-5.8) -- (7.9,-7);
\draw[->,dashed,>=latex,thick] (8.3,-5.8) -- (8.5,-7);
\draw[->,dashed,>=latex,thick] (10.7,-5.8) -- (11.2,-7);
\draw[->,dashed,>=latex,thick] (-1.2,-5.8) -- (-0.7,-7);
\draw[->,dashed,>=latex,thick] (-4.2,-5.8) -- (-3.7,-7);
\draw[->,dashed,>=latex,thick] (-7.2,-5.8) -- (-6.7,-7);
\draw[->,dashed,>=latex,thick] (-10.2,-5.8) -- (-9.7,-7);
\draw[->,dashed,>=latex,thick] (-1.7,-5.8) -- (-2.2,-7);
\draw[->,dashed,>=latex,thick] (-4.7,-5.8) -- (-5.2,-7);
\draw[->,dashed,>=latex,thick] (-7.7,-5.8) -- (-8.2,-7);
\draw[->,dashed,>=latex,thick] (-10.7,-5.8) -- (-11.2,-7);

\begin{scope}[xshift=-9cm,yshift=0cm]
\def\xv{1.0}
\node[vertex] (vert_0) at (0:\xv)   {$0$};
\node[vertex] (vert_1) at (180:\xv)   {$1$}
	edge [->,>=latex,out=-50,in=-130,thick] node[below,pos=0.5]{$a$} (vert_0)
	edge [<-,>=latex,out=50,in=130,thick] node[above,pos=0.5]{$a$} (vert_0);
\end{scope}
\begin{scope}[xshift=-3cm,yshift=0cm]
\def\xv{0.5}
\node[vertex] (vert_1) at (270:\xv)   {$2$}
	edge [->,>=latex,out=70,in=110,thick,loop] node[above,pos=0.5]{$a$} (vert_1);
\end{scope}
\begin{scope}[xshift=3cm,yshift=0cm]
\def\xv{1.0}
\node[vertex] (vert_0) at (0:\xv)   {$3$};
\node[vertex] (vert_1) at (180:\xv)   {$4$}
	edge [->,>=latex,out=-50,in=-130,thick] node[below,pos=0.5]{$b$} (vert_0)
	edge [<-,>=latex,out=50,in=130,thick] node[above,pos=0.5]{$b$} (vert_0);
\end{scope}
\begin{scope}[xshift=9cm,yshift=0cm]
\def\xv{1.0}
\node[vertex] (vert_0) at (0:\xv)   {$5$};
\node[vertex] (vert_1) at (180:\xv)   {$6$}
	edge [->,>=latex,out=-50,in=-130,thick] node[below,pos=0.5]{$b$} (vert_0)
	edge [<-,>=latex,out=50,in=130,thick] node[above,pos=0.5]{$b$} (vert_0);
\end{scope}

\begin{scope}[xshift=-10.5cm,yshift=-4cm]
\def\x{1.0}
\node[vertex] (vert_00) at (1:\x)   {$00$};
\node[vertex] (vert_10) at (90:\x)   {$10$}
	edge [<-,>=latex,out=0,in=90,thick] node[above,pos=0.5]{$a$} (vert_00);
\node[vertex] (vert_01) at (180:\x)   {$01$}
	edge [<-,>=latex,out=90,in=180,thick] node[above,pos=0.5]{$a$} (vert_10);
\node[vertex] (vert_11) at (270:\x)   {$11$}
	edge [<-,>=latex,out=180,in=270,thick] node[below,pos=0.5]{$a$} (vert_01)
	edge [->,>=latex,out=0,in=270,thick] node[below,pos=0.5]{$a$} (vert_00);
\end{scope}
\begin{scope}[xshift=-7.5cm,yshift=-4cm]
\def\x{1.0}
\node[vertex] (vert_00) at (1:\x)   {$30$};
\node[vertex] (vert_10) at (90:\x)   {$40$}
	edge [<-,>=latex,out=0,in=90,thick] node[above,pos=0.5]{$b$} (vert_00);
\node[vertex] (vert_01) at (180:\x)   {$31$}
	edge [<-,>=latex,out=90,in=180,thick] node[above,pos=0.5]{$b$} (vert_10);
\node[vertex] (vert_11) at (270:\x)   {$41$}
	edge [<-,>=latex,out=180,in=270,thick] node[below,pos=0.5]{$b$} (vert_01)
	edge [->,>=latex,out=0,in=270,thick] node[below,pos=0.5]{$b$} (vert_00);
\end{scope}
\begin{scope}[xshift=-4.5cm,yshift=-4cm]
\def\xv{1.0}
\node[vertex] (vert_0) at (0:\xv)   {$02$};
\node[vertex] (vert_1) at (180:\xv)   {$12$}
	edge [->,>=latex,out=-50,in=-130,thick] node[below,pos=0.5]{$a$} (vert_0)
	edge [<-,>=latex,out=50,in=130,thick] node[above,pos=0.5]{$a$} (vert_0);
\end{scope}
\begin{scope}[xshift=-1.5cm,yshift=-4cm]
\def\xv{1.0}
\node[vertex] (vert_0) at (0:\xv)   {$32$};
\node[vertex] (vert_1) at (180:\xv)   {$42$}
	edge [->,>=latex,out=-50,in=-130,thick] node[below,pos=0.5]{$b$} (vert_0)
	edge [<-,>=latex,out=50,in=130,thick] node[above,pos=0.5]{$b$} (vert_0);
\end{scope}
\begin{scope}[xshift=1.5cm,yshift=-4cm]
\def\xv{1.0}
\node[vertex] (vert_0) at (315:\xv)   {$23$}
	edge [->,>=latex,out=70,in=110,thick,loop] node[above,pos=0.5]{$a$} (vert_0);
\node[vertex] (vert_1) at (225:\xv)   {$24$}
	edge [->,>=latex,out=70,in=110,thick,loop] node[above,pos=0.5]{$a$} (vert_1);
\end{scope}
\begin{scope}[xshift=4.5cm,yshift=-4cm]
\def\x{1.0}
\node[vertex] (vert_00) at (1:\x)   {$53$};
\node[vertex] (vert_10) at (90:\x)   {$63$}
	edge [<-,>=latex,out=0,in=90,thick] node[above,pos=0.5]{$b$} (vert_00);
\node[vertex] (vert_01) at (180:\x)   {$54$}
	edge [<-,>=latex,out=90,in=180,thick] node[above,pos=0.5]{$b$} (vert_10);
\node[vertex] (vert_11) at (270:\x)   {$64$}
	edge [<-,>=latex,out=180,in=270,thick] node[below,pos=0.5]{$b$} (vert_01)
	edge [->,>=latex,out=0,in=270,thick] node[below,pos=0.5]{$b$} (vert_00);
\end{scope}
\begin{scope}[xshift=7.5cm,yshift=-4cm]
\def\xv{1.0}
\node[vertex] (vert_0) at (315:\xv)   {$25$}
	edge [->,>=latex,out=70,in=110,thick,loop] node[above,pos=0.5]{$a$} (vert_0);
\node[vertex] (vert_1) at (225:\xv)   {$26$}
	edge [->,>=latex,out=70,in=110,thick,loop] node[above,pos=0.5]{$a$} (vert_1);
\end{scope}
\begin{scope}[xshift=10.5cm,yshift=-4cm]
\def\x{1.0}
\node[vertex] (vert_00) at (1:\x)   {$55$};
\node[vertex] (vert_10) at (90:\x)   {$65$}
	edge [<-,>=latex,out=0,in=90,thick] node[above,pos=0.5]{$b$} (vert_00);
\node[vertex] (vert_01) at (180:\x)   {$56$}
	edge [<-,>=latex,out=90,in=180,thick] node[above,pos=0.5]{$b$} (vert_10);
\node[vertex] (vert_11) at (270:\x)   {$66$}
	edge [<-,>=latex,out=180,in=270,thick] node[below,pos=0.5]{$b$} (vert_01)
	edge [->,>=latex,out=0,in=270,thick] node[below,pos=0.5]{$b$} (vert_00);
\end{scope}
\end{tikzpicture}}
\caption{The first three Schreier graphs, $\Gamma_0$, $\Gamma_1$, and $\Gamma_2$, associated to the Katsura groupoid $(G,E)$.}
\label{Katsura_Schreier}
\end{figure}
\end{center}

The limit space $\Jj$ is a Cantor set of circles along with a covering map $\sigma:\Jj \to\Jj$ that aligns with forward paths in the graph $E$. Poincar\'{e} duality then gives us the $K$-groups of the solenoids induced by the shift map on this Cantor set fibre bundle of circles. \qed
\end{exm}

\end{document}